\documentclass{amsart}
\usepackage{amssymb}
\usepackage{amsfonts}
\usepackage{amssymb}
\usepackage{amsmath}
\usepackage{amsthm}
\usepackage{enumerate}
\usepackage{tabularx}
\usepackage{centernot}
\usepackage{mathtools}
\usepackage{stmaryrd}
\usepackage{amsthm,amssymb}
\usepackage{etoolbox}

\usepackage{afterpage}

\usepackage{url}
\usepackage{tikz}
\usepackage{amssymb}
\usetikzlibrary{matrix}
\usepackage{tikz-cd}
\usepackage{tikz}
\usepackage{marginnote}
\definecolor{mygray}{gray}{0.85}
\usepackage[backgroundcolor=mygray,colorinlistoftodos,prependcaption,textsize=small]{todonotes}
\usepackage{xargs}                      
\usepackage{float}
\usepackage{graphics}
 \usepackage{relsize}

 \tikzcdset{scale cd/.style={every label/.append style={scale=#1},
 		cells={nodes={scale=#1}}}}
 
 \usepackage[colorlinks,citecolor=blue,urlcolor=blue, linkcolor=blue]{hyperref}

\usepackage[nocompress]{cite}
\usepackage{lscape} 

\newcommand{\mrm}[1]{\mathrm{#1}}

\renewcommand{\leq}{\leqslant}
\renewcommand{\geq}{\geqslant}

\newcommand{\dleq}{\leqslant_{\oplus}}
\newcommand{\sleq}{\leqslant_{\ast}}
\newcommand{\pleq}{\leqslant_{\ast}^{\mrm{c}}}

\newcommand{\precleq}{\preccurlyeq^{\mrm{c}}}

\newcommand{\bigast}{\mathop{\LARGE  \mathlarger{\mathlarger{*}}}}

\newcommand{\hleq}{\leq_{\mrm{HF}}}

\makeatletter
\def\subsection{\@startsection{subsection}{3}%
  \z@{.5\linespacing\@plus.7\linespacing}{.3\linespacing}%
  {\bfseries\centering}}
\makeatother

\makeatletter
\def\subsubsection{\@startsection{subsubsection}{3}%
  \z@{.5\linespacing\@plus.7\linespacing}{.3\linespacing}%
  {\centering}}
\makeatother

\makeatletter
\def\myfnt{\ifx\protect\@typeset@protect\expandafter\footnote\else\expandafter\@gobble\fi}
\makeatother

\renewcommand{\restriction}{ {\upharpoonright} }

\def\Ind{\setbox0=\hbox{$x$}\kern\wd0\hbox to 0pt{\hss$\mid$\hss}
	\lower.9\ht0\hbox to 0pt{\hss$\smile$\hss}\kern\wd0}
\def\Notind{\setbox0=\hbox{$x$}\kern\wd0\hbox to 0pt{\mathchardef
		\nn=12854\hss$\nn$\kern1.4\wd0\hss}\hbox to
	0pt{\hss$\mid$\hss}\lower.9\ht0 \hbox to 0pt{\hss$\smile$\hss}\kern\wd0}

\def\ind{\mathop{\mathpalette\Ind{}}}

\newtheorem{theorem}{Theorem}[section]

\theoremstyle{definition}

\newtheorem{corollary}[theorem]{Corollary}
\newtheorem{definition}[theorem]{Definition}
\newtheorem{lemma}[theorem]{Lemma}
\newtheorem{proposition}[theorem]{Proposition}

\newtheorem{fact}[theorem]{Fact}
\newtheorem{context}[theorem]{Context}

\newtheorem{remark}[theorem]{Remark}
\newtheorem*{theorem*}{Main Theorem 1}
\newtheorem*{fact*}{Fact}
\newtheorem*{theorem2*}{Main Theorem 2}
\newtheorem*{theorem3*}{Main Theorem 3}

\newtheorem{notation}[theorem]{Notation}

\usepackage[backgroundcolor=mygray,colorinlistoftodos,prependcaption,textsize=small]{todonotes}
\usepackage{xcolor}

\newcommand{\pureindep}[1][]{%
	\mathrel{
		\mathop{
			\vcenter{
				\hbox{\oalign{\noalign{\kern-.3ex}\hfil$\vert$\hfil\cr
						\noalign{\kern-.7ex}
						$\smile$\cr\noalign{\kern-.3ex}}}
			}
		}\displaylimits_{#1}
	}
}

\newcommand{\indep}[2]{%
	\mathrel{
		\mathop{
			\vcenter{
				\hbox{%
					\oalign{
						\noalign{\kern-.3ex}\hfil$\vert$\hfil\cr
						\noalign{\kern-.7ex}
						$\smile$\cr\noalign{\kern-.3ex}
					}
				}
			}
		}^{\!\!\!\!\!#2}_{\!\!\hspace{-0.1em}#1}
	}
}

\newcommand{\displayindep}[2]{%
	\mathrel{
		\mathop{
			\vcenter{
				\hbox{%
					\oalign{
						\noalign{\kern-.3ex}\hfil$\vert$\hfil\cr
						\noalign{\kern-.7ex}
						$\smile$\cr\noalign{\kern-.3ex}
					}
				}
			}
		}^{\!\!\hspace{-0.1em}#2}_{\!\!\hspace{-0.1em}#1}
	}
}

\newcommand{\displayfindep}[2]{%
	\mathrel{
		\mathop{
			\vcenter{
				\hbox{%
					\oalign{
						\noalign{\kern-.3ex}\hfil$\vert$\hfil\cr
						\noalign{\kern-.7ex}
						$\smile$\cr\noalign{\kern-.3ex} 
					}
				}
			}
		}^{\!\hspace{-0.14em}#2}_{\!\!\hspace{-0.05em}#1}
	}
}

\begin{document}

\begin{abstract} 
We use the framework of Abstract Elementary Classes ($\mrm{AEC}$s) to introduce a new Construction Principle $\mrm{CP}(\mathbf{K},\ast)$, which generalises the Construction Principle of Eklof, Mekler and Shelah from \cite{EM, MSh} and allows for many novel applications beyond the setting of universal algebra. From this we derive, in ZFC, that several uncountably categorical classes of structures are not axiomatisable in the logic $\mathfrak{L}_{\infty,\omega_1}$, and, under $V=L$, that they are not axiomatisable in $\mathfrak{L}_{\infty,\infty}$. In particular, our methods apply to: free products of cyclic groups of fixed order, direct sums of a fixed torsion-free abelian group of rank~$1$ which is not $\mathbb{Q}$, free $(k,n)$-Steiner systems, and free generalised $n$-gons. 
\end{abstract}

\title{A New Construction Principle}

\thanks{The authors were supported by project PRIN 2022 ``Models, sets and classifications", prot. 2022TECZJA. The second listed author was also supported by INdAM Project 2024 (Consolidator grant) ``Groups, Crystals and Classifications''. The third listed author was also supported by an INdAM post-doc grant.}

\author{Tapani Hyttinen}
\author{Gianluca Paolini}
\author{Davide E. Quadrellaro}

\address{Department of Mathematics and Statistics, University of Helsinki, P.O. Box 68 (Pietari Kalmin katu 5), 00014 Helsinki, Finland}
\email{tapani.hyttinen@helsinki.fi}

\address{Department of Mathematics “Giuseppe Peano”,
	University of Torino,
	Via Carlo Alberto 10,
	10123 Torino, Italy.}
\email{gianluca.paolini@unito.it}
\email{davideemilio.quadrellaro@unito.it}

\address{Istituto Nazionale di Alta Matematica ``Francesco Severi'',
	Piazzale Aldo Moro 5
	00185 Roma, Italy.}
\email{quadrellaro@altamatematica.it}


\date{\today}
\maketitle

\section{Introduction}\label{sec:introduction}

Every textbook in mathematical logic is rich in theorems relating the syntactic and the semantic side of the discipline. These results are among the most classical accomplishments from the field, often dating back to the first half of the 20\emph{th} century. A classical example is Birkhoff's Theorem in universal algebra stating that varieties of algebras are exactly those classes of structures axiomatised by equations. Similarly, an important corollary of Keisler-Shelah's Theorem is that a class of structures is elementary if and only if it is closed under ultraproducts and its complement is closed under ultrapowers. More recently, a similar question has been considered in the context of Abstract Elementary Classes ($\mathrm{AEC}$s). These objects were originally introduced by Shelah in  \cite{Sh1} in order to provide a common model-theoretic framework to several logics and classes of structures. Essentially, an $\mathrm{AEC}$ $(\mathbf{K},\preccurlyeq)$ is simply a class of structures $\mathbf{K}$ in a language $L$ together with a ``strong submodel'' relation $\preccurlyeq$ that obeys certain structural conditions  (cf. Def.~\ref{def_AEC}).   The definition of $\mrm{AEC}$ is thus essentially semantical: we do not require that $\mathbf{K}$ is axiomatised by some theory in a formal language, nor that $\preccurlyeq$ is defined in terms of embeddings preserving formulas of a certain logic. However, already in his foundational text on the subject, Shelah proved that every $\mathrm{AEC}$ $(\mathbf{K},\preccurlyeq)$ consists exactly of the models of an infinitary sentence in an expansion of the language of $\mathbf{K}$ \cite[p.~424]{Sh1}. In a recent paper Shelah and Villaveces \cite{VSh} provided a major improvement of this result and showed that every $\mrm{AEC}$ can be axiomatised by a sentence in an infinitary logic \emph{in the same vocabulary of $\mathbf{K}$}. Crucially, the converse of this fact is false: not all classes of structures $\mathbf{K}$ which can be axiomatised in some infinitary, or higher-order logic, can be endowed with a strong submodel relation to form an $\mrm{AEC}$  $(\mathbf{K},\preccurlyeq)$.

In this spirit, in the present paper we want to give new insights in the following classical logical problem: \emph{suppose $\mathbf{K}$ is a class of structures categorical in all $\kappa>\aleph_0$, when is $\mathbf{K}$ axiomatisable in the infinitary logic $\mathfrak{L}_{\infty,\infty}$?} We attack this problem from the perspective of $ \mrm{AEC}$s, isolating some abstract conditions that allow us to conclude that: \emph{failing to form an $ \mrm{AEC}$ is failing to be axiomatisable}. Now, if $\mathbf{K}$ can be endowed with a strong submodel relation $\preccurlyeq$ in such a way that $(\mathbf{K},\preccurlyeq)$ forms an $\mrm{AEC}$, then Shelah-Villaveces' Theorem entails that $\mathbf{K}$ is exactly the class of models of an infinitary sentence in the language of $\mathbf{K}$. We are then left with the case where $\mathbf{K}$ does not admit a strong submodel relation $\preccurlyeq$ such that $(\mathbf{K}, \preccurlyeq)$ forms an $ \mrm{AEC}$, in this case its axiomatisability remains open. In this article we consider the case where we can associate to $\mathbf{K}$ a specific submodel relation $\sleq$ which satisfies all axioms of $\mrm{AEC}$s \emph{with the possible exception of Coherence and Smoothness}. If $(\mathbf{K},\sleq)$ has L\"owenheim-Skolem number $\aleph_0$, is uncountably categorical, has universal models and canonical amalgamation, then we prove, in $\mrm{ZFC}$, that failure of Smoothness entails that $\mathbf{K}$ cannot be axiomatised in $\mathfrak{L}_{\infty,\omega_1}$, and, under the additional set-theoretical assumption $V=L$, that $\mathbf{K}$ cannot be axiomatised in $\mathfrak{L}_{\infty,\infty}$. As we will see, this generalises the classical results of non-axiomatisability of free $R$-modules (for rings $R$ which are not left-perfect), and the related study of almost-free $R$-modules from \cite{EM2}. In case of $R$-modules, the relation $\sleq$ used to achieve non-axiomatisability is the strong submodel relation of ``being a direct summand'' (cf.~Def.~\ref{definition V-free product}). Crucially, the choice of $\sleq$ as the strong submodel relation is \emph{not} standard in the context of $\mrm{AEC}$s of modules, where this is often assumed to be the submodule relation, or the pure submodule relation (cf.~\cite{Boney}, and see \cite{Baldwin} for an $\mrm{AEC}$ of modules with yet another strong submodel relation).

It is noteworthy that the possible failure of Smoothness has already been considered in the literature on Abstract Elementary Classes. In \cite{HP1,HP2} the first two named authors of the present paper introduced \emph{almost $\mrm{AEC}$s} as classes $(\mathbf{K},\preccurlyeq)$ satisfying all the axioms of $\mrm{AEC}$s but Smoothness, and used them to study concrete examples, such as geometric lattices and right-angled Coxeter groups. Independently, failure of Smoothness was considered also in the general study of $\mrm{AEC}$s. On the one hand, lack of Smoothness was identified as a source of non-structure theorems in the abstract study of independence notions from \cite[\S 9]{jarden}. On the other hand, $\mrm{AEC}$s without the Smoothness Axiom were studied by Shelah in \cite{Sh6} (under the name of \emph{weak $\mrm{AEC}$s}) in the context of the classification theory for universal classes, and they play an important role in Vasey's proof of the eventual categoricity conjecture for universal classes (cf.~\cite{vasey1,vasey2}).  Despite these previous contributions, the most direct inspiration of our work does not come from the field of $\mrm{AEC}$s, but rather from the work of Eklof, Mekler and Shelah on the axiomatisability of free algebras (generalizing the case of almost-free $R$-modules mentioned above). As a matter of fact, this study represents the first comprehensive analysis of the question of axiomatisability of uncountably categorical classes, at this level of generality.  In \cite{EM} and \cite{MSh}, Eklof, Mekler and Shelah  provided a full characterisation of the possible \emph{spectra} of $\mathfrak{L}_{\infty,\kappa}$-free algebras, i.e., of algebras which are equivalent to the free algebra $F_{\mathbf{V}}(\kappa)$ in the infinitary logic $\mathfrak{L}_{\infty,\kappa}$ (cf.~\cite[p.~81]{EM}). They proved in particular the following theorem \cite[Thm.~16]{MSh} which sums up the main results from \cite{EM,MSh} (and which largely extends the previous results on locally-finite varieties established by Baldwin and Lachlan in \cite{Baldwin_Lachlan}). 

\begin{fact}[Eklof-Mekler-Shelah]\label{recap_fact}
	For any variety of algebras $\mathbf{V}$ in a countable language exactly one of the following possibilities holds:
	\begin{enumerate}[(i)]
		\item there is an $\mathfrak{L}_{\infty,\omega_1}$-free algebra of cardinality $\aleph_1$ which is not free;
		\item every $\mathfrak{L}_{\infty,\omega_1}$-free algebra of cardinality $\aleph_1$ is free, and for every infinite cardinal $\kappa$ there is an $\mathfrak{L}_{\infty,\kappa}$-free algebra of cardinality $\kappa^+$ which is not free;
		\item every $\mathfrak{L}_{\infty,\omega}$-free algebra is free.
	\end{enumerate}
	Moreover, under the set-theoretical assumption that $V=L$, item (i) is equivalent to the statement that the theory of $F_{\mathbf{V}}(\kappa)$ in $\mathfrak{L}_{\infty,\kappa}$ is not $\kappa$-categorical  for all regular, non-weakly compact cardinals $\kappa$.
\end{fact}

Interestingly, the problem of the axiomatisability of free algebras in infinitary logics arose independently in the context of group theory. In fact, the statement that an uncountable group $G$ is equivalent in $\mathfrak{L}_{\infty,\omega_1}$ to the absolutely free group $F(\aleph_1)$ (viz. free in the variety of \emph{all} groups) is equivalent to the statement that every countable subgroup of $G$ is free. The study of subgroups of free groups has a long history in algebra, which goes back to the Nielsen-Schreier Theorem that subgroups of free groups are free. Subsequently, Higman showed in \cite{higman} that there are $\aleph_1$-free groups of cardinality $\aleph_1$, i.e., non-free groups of size $\aleph_1$ whose countable subgroups are all free. Both Higman's work and the later results by Eklof, Mekler and Shelah rely on a careful combinatorial analysis of groups, and more generally of algebras in a specific variety. One of the most striking and interesting aspects of the work of Eklof, Mekler and Shelah on almost-free algebras is perhaps the identification of the so-called \emph{Construction Principle} ($ \mrm{CP} $)  (cf.~Definition~\ref{CP:mekler-shelah}), which is essentially  a combinatorial pattern of countable free algebras with deep consequences for the whole class of free algebras (with respect to a given variety $\mathbf{V}$). This pattern appears in Higman's construction, and also in all other constructions of a non-free $\mathfrak{L}_{\infty,\omega_1}$-free algebra in a variety $\mathbf{V}$. In \cite{EM} Eklof and Mekler proved that whether ($ \mrm{CP} $) holds in a variety of algebras is equivalent to item (i) from \ref{recap_fact}.

\begin{fact}[Eklof-Mekler]
	Let $\mathbf{V}$ be a variety in a countable language, ($ \mrm{CP} $) holds in $\mathbf{V}$ if and only if there is an $\mathfrak{L}_{\infty,\omega_1}$-free algebra of size $\aleph_1$ in $\mathbf{V}$ which is not free.
\end{fact}

The Construction Principle ($ \mrm{CP} $) thus provides us with a combinatorial test to determine whether a variety has $\mathfrak{L}_{\infty,\omega_1}$-free algebras which are not free. This paves the way for a vast number of applications, most notably to $R$-modules \cite{EM2}, where the presence of ($ \mrm{CP} $) is equivalent to the fact that $R$ is {\em not} left-perfect (cf.~\cite[p.~193]{EM2}). The following corollary of the previous theorem shows that ($ \mrm{CP} $) has a crucial connection with the problem of axiomatisability of free algebras.

\begin{fact}[Eklof-Mekler]
	Let $\mathbf{V}$ be a variety in a countable language that satisfies (CP), then its subclass $\mathbf{FV}$ of $\mathbf{V}$-free algebras is not axiomatisable in $\mathfrak{L}_{\infty,\omega_1}$. Moreover, assuming $V=L$, it also follows that  $\mathbf{FV}$ is not axiomatisable in $\mathfrak{L}_{\infty,\infty}$. 
\end{fact}

Abstractly, the Construction Principle can be seen essentially as a failure of the Smoothness Axiom in the class of free algebras with the relation of free factor (we elaborate on this later in Sections~\ref{pre:univ_algebra} and \ref{CP:coproduct_case}). The fundamental aim of this paper is to show that \emph{the undefinability of free algebras is an instance of a more general phenomenon}, which takes place essentially in any context where we have a combination of \emph{uncountable categoricity}, \emph{canonical amalgamation} and a version of the \emph{Construction Principle}. In order to reach this level of generality and extend the scope of the Eklof-Mekler-Shelah's Construction Principle, we introduce a new Construction Principle $\mrm{CP}(\mathbf{K},\ast)$ in the setting of $\mrm{AEC}$s, and we provide several examples of its applications beyond the scope of universal algebra. 

We proceed in this article as follows. In Section~\ref{sec:kappa_cat} and Section~\ref{CP:coproduct_case} we generalise the Construction Principle beyond the setting of universal algebra. We formalise our new Construction Principle in what we call \emph{weak $ \mrm{AEC}$s}, i.e., classes which satisfy all the  requirements of $ \mrm{AEC}$s except possibly the Coherence and the Smoothness axioms (this is thus a context strictly weaker than Shelah's {\em weak  $ \mrm{AEC}$s}, and also of the {\em almost $ \mrm{AEC}$s} of \cite{HP1,HP2}). This approach allows us to relate the two independent lines of research recalled above: the study of almost free algebras that originated with Higman, and the abstract research on $\mrm{AEC}$s. We proceed in two steps. First, in Section~\ref{sec:kappa_cat}, we consider weak $\mrm{AEC}$s $ (\mathbf{K},\preccurlyeq) $ that: (i) are $\kappa^+$-categorical, (ii) admit suitable universal models of size $\kappa$, and (iii) satisfy a first abstract version of the Construction Principle $\mrm{CP}^\kappa_{\lambda,\delta}(\mathbf{K},\preccurlyeq)$. Essentially, $\mrm{CP}^\kappa_{\lambda,\delta}(\mathbf{K},\preccurlyeq)$ is a failure of Smoothness for specific chains of size $\lambda$, i.e., for chains of \emph{limit models} of size~$\kappa$. Under these assumptions, and under the additional assumption that $V=L$, we then  prove that for every $\kappa\geq \aleph_0$, there are $\mathfrak{L}_{\infty,\kappa^+}$-free structures of size $\kappa^+$ that do not belong to $\mathbf{K}$. Secondly, in  Section~\ref{CP:coproduct_case} we consider uncountably (and eventually) categorical classes which satisfy a canonical form of amalgamation. We formulate in this setting our new Construction Principle $\mrm{CP}(\mathbf{K},\ast)$ and we show that in this case the failure of Smoothness for arbitrary chains of models entails the failure of Smoothness for chains of  \emph{arbitrarily large limit models} (cf.~Lemma~\ref{prop:build_large_CP}).  We thus obtain the following theorem, which is the first main result of the paper and which generalises the classical work on the Construction Principle by Eklof, Mekler and Shelah. We highlight here the case with countable L\"owenheim-Skolem number and direct the reader to Section~\ref{subsec:4.2} for a slightly more general statement.

\begin{theorem*} Let $(\mathbf{K},\sleq)$ be a canonical amalgamation class (cf.~\ref{context:amalgamation_class}) with $\mathrm{LS}(\mathbf{K}, \sleq)=\aleph_0$, with $(\mathbf{K}, \sleq)$-universal models for all  $\kappa\geq \aleph_0$ (cf.~\ref{def:universality}) and which is $\kappa^+$-categorical in all $\kappa\geq \aleph_0$. If $(\mathbf{K},\sleq)$ satisfies the Construction Principle $\mrm{CP}(\mathbf{K},\ast)$ (cf.~\ref{CP-AEC:coproducts}), then:
	\begin{enumerate}[(1)]
		\item there is an $\mathfrak{L}_{\infty,\omega_1}$-free structure of size $\aleph_1$  not in $\mathbf{K}$;
		\item if $V=L$ there is for all $\kappa\geq \aleph_0$  an $\mathfrak{L}_{\infty,\kappa^+}$-free structure of size $\kappa^+$ not in $\mathbf{K}$.
	\end{enumerate}
Thus, it follows, in $\mrm{ZFC}$, that $\mathbf{K}$ is not axiomatisable in $\mathfrak{L}_{\infty,\omega_1}$ and, under $V=L$, that $\mathbf{K}$ is not  axiomatisable also in $\mathfrak{L}_{\infty,\infty}$.
\end{theorem*}

One important consequence of this theorem is that it allows us to work in the abstract setting of weak $ \mrm{AEC}$s, rather than specifically with varieties of algebras. We then make use of the generality of our approach to obtain some novel non-axiomatisability results in concrete cases. In Section~\ref{Sec:application} we turn to this research direction and we study some concrete applications of $\mrm{CP}(\mathbf{K},\ast)$. We first show in Section~\ref{Application:varieties} that $\mrm{CP}(\mathbf{K},\ast)$  subsumes the classical Construction Principle by Eklof, Mekler and Shelah, and we then consider some new applications to free products of cyclic groups of fixed order, to directs sums of a fixed torsion-free abelian groups of rank~$1$ which is not $\mathbb{Q}$, and to incidence geometry (for an extensive overview of this topic see \cite{handbook}). In the process, we noticed while working on the applications to incidence geometry that there was no solid definition of free projective planes in the uncountable setting. We expand on this matter in Section~\ref{sec:remark_freeness}, where we prove  that the class of projective planes which admit a wellfounded $\mrm{HF}$-construction is uncountably categorical (cf. Theorem~\ref{theorem:free_categoricity}). Our main result from Section~\ref{Sec:application} is the following theorem (cf. \ref{def:steiner} for the definition of $(k,n)$-Steiner systems and \ref{def:n_gon} for the definition of generalised $n$-gons).

\begin{theorem2*}
	Let $\mathbf{K}$ be any of the following classes:
	\begin{enumerate}[(a)]
		\item free products of cyclic groups of fixed order $2 \leq n \in \mathbb{N} \cup \{\infty\}$;
		\item direct sums of a fixed torsion-free abelian group of rank~$1$ which is not $\mathbb{Q}$;		
		\item  free $(k,n)$-Steiner systems (for  $2\leq k<n<\omega$);
		\item  free generalised $n$-gons (for $3\leq n<\omega$);
	\end{enumerate}
	then there is an $\mathfrak{L}_{\infty,\omega_1}$-free structure $M\notin \mathbf{K}$ (cf.~\ref{def:L-free}(1)) and, under $V=L$, there is for every $\kappa\geq \aleph_0$ an $\mathfrak{L}_{\infty,\kappa^+}$-free structure $M\notin \mathbf{K}$. It follows in $\mrm{ZFC}$ that $\mathbf{K}$ is not axiomatisable in $\mathfrak{L}_{\infty,\omega_1}$ and, under $V=L$, it is not  axiomatisable also in $\mathfrak{L}_{\infty,\infty}$.
\end{theorem2*}

\noindent  These applications display the special role of the Construction Principle in several classes of structures of interest. Additionally, we also believe that the new Construction Principle $\mrm{CP}(\mathbf{K},\ast)$ may find several further applications, most notably in the case of modules. Modules are in fact the most natural application of the classical Construction Principle, as witnessed by the standard reference \cite{EM2} by Eklof and Mekler, and have recently been subject of several investigations in the context of $ \mrm{AEC}$s \cite{Boney,armida1,armida2,Trlifaj,Baldwin,SaTr}. We think that relating these two research directions is a natural strand of further research. We leave this as a pointer for future works.

\section{Preliminaries}

We collect in this section some preliminary definitions and informations that we will use in the rest of the article. We deal in  Section~\ref{pre:univ_algebra} with some preliminary facts from universal algebra,  in Section~\ref{pre:inf_logic} with infinitary logics, in Section~\ref{pre:aec} with $ \mrm{AEC}$s and, in Section~\ref{pre:inf_comb}, with some notions from set theory and infinitary combinatorics. We refer the reader to \cite{burris,Sh2, Dickmann,Vaananen,jech} for further background.

\subsection{Preliminaries on universal algebra}\label{pre:univ_algebra}

\begin{context}\label{context:univ_algebra}
	We recall that an algebra is a first-order structure in a functional signature, and a variety is a class of algebras axiomatised by equations. By Birkhoff's Theorem (cf.~\cite[Thm. 11.9]{burris}) varieties can be equivalently defined as classes of algebras which are closed under subalgebras, products and homomorphic images. Let $\mathbf{V}$ be a variety, then for any cardinal $\kappa$ there is a free algebra in $\mathbf{V}$ (unique up to isomorphism) with exactly $\kappa$-many generators, which we denote by $F_\mathbf{V}(\kappa)$ --- we drop the index $\mathbf{V}$ when it is clear from the context. Given a variety $\mathbf{V}$, we denote by $\mathbf{FV}$ its subclass of free algebras (i.e., $\mathbf{FV}$ are the free algebras in the variety $\mathbf{V}$).  Let $\mu$ be the size of the language of $\mathbf{V}$, then $\kappa\leq|F_\mathbf{V}(\kappa)|\leq \kappa+\mu+\aleph_0$. It follows that for all $\kappa>\mu+\aleph_0$ there is only one free algebra in $\mathbf{V}$ up to isomorphism, namely $F_\mathbf{V}(\kappa)$. This means that the class $\mathbf{FV}$ is eventually categorical and, for varieties $\mathbf{V}$ in a countable language, it is uncountably categorical.
\end{context}

\begin{definition}\label{definition V-free product}
	Let $\mathbf{V}$ be a variety, we then define the following notions.
	\begin{enumerate}[(1)]
		\item An algebra $A \in \mathbf{V}$ is said to be the $\mathbf{V}$-\emph{free product} of its substructures $B$ and $C$ if and only if it is generated by $B$ and $C$ and, for any structure $D\in \mathbf{V}$ and any two homomorphisms $f:B \rightarrow D$ and $g:C \rightarrow D$, there is a unique homomorphism $h:A \rightarrow D$ such that $h\restriction B=f$ and $h\restriction C=g$.
		\item If the free product of two objects $B,C\in \mathbf{V}$ exists then we denote it by $B\ast C$.
		\item If $A=B\ast C$ for $B,C\in \mathbf{V}$, then we also say that $B$ is a $\mathbf{V}$-\emph{free factor} of $A$, and we refer to $C$ as a \emph{complementary factor} of $B$ in $A$. 
		\item If $A=B\ast C$ and in addition $A, B,C\in \mathbf{FV}$, then we write $B \sleq A$, i.e., $B$ is a free factor of $A$ with a free complementary factor.
	\end{enumerate}
\end{definition}

\begin{remark}\label{fact:free_product_varieties}
Notice that, as observed in \cite[p.~83]{EM} the free product of two algebras  $A,B\in \mathbf{V}$ does not necessarily exist, but it always does when every finite subset of $A$ and $B$ is contained in some free subalgebra. \emph{A fortiori}, free products of free algebras are always well-defined. It follows that the class of free algebras is also closed under infinite free products, i.e., if $A_i\in\mathbf{FV}$ for all $i<\kappa$, then there is a free product $\bigast_{i<\kappa}A_i\in\mathbf{FV}$ which satisfies the infinite version of the properties from Definition~\ref{definition V-free product}. As a matter of fact, we prove in Lemma~\ref{prop:varieties_canon} that the class of $\mathbf{V}$-free algebras $\mathbf{FV}$ with the relation $\sleq$ is a weak $\mrm{AEC}$ (cf.~\ref{def:weak_aec}) satisfying Conditions~\ref{context:amalgamation_class}(\hyperref[C1]{C1})-(\hyperref[C3]{C6}), and thus makes for one major application of our setting.
\end{remark}

The notion of free factor lies at the heart of the so-called Construction Principle, i.e., a combinatorial pattern occurring in the construction of almost-free algebras. This was first isolated in \cite{EM} and then studied in several further works, most notably \cite{mekler,mekler2,EM2,MSh}. The following definition is essentially from \cite[p.~129]{mekler}, with the difference that we state it also for varieties in uncountable languages. For this reason we require that the algebras $A$ and $B$ are countably generated, but not necessarily countable, and also in Clause~(2) of \ref{definition CP} we ask that $A\not\sleq B\ast F(\mu) $ rather than $A\not\sleq B\ast F(\aleph_0) $, where $\mu$ is the size of the language of $\mathbf{V}$.

\begin{definition}\label{definition CP}\label{CP:mekler-shelah}
	Let $\mathbf{V}$ be a variety in a language of size $\mu$, we say that the \emph{Construction Principle (CP)} holds in $\mathbf{V}$ if there are two countably generated $\mathbf{V}$-free structures $A \leq B$ such that:
	\begin{enumerate}[(1)]
	\item there are free generators $\{a_i:i<\omega\}\subseteq B$ of $A$ such that $A_n :=\langle (a_i)_{i\leq n}\rangle_B\sleq B$ for every $n<\omega$;
	\item $A\not\sleq B\ast F(\mu) $.
\end{enumerate}	
\end{definition}

\noindent We also notice that clause (2) from the Construction Principle above is slightly different from the original one from \cite{EM}, but as remarked by Mekler in \cite[p.~130]{mekler2} it is equivalent to it. We shall often refer to the principle above as the Eklof-Mekler-Shelah Construction Pinciple.

\begin{remark}
	We notice that in their monograph on modules Eklof and Mekler already consider the Construction Principle in some varieties in uncountable languages \cite[p.~193]{EM2}, most specifically for modules over arbitrary rings. Our Definition~\ref{definition CP} above essentially generalises also their Construction Principle from \cite[p.~130]{mekler2}, with the only difference that in Clause~(2) they require $A\not\sleq B\ast F(\aleph_0) $ instead of $A\not\sleq B\ast F(\mu) $. The reason for this discrepancy is simply that for varieties of free $R$-modules we have that  $A\sleq B\ast F(\mu) $ already entails $A\sleq B\ast F(\aleph_0) $, as it readily follows from \cite[Cor. 7.18]{rotman}).
\end{remark}

\subsection{Preliminaries on infinitary logic}\label{pre:inf_logic}

In this article we are interested in the non-axiomatisability of several classes of structures in infinitary logics. We recall in this section only the fundamentals of infinitary logics and refer the reader to \cite{Dickmann} and \cite{Vaananen} for additional background.

\begin{definition}
	Let $L$ be a first-order signature, let $\kappa\geq \lambda\geq \aleph_0$ and let $V$ be a fixed set of $\lambda$-many variables, then the set of formulas of $\mathfrak{L}_{\kappa,\lambda}$ is obtained from $V$ and the atomic formulas from $L$ by recursively: (i) taking negations, (ii) taking conjunctions and disjunctions of sets of formulas $\Phi$ such that $|\Phi|<\kappa$ and $<\lambda$ variables occur in $\Phi$, and (iii) quantifying $<\lambda$ variables in a given formula. Then  $\mathfrak{L}_{\infty,\lambda}$ is a class-sized logic with formulas in $\mathfrak{L}_{\kappa,\lambda}$ for all cardinals $\kappa\geq \lambda$, and $\mathfrak{L}_{\infty,\infty}$ is the logic with formulas in $\mathfrak{L}_{\infty,\lambda}$ for all infinite cardinals $\lambda$. We write $\mathfrak{L}_{\kappa,\lambda}(L)$, $\mathfrak{L}_{\infty,\lambda}(L)$ and $\mathfrak{L}_{\infty,\infty}(L)$ when we want to make explicit the underlying signature $L$.
\end{definition}

\begin{notation}
	Given the logics $\mathfrak{L}_{\kappa,\lambda}$, $\mathfrak{L}_{\infty,\lambda}$ and $\mathfrak{L}_{\infty,\infty}$, we write $\preccurlyeq_{\kappa,\lambda}$, $\preccurlyeq_{\infty,\lambda}$ and $\preccurlyeq_{\infty,\infty}$ for the embeddings preserving formulas in the respective logics. In particular, we shall write $\preccurlyeq_{\omega,\omega}$ for the relation of elementary embedding from first-order logic. We write $\equiv_{\kappa,\lambda}$, $\equiv_{\infty,\lambda}$ and $\equiv_{\infty,\infty}$ for the corresponding notions of logical equivalence in the logics the logics $\mathfrak{L}_{\kappa,\lambda}$, $\mathfrak{L}_{\infty,\lambda}$ and $\mathfrak{L}_{\infty,\infty}$. Similarly, we write $\equiv_{\omega,\omega}$ for the relation of elementary equivalence from first-order logic.
\end{notation}

In this work we will often need a criterion to establish when two structures are equivalent in some infinitary logic. We recall the following back-and-forth criterion concerning  $\mathfrak{L}_{\infty,\lambda}$, which goes back to the seminal work of Karp \cite{Karp}.  We refer to \cite[Thm.~9.26]{Vaananen} for a proof of the following fact and to \cite[5.7,~9.2,~9.7]{Vaananen} for history and further  references.

\begin{fact}\label{criterion:Marker}
	Given two $L$-structures $A$ and $B$, we have that   $A\equiv_{\infty,\lambda}B$ if and only if there is a nonempty family $\mathcal{F}$ of partial isomorphisms $f:A\to B$ with the following \emph{back-and-forth property}:
	\begin{enumerate}[(1)]
		\item if $f\in \mathcal{F}$ and $C\subseteq A$ with $|C|<\lambda$, there is  $g\in \mathcal{F}$ with  $C\subseteq \mrm{dom}(g)$ and $g\supseteq f$; 
		\item if $f\in \mathcal{F}$ and $C\subseteq B$ with $|C|<\lambda$, there is  $g\in \mathcal{F}$ with  $C\subseteq \mrm{ran}(g)$ and $g\supseteq f$.
	\end{enumerate}
\end{fact}

\begin{remark}\label{remark:partial_iso}
	When two structures $A$ and $B$ admit a back-and-forth system of partial isomorphism $\mathcal{F}$ with the properties (1)-(2) from \ref{criterion:Marker} we also say (following \cite[Def.~9.22]{Vaananen}) that they are $\lambda$-partially isomorphic and we write $A\cong_\lambda B$. The previous fact thus establishes the equivalence between $A\equiv_{\infty,\lambda} B$ and $A\cong_\lambda B$.
\end{remark}

\subsection{Preliminaries on abstract elementary classes}\label{pre:aec}

In classical model theory it is customary to study the class of models of a complete first-order theory $T$ with the relation of elementary embedding $\preccurlyeq_{\omega,\omega}$, namely to consider the category $(\mrm{Mod}(T),\preccurlyeq_{\omega,\omega})$. Abstract elementary classes were introduced by Shelah in  \cite{Sh1} to generalise this setting, and to provide a common model-theoretic framework to classes of structure with a reasonable notion of ``strong embedding'', capturing also relations such as those of substructure, or of pure submodule. Notice that we write $A\leq B$ to say that $A$ is a substructure of $B$. First, we recall the following definition from \cite{Bon_et_al}.

\begin{definition}\label{def_AC} 
Let $\mathbf{K}$ be a class of $L$-structures and let $\preccurlyeq$ be a binary relation on $\mathbf{K}$. We say that $(\mathbf{K}, \preccurlyeq)$ is an {\em abstract class} if the following conditions are satisfied:
\begin{enumerate}[(1)]
	\item  $\mathbf{K}$ and $\preccurlyeq$ are closed under isomorphisms, i.e., if $A\in \mathbf{K}$ and $f:A\to B$ is an isomorphism then $B\in \mathbf{K}$ and if $C\in \mathbf{K}$ is such that $C\preccurlyeq A$ then $f(C)\preccurlyeq B$;
	\item if $A \preccurlyeq B$ then $A\leq B$ (i.e., $A$ is an $L$-substructure of $B$);
	\item the relation $\preccurlyeq$ is a partial order on $\mathbf{K}$.
\end{enumerate}
\end{definition}

\begin{definition}\label{def:strong embedding}
	Let $(\mathbf{K},\preccurlyeq)$ be an abstract class and $A,B\in \mathbf{K}$, we say that an embedding $f:A\to B$ is a \emph{strong embedding} if $f(A)\preccurlyeq B$.
\end{definition}

\begin{definition}\label{def_AEC}  Let $\mathbf{K}$ be a class of $L$-structures and let $\preccurlyeq$ be a binary relation on $\mathbf{K}$. We say that $(\mathbf{K}, \preccurlyeq)$ is an {\em abstract elementary class} ($\mathrm{AEC}$) if it is an abstract class (i.e., it respects \ref{def_AC}(1)-(3)) and the following conditions are satisfied.
	\begin{enumerate}[(4)]
		\item[(4)] If $(A_i)_{i < \delta}$ is an increasing continuous $\preccurlyeq$-chain of structures in $\mathbf{K}$, then:
		\begin{enumerate}[({4.}1)]
			\item $\bigcup_{i < \delta} A_i \in \mathbf{K}$;
			\item for each $j < \delta$, $A_j \preccurlyeq \bigcup_{i < \delta} A_i$;
			\item if $A_i \preccurlyeq B$ for all $i<\delta$, then $\bigcup_{i < \delta} A_i \preccurlyeq B$.
	\end{enumerate}
			\item[(5)] If $A, B, C \in \mathbf{K}$, $A \preccurlyeq C$, $B \preccurlyeq C$ and $A \leq B$, then $A \preccurlyeq B$; 
			\item[(6)] There is a L\"owenheim-Skolem number $\mathrm{LS}(\mathbf{K}, \preccurlyeq)\geq |L|+\aleph_0$ such that if $A \in \mathbf{K}$ and $B \subseteq A$, then there is $C \in \mathbf{K}$ with $B \subseteq C\preccurlyeq A$ and $|C| \leq |B| + \mathrm{LS}(\mathbf{K}, \preccurlyeq)$.
\end{enumerate}
When $(\mathbf{K},\preccurlyeq)$ is an $\mrm{AEC}$, we refer to $\preccurlyeq$ as a \emph{strong submodel relation} for the class $\mathbf{K}$. We refer to Condition~(4) as the \emph{Tarski-Vaught Axioms}, to Condition~(4.3) specifically as the \emph{Smoothness Axiom} of $\mrm{AEC}$, to Condition~(5) as the \emph{Coherence Axiom}, and to Condition~(6) as the \emph{L\"owenheim-Skolem-Tarski Axiom}.
 \end{definition}

\begin{notation}
	Let $\mathbf{K}$ be an abstract class of structures (cf.~\ref{def_AC}), we write $\mathbf{K}_{\lambda}$ for the collection of structures in  $\mathbf{K}$ of size $\lambda$, and we write $\mathbf{K}_{<\lambda}$ and $\mathbf{K}_{>\lambda}$ for the collections of structures in  $\mathbf{K}$ of size $<\lambda$ and $>\lambda$, respectively. Given two $L$-structures $A$ and $B$ we write $A\leq B$ if $A$ is a substructure of $B$.
\end{notation}

As we explained in the introduction (cf.~\S \ref{sec:introduction}), it follows from Shelah and Villaveces' recent result from \cite{VSh} that every $\mrm{AEC}$ $(\mathbf{K},\preccurlyeq)$ in the language $L$ is axiomatisable in $\mathfrak{L}_{\infty,\infty}(L)$. Since in this article we are interested in \emph{non-axiomatisability} results, it is then natural to consider the case when an abstract class of structures $(\mathbf{K},\preccurlyeq)$ does \emph{not} form an $\mrm{AEC}$. We thus define the notions of \emph{weak $\mrm{AEC}$} and \emph{almost $\mrm{AEC}$}. 

\begin{definition}\label{def:weak_aec}
	We say that an abstract class $(\mathbf{K}, \preccurlyeq)$ is a \emph{weak $\mathrm{AEC}$} if it has a L\"owenheim-Skolem number and is closed under continuous $\preccurlyeq$-chains, i.e., it  satisfies the axioms of $\mathrm{AEC}$s (from \ref{def_AEC}) except possibly Smoothness or Coherence (i.e., \ref{def_AEC}(4.3) and \ref{def_AEC}(5)). We say that $(\mathbf{K}, \preccurlyeq)$ is an \emph{almost $\mathrm{AEC}$} if all the axioms from \ref{def_AEC} are satisfied except possibly Smoothness (i.e., \ref{def_AEC}(4.3)). A weak $\mathrm{AEC}$ is \emph{strictly weak} if neither Smoothness nor Coherence is satisfied; an almost $\mathrm{AEC}$ is \emph{strictly almost} if it does not satisfy Smoothness.
\end{definition}

\noindent We briefly clarify the terminology. Almost $\mathrm{AEC}s$ were employed  by the first two named authors of the present paper already in \cite{HP1} to study geometric lattices, and later in \cite[\S 5]{HP2} in the context of right-angled Coxeter groups. We stress that almost $\mathrm{AEC}s$ coincide with Shelah's \emph{weak $\mrm{AEC}$} from \cite{Sh6}, which play an important role in Vasey's proof of the eventual categoricity conjecture for universal classes \cite{vasey1,vasey2}. Since we will often consider classes where Coherence may fail we prefer to distinguish between weak $\mrm{AEC}$s and almost $\mrm{AEC}$s. For the same reason, we introduce the following relation $\preccurlyeq^{\mrm{c}}$, which we call the {\em coherentisation of $\preccurlyeq$}. Essentially, this forces a weaker form of coherence relative to structure of some bounded size. In particular, if $(\mathbf{K},\preccurlyeq)$ is coherent then $\preccurlyeq^{\mrm{c}}$ coincides with $\preccurlyeq$.

\begin{definition}\label{def:coherentisation} Let $(\mathbf{K},\preccurlyeq)$ be a weak $\mrm{AEC}$ with $\mathrm{LS}(\mathbf{K}, \preccurlyeq)=\mu$, then for $A,B\in \mathbf{K}$ with $A\leq B$ we define $A\preccurlyeq^{\mrm{c}} B$ if there are a structure $C\in \mathbf{K}$ with $|C|=|A|+|B|+\mu$, and an embedding $f:B\to C$, such that $A\preccurlyeq C$ and $f(B)\preccurlyeq C$. We refer to $(\mathbf{K},\preccurlyeq^{\mrm{c}})$ as the \emph{coherentisation} of $(\mathbf{K},\preccurlyeq)$.
\end{definition}

\subsection{Preliminaries on infinitary combinatorics}\label{pre:inf_comb}

For readers of various background, we recall the notions of club and stationary set, and the existence under suitable set-theoretical axioms of stationary sets satisfying some special conditions. We refer the reader to \cite{jech} and  \cite{EM} for references.

\begin{remark}\label{remark:regular}
	Given an infinite ordinal $\alpha$, we write $\mrm{cf}(\alpha)$ for its \emph{cofinality}. We recall that an infinite cardinal $\lambda$ is \emph{regular} if $\mrm{cf}(\lambda)=\lambda$ and it is \emph{singular} otherwise. In particular, $\omega$ is regular as clearly $\mrm{cf}(\omega)=\omega$.
\end{remark}

\begin{definition}
	Let $\alpha$ be a limit ordinal with uncountable cofinality. A \emph{closed unbounded set (club)} $C\subseteq \alpha$ is a subset of $\alpha$ which is unbounded and, for every sequence $(\beta_i)_{i<\gamma}$ with $0<\gamma<\alpha$ and $\beta_i\in C$ for all $i<\gamma$ we have that $\bigcup_{i<\gamma}\beta_i \in C$.  A \emph{stationary set} in $\alpha$ is a subset $X\subseteq \alpha$ such that $X\cap C\neq \emptyset$ for every club $C$ in $\alpha$. 
\end{definition}

\begin{definition}\label{stationary:condition_E}
	Let $\alpha$ be a limit ordinal with uncountable cofinality and let $\beta<\alpha$ be a limit ordinal. We say that a stationary set $S\subseteq \alpha$ \emph{reflects} at $\beta$ if $S\cap \beta$ is stationary in $\beta$.  Given a regular uncountable cardinal $\kappa$, we let $E(\kappa)$ and $E^+(\kappa)$ be the following statements:
	\begin{enumerate}[(1)]
	\item $E(\kappa)$ is the statement that there exists a stationary set $S\subseteq \kappa$ which does not reflect at any limit $\alpha<\kappa$ and, additionally, $\mrm{cf}(\beta)=\omega$ for all $\beta\in S$;
	\item $E^+(\kappa)$ is the statement that for all regular cardinals $\lambda<\kappa$, there is a stationary set $S_\lambda \subseteq \kappa$ such that $S_\lambda$ does not reflect at any limit ordinal $\alpha<\kappa$ and, additionally, $\mrm{cf}(\beta)=\lambda$ for all $\beta\in S_\lambda$.
\end{enumerate}	
\end{definition}

\noindent We finally recall the following facts (cf. \cite[p.~88]{EM}): Fact (1) follows immediately from the previous definitions, Fact (2) is due to Jensen \cite{Jen} (but see also \cite{Beller}).

\begin{fact}\label{existence:stationary} The following statements hold:
	\begin{enumerate}[(1)]
		\item $E(\omega_1)$ is true;
		\item under the set-theoretical principle $V=L$, the statement $E^+(\kappa)$ holds for every regular, not weakly compact cardinal $\kappa$.
	\end{enumerate}
\end{fact}

\begin{remark}\label{remark:many_stationary}
	We notice that in the two statements above the stationary sets witnessing $E(\omega_1)$ and $E^+(\kappa)$ must contain exclusively limit ordinals, for the cofinality of successor ordinals is simply 1. Moreover, we stress that starting from one stationary set witnessing $E(\omega_1)$ or $E^+(\kappa)$, one can use Solovay's Theorem~\cite[Thm.~8.10]{jech} to obtain a family of disjoint stationary sets of size, respectively, $2^{\aleph_1}$ or $2^\kappa$.
\end{remark}

\section{$\kappa^+$-categorical weak $ \mrm{AECs} $}\label{sec:kappa_cat}

We introduce in this section the main context of our work, namely weak $ \mrm{AEC}$s which are $\kappa^+$-categorical for some (or all) infinite cardinal $\kappa$ and that have so-called $(\mathbf{K}, \preccurlyeq)$-universal models in $\kappa$. In Section~\ref{subsec:limits} we review the notion of limit model and we introduce some background definitions. In Section~\ref{subsec:almost_free} we adapt the notion of $\mathfrak{L}_{\infty,\kappa^+}$-free algebras to the setting of $\kappa^+$-categorical $ \mrm{AEC}$s. In Section~\ref{CP:general_case} we introduce the Construction Principle $\mrm{CP}^\kappa_{\lambda,\delta}(\mathbf{K},\preccurlyeq)$  and we show that, with the possible addition of the set-theoretical assumption $V=L$, $\mrm{CP}^\kappa_{\lambda,\delta}(\mathbf{K},\preccurlyeq)$ entails the existence of $\mathfrak{L}_{\infty,\kappa^+}$-free structures of size $\kappa^+$ which are not in $\mathbf{K}$. We start by fixing some definitions. The notation $F_{\mathbf{K}}(\kappa)$ in the next definition is meant to indicate the analogy with the case of varieties of algebras, which are classes with a \emph{unique} free algebra of size $\kappa$ (for any $\kappa$ greater than the size of the language).

\begin{definition}
	Let $(\mathbf{K},\preccurlyeq)$ be an abstract class, $(\mathbf{K},\preccurlyeq)$  is \emph{$\kappa$-categorical} if there is one model  in $\mathbf{K}$ of size $\kappa$ up to isomorphism. We denote this model by $F_{\mathbf{K}}(\kappa)$ and we omit the index $\mathbf{K}$ when it is clear from the context. We say that $(\mathbf{K},\preccurlyeq)$  is \emph{uncountably categorical} if it is $\kappa$-categorical for all $\kappa\geq \aleph_1$.
\end{definition}

\begin{definition}\label{def:universality}
	Let $(\mathbf{K},\preccurlyeq)$ be an abstract class. We say that $B\in \mathbf{K}$ is \emph{$(\mathbf{K}, \preccurlyeq)$-universal} if for every $C\in \mathbf{K}$ with $|C|\leq |B|$ there is a strong embedding $f:C\to B$. Let $A,B\in\mathbf{K}$, we say that $B$ is \emph{$(\mathbf{K}, \preccurlyeq)$-universal over $A$} if for every $C\in \mathbf{K}$ such that $A\preccurlyeq C$ and $|C|=|A|$  there is a strong embedding $f:C\to B$ with $f\restriction A=\mrm{id}_A$.
\end{definition}

\noindent The following definition provides the setting of this section. As we shall justify later with Proposition~\ref{prop_1:amalg_class} and Lemma~\ref{prop:varieties_canon}, this definition accommodates the case of free algebras with the relation of being a free factor, but is more general. Essentially, \ref{context:strong_categoricity}(2) states that $\mathbf{K}_{\kappa}$ has no maximal model.

\begin{definition}\label{context:strong_categoricity}
	Let $(\mathbf{K},\preccurlyeq)$ be a weak $\mrm{AEC}$ with $\mathrm{LS}(\mathbf{K}, \preccurlyeq)=\mu$ and let $\kappa\geq \mu+\aleph_0$, we say that $(\mathbf{K},\preccurlyeq)$ is \emph{strongly $\kappa^+$-categorical} if it satisfies the following properties:
	\begin{enumerate}[(1)]
		\item\label{strong_categoricity:one} $\mathbf{K}$ is $\kappa^+$-categorical;
		\item\label{no-nonsense}\label{strong_categoricity:three} for every $A\in \mathbf{K}_{\kappa}$ there is $B\in \mathbf{K}_{\kappa}$ with $A\neq B$ and  $A\preccurlyeq B$;
		\item\label{strong_categoricity:two} for every $A\in \mathbf{K}_{\kappa}$ there is $B\in \mathbf{K}_{\kappa}$ which is $(\mathbf{K}, \preccurlyeq)$-universal over $A$.
	\end{enumerate}
\end{definition}

\begin{context}\label{context:one}
	Unless stated otherwise, for the whole of Section~\ref{sec:kappa_cat}, $(\mathbf{K},\preccurlyeq)$ is a fixed weak $\mrm{AEC}$ with $\mathrm{LS}(\mathbf{K}, \preccurlyeq)=\mu$ and which is strongly $\kappa^+$-categorical, i.e., it satisfies Conditions~\ref{context:strong_categoricity}(\ref{strong_categoricity:one})-(\ref{strong_categoricity:two}) for some fixed cardinal $\kappa\geq \mu+\aleph_0$.
\end{context}

\subsection{Limit Models}\label{subsec:limits}

We recall the notion of limit model, which was first considered in \cite[Ch.~II]{Sh2} and which is now a standard tool from the literature on $ \mrm{AEC}$s. Essentially, limit models provide a version of the first-order notion of saturation for $\mrm{AEC}$s.

\begin{definition}
	Let $(\mathbf{K},\preccurlyeq)$ be an abstract class and $B\in \mathbf{K}$, we say that $(B_i)_{i<\gamma}$ is a \emph{$\gamma$-filtration} of $B$ in $(\mathbf{K},\preccurlyeq)$ if $(B_i)_{i<\gamma}$ is a continuous $\preccurlyeq$-chain of structures in $\mathbf{K}$ such that $B=\bigcup_{i<\gamma}B_i$. A \emph{filtration} of $B\in \mathbf{K}$ is simply a $\gamma$-filtration of $B$  in $(\mathbf{K},\preccurlyeq)$ for some ordinal $\gamma$.
\end{definition}

\begin{definition}\label{def:limits}
	Let $(\mathbf{K},\preccurlyeq)$ be an abstract class  and let $A \preccurlyeq B\in \mathbf{K}$. Let $\kappa$ be a cardinal and $\delta$ a limit ordinal such that $\delta<\kappa^+$, we say that $B$ is a \emph{$(\kappa,\delta)$-$(\mathbf{K},\preccurlyeq)$-limit model over $A$} if there is a continuous $\preccurlyeq$-chain of models $(B_i)_{i<\delta}$ s.t.:
	\begin{enumerate}[(i)]
		\item $B_0=A$;
		\item $B=\bigcup_{i < \delta} B_i $;	
		\item for $i <\delta$, $B_{i+1}$ is $(\mathbf{K}, \preccurlyeq)$-universal over $B_i$;
		\item for $i < \delta$, $|B_i| = \kappa$.
	\end{enumerate}	
	In other words, a \emph{$(\kappa,\delta)$-$(\mathbf{K}, \preccurlyeq)$-limit model over $A$} is a model that admits a $\delta$-filtration of models of size $\kappa$ starting from $A$, such that each model indexed by a successor ordinal is $(\mathbf{K}, \preccurlyeq)$-universal over its predecessor.
\end{definition}

We start by showing the existence of $(\kappa,\delta)$-$(\mathbf{K}, \preccurlyeq)$-limit model. Additionally, in the next proposition we also establish the uniqueness of limit models with filtrations of the same cofinality. We remark that this is  a well-known fact from the theory of $\mrm{AEC}$s (see e.g., \cite[Lem.~10.8]{categoricity}), but we provide details to emphasize that it holds also in weak $ \mrm{AEC}$s.

\begin{proposition}\label{prop:uniqueness_limits}
	(Recalling \ref{context:strong_categoricity}.) Let $A\in \mathbf{K}_{\kappa}$, then for all limit $\delta< \kappa^+$  there is a model $B\in \mathbf{K}_{\kappa}$ which is a $(\kappa,\delta)$-$(\mathbf{K}, \preccurlyeq)$-limit model over $A$. Moreover, if $\gamma< \kappa^+$ is a limit, $\mrm{cf}(\delta)=\mrm{cf}(\gamma)$ and $C$ is a $(\kappa,\gamma)$-$(\mathbf{K}, \preccurlyeq)$-limit model over $A$, then $B\cong_A C$, i.e., there is an isomorphism $f:B\to C$ such that $f\restriction A = \mrm{id}_A$.
\end{proposition}
\begin{proof}
	Let $\delta< \kappa^+$, we first show the existence of a $(\kappa,\delta)$-$(\mathbf{K}, \preccurlyeq)$-limit $B$ over $A$. We define by induction on $\alpha<\delta$ a $\preccurlyeq$-chain $(B_i)_{i<\delta}$ with $B_i\in \mathbf{K}$ for all $i<\delta$.  If $\alpha=0$ then we let $B_0=A\in\mathbf{K}_{\kappa}$.  If $\alpha$ is limit we let $B_{\alpha}=\bigcup_{i<\alpha}B_i$. Since $(\mathbf{K},\preccurlyeq)$ is a weak $ \mrm{AEC}$ it follows by closure under $\preccurlyeq$-chains that $B_\alpha\in \mathbf{K}$. Since $\alpha<\delta<\kappa^+$ and by induction hypothesis all $B_i$ for $i<\alpha$ have power $\kappa$, it follows that $|B_\alpha|=\kappa$. If $\alpha=\beta+1$ then we define $B_\alpha\in \mathbf{K}_{\kappa}$ to be $(\mathbf{K}, \preccurlyeq)$-universal over $B_\beta$, this exists by \ref{context:strong_categoricity}(\ref{strong_categoricity:two}). The resulting model $B_\delta=\bigcup_{i<\delta}B_{i}$ is a $(\kappa,\delta)$-$(\mathbf{K}, \preccurlyeq)$-limit model over $A$.
	
	\smallskip
	\noindent Now, let $\delta< \kappa^+$, $\gamma< \kappa^+$ with $\mrm{cf}(\delta)=\mrm{cf}(\gamma)$. Let  $B=\bigcup_{i<\delta}B_i$ and $C=\bigcup_{i<\gamma}C_i$ witness the fact that $B$ and $C$ are, respectively, $(\kappa,\delta)$-$(\mathbf{K}, \preccurlyeq)$-limit and \mbox{$(\kappa,\gamma)$-$(\mathbf{K}, \preccurlyeq)$-limit} over $A$. Let $\varepsilon\coloneqq \mrm{cf}(\delta)=\mrm{cf}(\gamma)$, then there are subsets $I\subseteq \delta$ and $J\subseteq \gamma$ with both $I$ and $J$ of order type $\varepsilon$ and closed under suprema of chains of length $<\varepsilon$ such that $B=\bigcup_{i\in I}B_i$ and $C=\bigcup_{i\in J}C_i$. Also, we assume without loss of generality that $0\in I$ and $0\in J$.
	
	\smallskip
	\noindent 	Since  $I$ and $J$ are both of order type $\varepsilon$, and they are closed under suprema of chains  of length $<\varepsilon$, we can find two increasing and continuous enumerations $(\alpha_i)_{i<\varepsilon}$ and $(\beta_i)_{i<\varepsilon}$ of the ordinals from $I$ and $J$. We define inductively on $i<\varepsilon$ two $\subseteq$-chains of strong embeddings $f_{i}:B_{\alpha_i}\to C_{\beta_i}$ and $g_i:C_{\beta_i}\to B_{\alpha_{i+1}}$ with the following properties:
	\begin{enumerate}[(a)]
		\item $g_{i}\circ f_i=\mrm{id}_{B_{\alpha_i}}$ and $f_{i+1}\circ g_{i}=\mrm{id}_{C_{\beta_i}}$ for all $i<\varepsilon$;
		\item if $i\leq \varepsilon$ is limit, $f_i=\bigcup_{j<i}f_j$ and $g_i=\bigcup_{j<i}g_{j}$ are isomorphisms.
	\end{enumerate}
	If (a)-(b) hold  then for $i=\varepsilon$ this shows in particular that $B\cong_A C$, as desired.
	
	\smallskip
	\noindent For $i=0$ we have by assumption that $\alpha_0=\beta_0=0$, thus we simply let  $f_0=g_0=\mrm{id}_A$ (as $B_0=C_0=A$). If $i$ is a limit ordinal we simply let $f_i=\bigcup_{j<i}f_j$ and $g_i=\bigcup_{j<i}g_{j}$. Since $I$ and $J$ are both closed under suprema of chains of length $<\varepsilon$, it follows that $\alpha_i=\bigcup_{j<i}\alpha_j$ and $\beta_i=\bigcup_{j<i}\beta_j$, thus $B_{\alpha_i}=\bigcup_{j<i}B_{\alpha_j}$ and $C_{\beta_i}=\bigcup_{j<i}C_{\beta_j}$ and so the maps $f_i$ and $g_i$ are welldefined maps between $B_{\alpha_i}$ and $C_{\beta_i}$. By the inductive property (a) for $j<i$ it follows that both $f_i$ and $g_i$ are isomorphisms and so (b) is verified (notice in particular that this argument did not require Smoothness).
	
	\smallskip
	\noindent Consider the successor case and suppose we have defined $f_i:B_{\alpha_i}\to C_{\beta_i}$ and $g_i:C_{\beta_i}\to B_{\alpha_{i+1}}$. Notice that by Definition~\ref{def:universality} and the choice of the two $\preccurlyeq$-chains $(B_i)_{i<\delta}$ and $(C_i)_{i<\gamma}$ it follows immediately that if $j<\ell<\delta$ then $B_\ell$ is $(\mathbf{K}, \preccurlyeq)$-universal over $B_j$ and, similarly, if $j<\ell<\gamma$ then $C_\ell$ is $(\mathbf{K}, \preccurlyeq)$-universal over $C_j$. Thus in particular we have that  $C_{\beta_{i+1}}$ is $(\mathbf{K}, \preccurlyeq)$-universal over $C_{\beta_{i}}$, whence we can find a strong embedding $f_{i+1}:B_{\alpha_{i+1}}\to C_{\beta_{i+1}}$ such that $f_{i+1}\restriction g_i(C_{\beta_i})=g^{-1}_i$, i.e., $f_{i+1}\circ g_{i}=\mrm{id}_{C_{\beta_{i}}}$. By the same argument we have that $B_{\alpha_{i+2}}$ is $(\mathbf{K}, \preccurlyeq)$-universal over $B_{\alpha_{i+1}}$, and thus we can also find a strong embedding $g_{i+1}:C_{\beta_{i+1}}\to B_{\alpha_{i+2}}$ with $g_{i+1}\circ f_{i+1}=\mrm{id}_{B_{\alpha_{i+1}}}$. Thus property (a) holds and by the reasoning above this completes our proof.
\end{proof}

\noindent The following corollary is a straightforward application of the previous proposition, and it will be often used in the rest of this article.

\begin{corollary}\label{prop:limitoflimit}
	Let $\gamma<\kappa^+$ be a limit ordinal and let $(D_i)_{i\leq \alpha}$ be a continuous $\preccurlyeq$-chain in $\mathbf{K}_\kappa$ such that  $D_{i+1}$ is $(\kappa,\gamma)$-$(\mathbf{K}, \preccurlyeq)$-limit over $D_i$ for all $i<\alpha$. If $\mrm{cf}(\alpha)=1$ (i.e.,  $\alpha$ is successor) or $\mrm{cf}(\alpha)=\gamma$ then $D_\alpha$ is $(\kappa,\gamma)$-$(\mathbf{K}, \preccurlyeq)$-limit over $D_0$. 
\end{corollary}
\begin{proof}
	Since for all $i<\alpha$ we have that $D_{i+1}$ is $(\kappa,\gamma)$-$(\mathbf{K}, \preccurlyeq)$-limit over $D_i$, we have that $D_{i+1}=\bigcup_{j<\gamma}D_i^j$ for some $\gamma$-filtration of $D_{i+1}$ where each $D_i^{j+1}$ is $(\mathbf{K}, \preccurlyeq)$-universal over $D_i^j$. Thus the chain $(D_i)_{i< \alpha}$ of $(\kappa,\gamma)$-$(\mathbf{K}, \preccurlyeq)$-limit models determines a continuous chain $(D^j_i)_{i<\alpha, \, j<\gamma}$ of order type $\gamma\cdot \alpha$ where each $D_i^{j+1}$ is $(\mathbf{K}, \preccurlyeq)$-universal over $D_i^j$. Since $\mrm{cf}(\alpha)=1$ or $\mrm{cf}(\alpha)=\gamma$, it follows in both cases that $\mrm{cf}(\gamma\cdot \alpha)=\mrm{cf}(\gamma)$. From the uniqueness of limit models from \ref{prop:uniqueness_limits} we have that $\bigcup_{i<\alpha}D_i\cong_{D_0} D_1=\bigcup_{j<\gamma} D^j_0$, and so $\bigcup_{i<\alpha}D_i$ is $(\kappa,\gamma)$-$(\mathbf{K}, \preccurlyeq)$-limit over $D_0$. Now, if $\alpha$ is limit then $\bigcup_{i<\alpha}D_i=D_{\alpha}$ and so $D_\alpha$ is $(\kappa,\gamma)$-$(\mathbf{K}, \preccurlyeq)$-limit over $D_0$. If otherwise $\alpha=\beta+1$ is successor, then $\bigcup_{i<\alpha}D_i=D_{\beta}$ and so $D_{\beta}$ is $(\kappa,\gamma)$-$(\mathbf{K}, \preccurlyeq)$-limit over $D_0$. Since by construction we have that $D_\alpha$ is $(\kappa,\gamma)$-$(\mathbf{K}, \preccurlyeq)$-limit over $D_\beta$ and clearly $\mrm{cf}(\gamma+\gamma)=\gamma$, it then follows as before from \ref{prop:uniqueness_limits} that $D_{\alpha}$ is $(\kappa,\gamma)$-$(\mathbf{K}, \preccurlyeq)$-limit over $D_0$.
\end{proof}

\subsection{Strongly $\kappa^+$-free structures in weak $ \mrm{AECs}$}\label{subsec:almost_free}

In the setting of universal algebra, an algebra in a variety $\mathbf{V}$ is said to be $\mathfrak{L}_{\infty,\kappa^+}$-free if it is equivalent in the logic $\mathfrak{L}_{\infty,\kappa^+}$ to the unique free algebra in $\mathbf{FV}$ of size $\kappa^+$ (cf.~\cite{MSh,EM2}). In this section, we extend the notion of $\mathfrak{L}_{\infty,\kappa^+}$-free structure from the setting of universal algebra to an arbitrary (strongly) $\kappa^+$-categorical weak $ \mrm{AEC}$. The use of the adjectives ``free'' and ``almost-free'' for arbitrary weak $ \mrm{AEC}$s is meant to reflect the fact that one the main source of motivation for the current work are the weak $ \mrm{AEC}$s $(\mathbf{FV},\sleq)$ of free algebras with the relation of free factor with a free complement --- we shall prove later in Lemma~\ref{prop:varieties_canon} that $(\mathbf{FV},\sleq)$ is a weak $ \mrm{AEC}$, which additionally fails to satisfy Smoothness when (CP) holds in $\mathbf{V}$. Firstly, we show that under the assumptions of Definition~\ref{context:strong_categoricity} the free structure of size $\kappa^+$ in $\mathbf{K}$ admits a filtration by limit models. Recall that, crucially, we always work in the context fixed in \ref{context:one}.

\begin{proposition}\label{free_filtration}\label{free_filtration_kappa}
	For every model $D\in\mathbf{K}_{\kappa}$ and every limit $\gamma< \kappa^+$ there is a filtration $(A_i)_{i<\kappa^+}$ of $F(\kappa^+)$ such that $A_0\cong D$ and for all $\alpha<\beta<\kappa^+$, with $\beta$ successor, we have that $A_{\beta}$ is a $(\kappa,\gamma)$-$(\mathbf{K}, \preccurlyeq)$-limit over $A_\alpha$.
\end{proposition}
\begin{proof}
	Let $A_0 \cong D \in \mathbf{K}_{\kappa}$, then, by Proposition~\ref{prop:uniqueness_limits}, $A_0$ can be extended to a $(\kappa,\gamma)$-$(\mathbf{K}, \preccurlyeq)$-limit model $A_1$ (which by \ref{context:strong_categoricity}(\ref{no-nonsense}) is a strict extension of $A_0$, i.e., $A_0\subsetneq A_1$). Proceeding inductively, we obtain a strict continuous $\preccurlyeq$-chain of models $(A_i)_{i<\kappa^+}$ of size $\kappa$ such that $A_{i+1}$ is $(\kappa,\gamma)$-$(\mathbf{K}, \preccurlyeq)$-limit over $A_i$ for every $i<\kappa^+$. Let $A=\bigcup_{i<\kappa^+}A_i$, then $A$ has size $\kappa^+$ and thus $A\cong F(\kappa^+)$. 
	\newline Now, let $\alpha<\beta<\kappa^+$ and suppose that $\beta$ is successor, thus in particular $\mrm{cf}(\beta)=1$. By construction, the chain of models $(A_i)_{\alpha\leq i\leq \beta}$ is a continuous $\preccurlyeq$-chains of structures in $\mathbf{K}_\kappa$ such that for every $i< \beta$ with $\alpha\leq i$ the model $A_{i+1}$ is $(\kappa,\gamma)$-$(\mathbf{K}, \preccurlyeq)$-limit over $A_i$. It follows that $A_\beta$ and $(A_i)_{\alpha\leq i\leq \beta}$  satisfy the assumptions from  Corollary~\ref{prop:limitoflimit} and thus $A_{\beta}$ is a $(\kappa,\gamma)$-$(\mathbf{K}, \preccurlyeq)$-limit over $A_\alpha$.
\end{proof}

We next introduce the notion of $\mathfrak{L}_{\infty,\kappa^+}$-free structure and provide an ``algebraic'' characterisation of this notion. Essentially, this extends the classical characterisation of almost free objects in varieties of algebras (cf.~\cite[Thm.~1.3]{MSh3} and  \cite{Kueker}).

\begin{definition}\label{def:L-free}\label{def:strong_free} Let $(\mathbf{K},\preccurlyeq)$ be a weak $\mrm{AEC}$ with $\mathrm{LS}(\mathbf{K}, \preccurlyeq)=\mu$ and which is strongly $\kappa^+$-categorical (i.e., it satisfies Conditions~\ref{context:strong_categoricity}(\ref{strong_categoricity:one})-(\ref{strong_categoricity:two})), then we define the following notions.
\begin{enumerate}[(1)]
	\item A structure $A$ is  \emph{$\mathfrak{L}_{\infty,\kappa^+}$-free} if $A\equiv_{\infty,\kappa^+} F(\kappa^+)$, i.e., $A$ and $F(\kappa^+)$ satisfy exactly the same sentences of the infinitary logic $\mathfrak{L}_{\infty,\kappa^+}$.
	
	\item A structure $A$  is \emph{strongly $\kappa^+$-free} if for some limit $\delta< \kappa^+$ there is a set $\mathcal{S}$ of substructures of $A$ from $\mathbf{K}_\kappa$ such that, for every $B\in \mathcal{S}$ and $C\subseteq A$ of size $\leq \kappa$, there is some $B'\in \mathcal{S}$ with $B\cup C\subseteq B'$ and $B'$ is a $(\kappa,\delta)$-$(\mathbf{K}, \preccurlyeq)$-limit over $B$.
\end{enumerate} 
\end{definition}

\begin{remark}
	Definition~\ref{def:strong_free}(2) adapts the algebraic notion of strongly $\kappa^+$-free algebras to the setting of weak $ \mrm{AEC}$s. However, in Definition~\ref{def:strong_free}(2) we had to strengthen the definition from Eklof and Mekler in \cite[p.~92]{EM2} and from Mekler and Shelah in \cite{MSh3} and require that $B'$ is a  $(\kappa,\delta)$-$(\mathbf{K}, \preccurlyeq)$-limit over $B$ and not simply that $B\preccurlyeq B'$.  We notice that limit models essentially replace the notion of \emph{$\kappa$-purity} from \cite[Theorem~1.3]{MSh3}.
\end{remark}

\noindent We notice that, in the next theorem, the structure $A$ does not necessarily belong to the class $\mathbf{K}$. Later in \ref{thm:2}, we will prove the existence of a $\mathfrak{L}_{\infty,\kappa^+}$-free  structure exactly to show that $\mathbf{K}$ is not axiomatisable in $\mathfrak{L}_{\infty,\kappa^+}$ (under suitable set-theoretical assumptions).

\begin{theorem}\label{thm:kueker}
	Let $(\mathbf{K},\preccurlyeq)$ be a weak $\mrm{AEC}$ which is strongly $\kappa^+$-categorical (i.e., satisfying the conditions from \ref{context:strong_categoricity}), then the following are equivalent:
	\begin{enumerate}[(1)]
		\item $A$ is $\mathfrak{L}_{\infty,\kappa^+}$-free (as in Def.~\ref{def:L-free}(1));
		\item for all limit $\delta< \kappa^+$ there is a set $\mathcal{S}_\delta$ of substructures of $A$ from $\mathbf{K}_\kappa$ such that, for every $B\in \mathcal{S}_\delta$ and $C\subseteq A$ of size $\leq \kappa$, there is some $B'\in \mathcal{S}_\delta$ such that $B\cup C\subseteq B'$ and $B'$ is a $(\kappa,\delta)$-$(\mathbf{K}, \preccurlyeq)$-limit over $B$;
		\item $A$ is strongly $\kappa^+$-free  (as in Def.~\ref{def:strong_free}(2)).
	\end{enumerate}
\end{theorem}
\begin{proof}
	We first prove that (1) entails (2). Suppose $A$ is $\mathfrak{L}_{\infty,\kappa^+}$-free, then by Fact \ref{criterion:Marker} we also have that $A\cong_{\kappa^+} F(\kappa^+)$ (cf.~\ref{remark:partial_iso}). Let $\mathcal{F}$ be a family of partial isomorphisms witnessing $A\cong_{\kappa^+} F(\kappa^+)$ and suppose without loss of generality that $f\in\mathcal{F}$ entails $g\in\mathcal{F}$ for all $g\subseteq f$. Also, let $\delta<\kappa^+$ and fix a filtration $(D_i)_{i<\kappa^+}$ of $F(\kappa^+)$ with each $D_{i+1}$ being $(\kappa,\delta)$-$(\mathbf{K}, \preccurlyeq)$-limit over $D_i$ (this exists by Proposition~\ref{free_filtration_kappa}).  We define $\mathcal{S}_\delta$ to be the family of all substructures $B\leq A$ for which there is some $f\in \mathcal{F}$ such that $f:B\cong D_i$ for some $i<\kappa^+$.  We show that $\mathcal{S}_\delta$ witnesses clause (2) with respect to $\delta$.  Let $B\in \mathcal{S}_\delta$ and let $C\subseteq A$ be of size $\leq \kappa$. By definition of $\mathcal{S}_\delta$ there is some $\alpha<\kappa^+$ and $f\in \mathcal{F}$ such that $f:B\cong D_\alpha$. Then, since $A\cong_{\kappa^+} F(\kappa^+)$, we can find by \ref{criterion:Marker}(1) some successor ordinal $\beta>\alpha$  and $g\supseteq f$ such that $g\in \mathcal{F}$, $C\subseteq \mrm{dom}(g)$ and $ \mrm{ran}(g)\subseteq D_\beta$. Then, by reasoning similarly using \ref{criterion:Marker}(2), we can find some $h\in\mathcal{F}$ such that $h\supseteq g$ and $\mrm{ran}(h)=D_\beta$. Let $B'=\mrm{dom}(h)$, then $B'\in \mathcal{S}_\delta$ and $B\cup C\subseteq B'$. Since $\beta$ is successor it follows by Corollary~\ref{prop:limitoflimit} that $D_\beta$ is $(\kappa,\delta)$-$(\mathbf{K}, \preccurlyeq)$-limit over $D_\alpha$. Since $h$ is an isomorphism extending $f$, we obtain that  $B'$ is $(\kappa,\delta)$-$(\mathbf{K}, \preccurlyeq)$-limit over $B$. Using $\mathcal{F}\neq\emptyset$ it follows by a similar argument that $\mathcal{S}_\delta\neq \emptyset$. We thus conclude that $A$ satisfies clause (2).
	
	\smallskip
	\noindent Clearly, (2) entails (3) immediately by Definition~\ref{def:strong_free}(2), so it remains to show that  (3) entails~(1). Suppose that for some limit $\delta< \kappa^+$ there is a set $\mathcal{S}$ of substructures of $A$ satisfying (3). We exhibit a back-and-forth system $\mathcal{F}$ between $A$ and $F(\kappa^+)$ satisfying the conditions from \ref{criterion:Marker}. Let $E\leq A$ be any substructure of $A$ of size $\kappa$ (this exists by $\mathrm{LS}(\mathbf{K}, \preccurlyeq)=\mu\leq \kappa$) and fix by Proposition~\ref{free_filtration_kappa} a filtration  $F(\kappa^+)=\bigcup_{i<\kappa^+}D_i$  with $D_0\cong E$ and $D_\beta$ a $(\kappa,\delta)$-$(\mathbf{K}, \preccurlyeq)$-limit over $D_\alpha$ for all successors $\beta$ with $\alpha<\beta<\kappa^+$. Let $\mathcal{F}$ be the family of all isomorphisms $f:B\to D_{\alpha}$ for some $B\leq A$ and $\alpha<\kappa^+$. Clearly, since $E\leq A$ and $E\cong D_0$, we immediately have that $\mathcal{F}$ is nonempty. Let $f:B\cong D_{\alpha}$ be a given isomorphism in $\mathcal{F}$, we verify the back and forth conditions from Fact~\ref{criterion:Marker}. First, consider some $C\subseteq A$ with $|C|<\kappa^+$. By (3) there is some $B'\leq A$ with $B\cup C\subseteq B'$ and $B'$ is $(\kappa,\delta)$-$(\mathbf{K}, \preccurlyeq)$-limit over $B$. Since $D_{\alpha+1}$ is also $(\kappa,\delta)$-$(\mathbf{K}, \preccurlyeq)$-limit over $D_\alpha$, it follows from Proposition~\ref{prop:uniqueness_limits} that there is $g\supseteq f$ such that $g:B'\cong D_{\alpha+1}$, whence $C\subseteq \mrm{dom}(g)$ and $g\in \mathcal{F}$. Secondly, consider any set $C\subseteq F(\kappa^+)$ with $|C|<\kappa^+$. Since $\kappa^+$ is regular we have that $C\subseteq D_{\beta}$ for some $\alpha<\beta<\kappa^+$ and, additionally, we can assume that $\beta$ is successor. By the choice of the filtration $(D_i)_{i<\kappa^+}$ we have that $D_{\beta}$ is a $(\kappa,\delta)$-$(\mathbf{K}, \preccurlyeq)$-limit over $D_{\alpha}$. Also, again by assumption (3), there is some $B'\leq A$ such that $B'$ is $(\kappa,\delta)$-$(\mathbf{K}, \preccurlyeq)$-limit over $B$. By Proposition~\ref{prop:uniqueness_limits} we obtain an isomorphism $g\supseteq f$ witnessing $g:B'\cong D_\beta$ and satisfying $C\subseteq \mrm{ran}(g)$, whence $g\in\mathcal{F}$. This shows that $A\cong_{\kappa^+}F(\kappa^+)$ and thus by \ref{criterion:Marker} it completes our proof.
\end{proof}

\subsection{The Construction Principle $\mrm{CP}^\kappa_{\lambda,\delta}{(\mathbf{K},\preccurlyeq)}$}\label{CP:general_case}

We introduce in this section a general version of the Construction Principle for weak $ \mrm{AEC}$s satisfying the assumptions from \ref{context:strong_categoricity}. Recall from Definition~\ref{def:coherentisation} the  coherentisation $\precleq$ of $\preccurlyeq$; so that we work in a context where $\mathbf{K}$ is considered together with the two relations $\preccurlyeq$ and $\precleq$. Notice that the following Construction Principle $\mrm{CP}^\delta_{\kappa,\lambda}{(\mathbf{K},\preccurlyeq)}$ essentially amount to a strong form of failure of the Smoothness Axiom for  $\mrm{AEC}$s, and thus can hold only for weak $\mrm{AEC}$s (and almost $\mrm{AEC}$s) which do not satisfy Smoothness (but which may or may not satisfy Coherence). In the following definition, recall also that for $A,B\in \mathbf{K}$ we clearly have that $A\not \precleq B $ entails $A\not \preccurlyeq B $.

\begin{definition}\label{CP-AEC:kappa_edit}
	Let $(\mathbf{K},\preccurlyeq)$ be a weak $\mrm{AEC}$ with $\mathrm{LS}(\mathbf{K}, \preccurlyeq)=\mu$ and which is strongly $\kappa^+$-categorical for some $\kappa\geq \mu$, i.e., $(\mathbf{K},\preccurlyeq)$  satisfies Conditions~\ref{context:strong_categoricity}(\ref{strong_categoricity:one})-(\ref{strong_categoricity:two}). For any regular cardinal $\lambda$ such that $\lambda\leq \kappa$ and any limit ordinal $\delta< \kappa^+$, we define the \emph{Construction Principle $\mrm{CP}^\delta_{\kappa,\lambda}{(\mathbf{K},\preccurlyeq)}$} for weak $ \mrm{AECs} $ as the following condition:
	\begin{enumerate}[(4)]
		\item[(4)]\label{cond:CP-general} there are a continuous $\preccurlyeq$-chain of models $(A_i)_{i<\lambda}$ in $\mathbf{K}_\kappa$ and $B\in\mathbf{K}_\kappa$ such that:
		\begin{enumerate}[(a)]
			\item $A_{i+1}$ is a $(\kappa,\delta)$-$(\mathbf{K}, \preccurlyeq)$-limit model over $A_i$ for all $i<\lambda$;			
			\item  $B$ is a $(\kappa,\delta)$-$(\mathbf{K}, \preccurlyeq)$-limit model over $A_i$ for all $i<\lambda$;
			\item $\bigcup_{i<\lambda}A_i \not \precleq B$ (recall \ref{def:coherentisation}).
		\end{enumerate}
	\end{enumerate} 
Whenever $\lambda=\delta$ we simply write $\mrm{CP}_{\kappa,\lambda}{(\mathbf{K},\preccurlyeq)}$ instead of $\mrm{CP}^{\lambda}_{\kappa,\lambda}{(\mathbf{K},\preccurlyeq)}$. Notice in particular that $\aleph_0$ is regular (cf.~\ref{remark:regular}) and so $\mrm{CP}_{\kappa,\omega}{(\mathbf{K},\preccurlyeq)}$ is well-defined.
\end{definition}

\begin{context}
	We assigned number (4) to the clauses of $\mrm{CP}^\delta_{\kappa,\lambda}{(\mathbf{K},\preccurlyeq)}$ in the previous definition to make it explicit that we always consider in this section the Construction Principle $\mrm{CP}^\delta_{\kappa,\lambda}{(\mathbf{K},\preccurlyeq)}$ in the context of  strongly $\kappa^+$-categorical weak  $\mrm{AEC}$s. More precisely, we fix in this section a weak $\mrm{AEC}$ $(\mathbf{K},\preccurlyeq)$ with $\mathrm{LS}(\mathbf{K}, \preccurlyeq)=\mu$  and which satisfies Conditions~\ref{context:strong_categoricity}(\ref{strong_categoricity:one})-(\ref{strong_categoricity:two}) and Condition~\ref{CP-AEC:kappa_edit}(\hyperref[cond:CP-general]{4}), i.e.,  $\mrm{CP}^\delta_{\kappa,\lambda}{(\mathbf{K},\preccurlyeq)}$, for some regular cardinal $\lambda\leq \kappa$ and some limit ordinal $\delta< \kappa^+$.
\end{context}

\begin{remark}
	As the reader can immediately see, $\mrm{CP}^\delta_{\kappa,\lambda}{(\mathbf{K},\preccurlyeq)}$ is analogous to the Eklof-Mekler-Shelah Construction Principle \ref{definition CP}, but with two crucial differences.  First, differently from the classical version of CP, in $\mrm{CP}^\delta_{\kappa,\lambda}{(\mathbf{K},\preccurlyeq)}$ we consider chains of arbitrary length $\lambda$. Although in our applications we will be dealing exclusively with chains of length $\lambda=\omega$, we believe that it is worth stressing the generality of our arguments, and the fact that they could be applied to the setting where  the Smoothness Axiom fails with respect to chains strictly longer than $\omega$. Secondly, in  $\mrm{CP}^\delta_{\kappa,\lambda}{(\mathbf{K},\preccurlyeq)}$ we add the further requirement that both $B$ and each structure $A_{i+1}$ are $(\kappa,\delta)$-$(\mathbf{K}, \preccurlyeq)$-limit models over $A_i$. We will consider in our applications only $(\kappa,\omega)$-$(\mathbf{K}, \preccurlyeq)$-limit models, but our argument covers the general case where $B$ and each $A_{i+1}$ are $(\kappa,\delta)$-$(\mathbf{K}, \preccurlyeq)$-limit models over $A_i$ for some $\delta< \kappa^+$. We will consider later in Section~\ref{CP-AEC:coproducts} a version of the Construction Principle $ \mrm{CP}(\mathbf{K},\ast) $ more closely related to the Eklof-Mekler-Shelah Construction Principle. In Section~\ref{CP-AEC:coproducts} we will assume a form of canonical amalgamation (cf.~\ref{context:amalgamation_class}) and show that we do not need the additional constraint that $B$ and $A_{i+1}$ are limit over $A_i$ (as this can be achieved via amalgamation). We elaborate this further in Section~\ref{CP:coproduct_case}.
\end{remark}

\begin{remark}
	The main goal of this section is to study classes of structures which are $\kappa^+$-categorical, but which fail to form an $\mrm{AEC}$ because of the presence of $\mrm{CP}^\kappa_{\lambda,\delta}{(\mathbf{K},\preccurlyeq)}$. We notice that, in a similar direction, Šaroch and Trlifaj consider in \cite{SaTr} the class of $\kappa^{+}$-decomposable modules, which also does not form an $\mrm{AEC}$. In particular, they prove that if a $\kappa^{+}$-decomposable class of modules over an arbitrary ring $R$ is categorical in some $\lambda \geq (2^{\kappa})^{+}$ with $\kappa \geq |R| + \aleph_{0}$, then it is also categorical in all $\lambda \geq (2^{\kappa})^{+}$, thus establishing Shelah's eventual categoricity conjecture relative to $\kappa^{+}$-decomposable modules.
\end{remark}

The next theorem adapts to our general setting the classical strategy of using CP to build non-free $\mathfrak{L}_{\infty,\kappa^+}$-free algebras. Notice that we work under the extra set-theoretical assumption $E^+(\kappa^+)$ and that, as mentioned in \ref{existence:stationary}, for $\kappa=\aleph_0$ the principle $E^+(\kappa^+)$   is redundant, as $E^+(\aleph_1)$ coincides with $E(\aleph_1)$ and $E(\aleph_1)$ is true in ZFC. On the other hand, for $\kappa\geq \aleph_1$, the principle $E^+(\kappa^+)$ always hold in $V=L$, since successor cardinals are regular and not weakly compact (cf.~\ref{stationary:condition_E} and \ref{existence:stationary}).  The next result is thus a theorem of ZFC for $\kappa=\aleph_0$, and a relative consistency statement for $\kappa\geq \aleph_1$. We refer the reader to the classical paper \cite{MSh2} by Mekler and Shelah for the study of the set-theoretical strength of $E^+(\kappa^+)$ and related axioms.

\begin{theorem}[$\mrm{ZFC}$ + $E^+(\kappa^+)$]\label{thm:2}
 Let $(\mathbf{K},\preccurlyeq)$ be a strongly $\kappa^+$-categorical weak $\mrm{AEC}$ with $\mrm{CP}^\delta_{\kappa,\lambda}{(\mathbf{K},\preccurlyeq)}$ as in \ref{CP-AEC:kappa_edit}, then there is an $\mathfrak{L}_{\infty,\kappa^+}$-free structure $M$ of size $\kappa^+$ such that $M\not\cong F(\kappa^+)$. In particular, $\mathbf{K}$ is not axiomatisable in $\mathfrak{L}_{\infty,\kappa^+}$.
\end{theorem}
\begin{proof}
We construct a structure $M$ witnessing the statement. First, since we assume $E^+(\kappa^+)$ and $\lambda$ is a regular cardinal such that $\lambda \leq \kappa$, then there is a stationary set $X\subseteq \kappa^+$ such that $\mrm{cf}(\alpha)=\lambda$ for all $\alpha\in X$ and $X\cap \alpha$ is not stationary for any limit ordinal $\alpha<\kappa^+$ (cf.~\ref{existence:stationary}). We assume without loss of generality that $0\notin X$. By Condition~\ref{CP-AEC:kappa_edit}(\hyperref[cond:CP-general]{4}) we fix a  sequence  $(A_i)_{i<\lambda}$ in $\mathbf{K}_{\kappa}$ and $B\in\mathbf{K}_\kappa$ such that $A_{i+1}$ and $B$ are $(\kappa,\delta)$-$(\mathbf{K}, \preccurlyeq)$-limit over $A_i$ for all $i<\lambda$, and $\bigcup_{i<\lambda}A_i \not \precleq B$.
\newline We use these data to construct a model $M = M^X=\bigcup_{\alpha<\kappa^+}M^X_\alpha$ such that:	
	\begin{enumerate}[(a)]
		\item $M^X_\alpha\in \mathbf{K}_{\kappa}$ for all $\alpha<\kappa^+$;
		\item for all $\alpha<\beta<\kappa^+$ and $\alpha\notin X$ we have that $M^X_{\alpha}\preccurlyeq M^X_{\beta}$ and, if additionally $\beta$ is successor, then $M^X_{\beta}$ is $(\kappa,\delta)$-$(\mathbf{K}, \preccurlyeq)$-limit over $ M^X_{\alpha}$ (so $M^X_{\alpha} \preccurlyeq M^X_{\beta}$);
		\item for limits $\alpha<\kappa^+$, we have $M^X_\alpha=\bigcup_{\beta<\alpha}M^X_\beta$;
		\item $M^X_\alpha\not\precleq M^X_{\alpha+1}$ for all $\alpha\in X$.		
	\end{enumerate} 
	
	\smallskip
	\noindent We describe how to build $M^X_\alpha\in \mathbf{K}_{\kappa}$ recursively on $\alpha<\kappa^+$. At each step $\alpha<\kappa^+$ we assume properties (a)-(d) hold for all $\beta<\alpha$, and we later verify them inductively.
	
	\smallskip
	\noindent \underline{Case 1}. $\alpha=0$. 
	\newline We let $M^X_0=A_0$, where $A_0$ is the first structure from the chain $(A_i)_{i<\lambda}$ fixed above that witnesses the Construction Principle $\mrm{CP}^\delta_{\kappa,\lambda}(\mathbf{K},\preccurlyeq)$.
	
	\smallskip
	\noindent \underline{Case 2}. $\alpha$ is a limit ordinal.
	\newline We let $M^X_{\alpha}=\bigcup_{\gamma<\alpha}M^X_{\gamma}$.
	
	\smallskip
	\noindent \underline{Case 3}. $\alpha=\beta+1$.
	
	\smallskip
	\noindent \underline{Case 3.1}. $\beta\notin X$.
	\newline We let $M^X_{\alpha}$ be a $(\kappa,\delta)$-$(\mathbf{K}, \preccurlyeq)$-limit model over $M^X_{\beta}$ --- this exists by Proposition~\ref{prop:uniqueness_limits}.
	
	\smallskip
	\noindent \underline{Case 3.2}. $\beta \in X$.
	\newline By the choice of the stationary $X$, we have that $\beta$ is a limit ordinal with cofinality $\lambda$. Therefore, there is a sequence $(\gamma_i)_{i<\lambda}$ cofinal in $\beta$ and such that each $\gamma_i$ is a successor ordinal not in $X$.  We construct inductively a continuous $\subseteq$-chain of isomorphisms $(f_i)_{i<\lambda}$ for $i\geq 1$ such that $f_{i}:M^X_{\gamma_i}\cong A_{i}$.
	
	\smallskip
	\noindent \underline{Subcase 3.2(a)}.  $i = 1$.
	\newline If $i=1$ then, by Condition~(b) and $0\notin X$ we have that $M^X_{\gamma_1}$ is a $(\kappa,\delta)$-$(\mathbf{K}, \preccurlyeq)$-limit over $A_0$. Also, by the choice of the models $(A_i)_{i<\lambda}$ we have that $A_1$ is also a $(\kappa,\delta)$-$(\mathbf{K}, \preccurlyeq)$-limit over $A_0$. It then follows from Proposition~\ref{prop:uniqueness_limits} that there is an isomorphism $f_1:M^X_{\gamma_1}\cong A_1$ such that $f_1\restriction A_0=\mrm{id}_{A_0}$. 
	
	\smallskip
	\noindent \underline{Subcase 3.2(b)}.  $i=\varepsilon+1$.
	\newline If $i=\varepsilon+1$, then by Condition~(b) and the fact that $\gamma_{\varepsilon+1}$ is successor we have that $M^X_{\gamma_{\varepsilon+1}}$ is a $(\kappa,\delta)$-$(\mathbf{K}, \preccurlyeq)$-limit over $M^X_{\gamma_\varepsilon}$. Also, by our choice of the structures $(A_i)_{i<\lambda}$ we have that $A_{\varepsilon+1}$ is  a $(\kappa,\delta)$-$(\mathbf{K}, \preccurlyeq)$-limit over $A_\varepsilon$. It follows by Proposition~\ref{prop:uniqueness_limits}  that we can find an isomorphism $f_{\varepsilon+1}\supseteq f_\varepsilon$ such that $f_{\varepsilon+1}:M^X_{\gamma_{\varepsilon+1}}\cong A_{\varepsilon+1}$.	
	
	\smallskip
	\noindent \underline{Subcase 3.2(c)}.  $i$ is a limit ordinal.
	\newline  If $i$ is limit, then we have a $\subseteq$-chain of isomorphisms $f_{\varepsilon}:M^X_{\gamma_\varepsilon}\cong A_\varepsilon$ for all $\varepsilon<i$. Then $f_i=\bigcup_{\varepsilon<i}f_\varepsilon$ is an isomorphism between $M^X_i$ and $A_i$ (notice in particular that we do not need Smoothness here, as each $f_\varepsilon$ is an isomorphism). 
	
	\smallskip
	\noindent We thus obtained a continuous $\subseteq$-chain of isomorphisms $(f_i)_{i<\lambda}$ for $i\geq 1$ such that $f_{i}:M^X_{\gamma_i}\cong A_{i}$. Let $f=\bigcup_{i<\lambda}f_i$, then since the sequence $(\gamma_i)_{i<\lambda}$ is cofinal in $\beta$, it follows that  $f$ is an isomorphism between $M^X_\beta$ and $A\coloneqq\bigcup_{i<\lambda}A_i$.	Now, since the $\preccurlyeq$-chain $(A_i)_{i<\lambda}$ and the model $B$ were chosen to satisfy $\mrm{CP}^\delta_{\kappa,\lambda}(\mathbf{K},\preccurlyeq)$, we have  in particular that $A\leq B$ (but $A\not \precleq B$) and that $B$ is a $(\kappa,\delta)$-$(\mathbf{K}, \preccurlyeq)$-limit over every $A_i$.  We can now define $M^X_\alpha$ to be a model isomorphic to $B$ over $M^X_\beta=\bigcup_{i<\lambda}M^X_{\gamma_i}$, i.e., we define $M^X_\alpha$ so that there exists an isomorphism $g\supseteq f$  such that $g:M^X_\alpha\cong B$. It immediately follows from the choice of $B$ that $M^X_\beta\leq M^X_{\alpha}$, $M^X_\beta\not\precleq M^X_{\alpha}$ and $M^X_\alpha$ is $(\kappa,\delta)$-$(\mathbf{K}, \preccurlyeq)$-limit over $M^X_{\gamma_i}$ for every $i<\lambda$.
	
	\medskip 
	\noindent We verify that this construction satisfies (a)-(d). We proceed by induction and whenever we verify (a)-(d) for the stage $\alpha$ we assume (a)-(d) for all $\beta<\alpha$.
	
	\smallskip
	\noindent \underline{Condition~(a)}. $M^X_{\alpha}\in \mathbf{K}_{\kappa}$ for all $\alpha<\kappa^+$.
	\newline If $\alpha=0$ then $M^X_0= A_0$ and, by the statement of $\mrm{CP}^\delta_{\kappa,\lambda}(\mathbf{K},\preccurlyeq)$, $A_0\in \mathbf{K}_{\kappa}$.
	\newline If $\alpha=\beta+1$ is successor and $\beta\notin X$ then the fact that $M^X_\alpha\in \mathbf{K}_\kappa$ follows immediately by construction and Proposition~\ref{prop:uniqueness_limits}. If $\alpha=\beta+1$ is successor and $\beta\in X$ then $M^X_\alpha\in \mathbf{K}_\kappa$ follows from the fact that $M^X_\alpha\cong B$ and $B\in \mathbf{K}_\kappa$ by $\mrm{CP}^\delta_{\kappa,\lambda}(\mathbf{K},\preccurlyeq)$.
	\newline  If $\alpha$ is a limit ordinal, then $M^X_{\alpha}=\bigcup_{\gamma<\alpha}M^X_{\gamma}$. Since the stationary $X$ does not reflect at $\alpha<\kappa^+$, we have that $X\cap \alpha$ is not stationary and so there is a closed unbounded set $I\subseteq \alpha$ with $X\cap I=\emptyset$. By conditions (a), (b) and (c) it follows that $(M^X_{\gamma})_{\gamma\in I}$ is a continuous $\preccurlyeq$-chain in $\mathbf{K}_{\kappa}$, and thus $\bigcup_{\gamma\in I}M^X_{\gamma}\in \mathbf{K}_{\kappa}$. Since $I$ is a club we have that $M^X_{\alpha}=\bigcup_{\gamma\in I}M^X_{\gamma}$, which proves our claim.
	
	\smallskip
	\noindent \underline{Condition~(b)}. For all $\alpha<\beta<\kappa^+$ and $\alpha\notin X$ we have that $M^X_{\alpha}\preccurlyeq M^X_{\beta}$ and, if additionally $\beta$ is successor, then $M^X_{\beta}$ is $(\kappa,\delta)$-$(\mathbf{K}, \preccurlyeq)$-limit over $ M^X_{\alpha}$.
	\newline Let $\alpha<\kappa^+$ with $\alpha\notin X$, we proceed inductively on $\beta$ for $\alpha<\beta<\kappa^+$.
	\newline If $\beta=\alpha+1$ then  $M^X_{\beta}$ was chosen exactly to be $(\kappa,\delta)$-$(\mathbf{K}, \preccurlyeq)$-limit over $ M^X_{\alpha}$ (which clearly entails also $M^X_{\alpha}\preccurlyeq M^X_{\beta}$). 
	\newline If $\beta$ is limit, then $M^X_{\beta}=\bigcup_{\alpha<\gamma<\beta} M^X_\gamma$ where $(M^X_\gamma)_{\alpha<\gamma<\beta}$ forms a $\preccurlyeq$-chain. It then follows immediately from the definition of weak $\mrm{AEC}$ (cf.~\ref{def:weak_aec}) that $M^X_{\alpha}\preccurlyeq M^X_{\beta}$. 
	\newline If $\beta=\gamma+1$ for some $\gamma\neq \alpha$ with $\gamma\notin X$, then $M^X_{\beta}$ is a $(\kappa,\delta)$-$(\mathbf{K}, \preccurlyeq)$-limit model over $M^X_{\gamma}$. By the inductive hypothesis and using the fact that $\beta$ is successor, it follows that the chain $(M^X_i)_{\alpha\leq i\leq \beta}$ satisfies the assumptions from Corollary~\ref{prop:limitoflimit}, and thus $M^X_{\beta}$ is a $(\kappa,\delta)$-$(\mathbf{K}, \preccurlyeq)$-limit model over $M^X_{\alpha}$.
	\newline If $\beta=\gamma+1$ for some $\gamma\neq \alpha$ with $\gamma\in X$, then by construction we have that $M^X_{\beta}$ is a $(\kappa,\delta)$-$(\mathbf{K}, \preccurlyeq)$-limit over $M^X_{\gamma_i}$ for some $\alpha<\gamma_i<\gamma$. By reasoning as in the case with $\gamma\notin X$ it follows from Corollary~\ref{prop:limitoflimit} that $M^X_{\beta}$ is a $(\kappa,\delta)$-$(\mathbf{K}, \preccurlyeq)$-limit over $M^X_{\alpha}$.
	
	\smallskip
	\noindent \underline{Condition~(c)}. For limits $\alpha<\kappa^+$ we have that $M^X_\alpha=\bigcup_{\beta<\alpha}M^X_\beta$.
	\newline Immediate by construction.
	
	\smallskip
	\noindent \underline{Condition~(d)}. $M^X_\alpha\not\precleq M^X_{\alpha+1}$ for all $\alpha\in X$.
	\newline This follows immediately by the construction of $M^X_{\alpha+1}$ for $\alpha\in X$.

	\medskip 
	\noindent Let $M^X=\bigcup_{\alpha<\kappa^+}M^X_\alpha$. We show that $M^X\not\cong F(\kappa^+)$ but $M^X\equiv_{\infty,\kappa^+}F(\kappa^+)$. We first prove that $M^X\not\cong F(\kappa^+)$. By Proposition~\ref{free_filtration_kappa} we can (without loss of generality) let $F(\kappa^+)=\bigcup_{i<\kappa^+}D_i$, with $D_0\cong A_0$ and $D_{\beta}$ a $(\kappa,\delta)$-$(\mathbf{K}, \preccurlyeq)$-limit over $D_\alpha$ for all $\alpha<\beta<\kappa^+$ with $\beta$ successor. Suppose now that there is an isomorphism $h:M^X\to F(\kappa^+)$ and denote by $h_\alpha$ its restriction to $M^X_\alpha$. Consider the set:
	\[C_0=\{\alpha<\kappa^+ : h_\alpha \text{ is onto } D_\alpha\}.\]
	By a standard argument, one can see that $C_0$ is a club, and thus since  $X$ is stationary it follows that there is some limit ordinal $\alpha<\kappa^+$ such that $\alpha\in X\cap C_0$. Thus there is an isomorphism $h_\alpha:M^X_\alpha\to D_\alpha$ for some $\alpha\in X$. By Condition~(d) we have $M^X_{\alpha}\not \precleq M^X_{\alpha+1} $. Additionally, it is also easy to see that the following set is a club:
	\[C_1=\{\beta<\kappa^+ : \alpha<\beta \text{ and } h_\beta \text{ is onto } D_\beta\}.\]	
	Thus we can find $h_\beta\supseteq h_\alpha$ such that $h_{\beta}$ is onto $D_\beta$. Then we obtain by Condition~(b) that $ M^X_{\alpha+1} \preccurlyeq M^X_{\beta}$. Since $M^X_{\alpha}\not \precleq M^X_{\alpha+1} $, and $| M^X_\beta|= |M^X_\alpha|$ it follows from Definition~\ref{def:coherentisation} that  $M^X_{\alpha}\not \preccurlyeq M^X_\beta$. On the other hand, we have immediately from the choice of the models $(D_i)_{i<\kappa^+}$ that $D_{\alpha}\preccurlyeq D_\beta$. But since $h_\alpha:M^X_\alpha\cong D_\alpha$ and $h_\beta:M^X_\beta\cong D_\beta$, this contradicts the existence of the isomorphism $h:M^X\cong F(\kappa^+)$.
	
	\smallskip
	\noindent Finally, we show that $M^X\equiv_{\infty,\kappa^+}F(\kappa^+)$. By Theorem~\ref{thm:kueker} it suffices to show that $M^X$ is strongly $\kappa^+$-free. Let $\mathcal{S}=\{M^X_\alpha : \alpha\notin X  \}$, we show this satisfies Definition~\ref{def:strong_free}(2). Let $M^X_\alpha\in \mathcal{S}$ and consider $C\subseteq M^X$ such that $|C|\leq \kappa$. Since $\kappa^+$ is regular we can find some $\beta<\kappa^+$ such that $\beta$ is successor, $\alpha<\beta$, and $M^X_\alpha\cup C\subseteq M^X_\beta$. Then we have that $\beta\notin X$ and so $M^X_\beta\in \mathcal{S}$. Moreover, it follows immediately from Condition~(b) that $M^X_\beta$ is a $(\kappa,\delta)$-$(\mathbf{K}, \preccurlyeq)$-limit over $M^X_\alpha$. This shows that $M^X$ is strongly $\kappa^+$-free and so that $M^X\equiv_{\infty,\kappa^+}F(\kappa^+)$.
\end{proof}

\begin{remark}
	It is well-known that in the context of $\mrm{AEC}$s the assumption that the $\preccurlyeq$-chain $(A_i)_{i<\delta}$ from Conditions~\ref{def_AEC}(4.1) and \ref{def_AEC}(4.2) is continuous is redundant. As a matter of fact, \ref{def_AEC}(4.1) and \ref{def_AEC}(4.2) could be further weakened so that instead of a continuous $\preccurlyeq$-chain $(A_i)_{i<\delta}$ one considers an arbitrary directed system $(A_i)_{i\in I}$. Interestingly, this does not hold in the setting of weak $\mrm{AEC}$s where the Smoothness Axiom (cf.~\ref{def_AEC}(4.3)) may fail. In fact, in the proof of the previous theorem, the family $(M^X_\alpha)_{\alpha\notin X}$ is a $\preccurlyeq$-chain with $\bigcup_{\alpha\notin X}M^X_\alpha=M^X\notin \mathbf{K}$. We can re-enumerate this family so to make it a $\preccurlyeq$-chain, but the result will not be continuous. This shows that in the setting without Smoothness the Conditions~\ref{def_AEC}(4.1) and \ref{def_AEC}(4.2) of $ \mrm{AEC}$s may fail with respect to arbitrary directed systems.
\end{remark}

\begin{remark}
	As noticed in Remark~\ref{remark:many_stationary}, from $E^+(\kappa^+)$ and using Solovay's Theorem one can build $2^{\kappa^+}$-many different non-reflecting stationary subsets of $\kappa^+$ with elements of cofinality $\lambda$. Then, reasoning as in the previous theorem, one obtains for each pair of different stationary sets $X\neq Y$ two corresponding $\mathfrak{L}_{\infty,\kappa^+}$-free structures $M^X\not \cong M^Y$ of size $\kappa^+$. This shows that under $E^+(\kappa^+)$ there are $2^{\kappa^+}$-many non-isomorphic $\mathfrak{L}_{\infty,\kappa^+}$-free structures of size $\kappa^+$ with respect to a given weak $\mrm{AEC}$ $(\mathbf{K},\preccurlyeq)$ satisfying \ref{context:strong_categoricity} and \ref{CP-AEC:kappa_edit}. \emph{A fortiori}, and always under the set-theoretic assumption $E^+(\kappa^+)$, it follows that if ${(\mathbf{K}',\preccurlyeq')}$ is an $\mrm{AEC}$ which satisfies:
	\begin{enumerate}[(i)]
	\item $\mathbf{K}_\kappa\subseteq \mathbf{K}_\kappa'$;
	\item $A\preccurlyeq' B$ if and only if $A\preccurlyeq B$ for $A,B\in \mathbf{K}$;\end{enumerate}
then $\mathbf{K}'$ has $2^{\kappa^+}$-many non-isomorphic models of size $\kappa^+$. Clearly, for $\kappa^+=\aleph_1$ the previous statement holds in ZFC.
\end{remark}	

\section{The Construction Principle $\mrm{CP}(\mathbf{K},\ast)$}\label{CP:coproduct_case}

We provide in this section a more straightforward generalisation of the Eklof-Mekler-Shelah Construction Principle in the context of weak $ \mrm{AEC}$s, and we derive from it the conditions from \ref{context:strong_categoricity} and \ref{CP-AEC:kappa_edit} (and so also the consequences of such conditions, mainly \ref{thm:2}). While in Section~\ref{sec:kappa_cat} we isolated the technical notions needed to prove Theorem~\ref{thm:2}, we next focus on weak $ \mrm{AEC}$s which have a notion of canonical amalgam that \emph{defines} the strong submodel relation, as it is the case in all examples we shall consider in Section~\ref{Sec:application}. In particular, in this context we introduce the Construction Principle $\mrm{CP}(\mathbf{K},\ast)$ (Condition~\ref{CP-AEC:coproducts}(\hyperref[C4]{C4})), and we show that in this setting it entails $\mrm{CP}_{\kappa,\omega}(\mathbf{K},\sleq)$ (recall \ref{CP-AEC:kappa_edit}). In all concrete cases from Section~\ref{Sec:application}, we shall then directly verify that the class in question satisfies the Conditions~\ref{context:amalgamation_class}(\hyperref[C1]{C1})-(\hyperref[C3]{C3}), Condition~\ref{CP-AEC:coproducts}(\hyperref[C4]{C4}) and Conditions~\ref{assumptions:subsec_1}(\hyperref[C5]{C5})-(\hyperref[C6]{C6}) from this section (and thus in particular  $\mrm{CP}(\mathbf{K},\ast)$). We introduce canonical amalgamation classes in Section~\ref{subsec:4.1} and relate them to the setting of strongly $\kappa^+$-categorical weak $ \mrm{AEC}$s in Section~\ref{subsec:4.2}.

\subsection{Canonical amalgamation classes}\label{subsec:4.1}

We introduce in this section the concepts of (infinitary) canonical joint embedding and canonical amalgamation. Our definition follows the notion of canonical amalgamation (from, e.g., \cite[Def.~7]{PaSh}), but with the the additional requirements that the operator $\ast$ is infinitary, associative, and that it respects the strong submodel relation~$\preccurlyeq$. Notice that, in the following definition, the requirement that the intersection of the structures $B_i$ is empty corresponds to the fact that $\bigast_{i\in I}B_i$ depends solely on the isomorphism type of the structures $B_i$  (as in the case of free factors), and not on the choice of a specific witness.

\begin{definition}\label{def:canonical_jep}
	Let $(\mathbf{K},\preccurlyeq)$ be an abstract class. We say that $(\mathbf{K},\preccurlyeq)$ has \emph{infinitary canonical joint embedding} if for all sets $I$ there is an operator $\bigast_{i\in I}B_i$ on $\mathbf{K}^I$ which is defined whenever $B_i\in\mathbf{K}$ for all $i\in I$, $\bigcap_{i\in I} B_i=\emptyset$ and which has the following properties:
	\begin{enumerate}[(a)]
		\item  $B_i\preccurlyeq \bigast_{i\in I}B_i$ for all $i\in I$;
		\item if $\bigast_{i\in I}B_i$ and $\bigast_{i\in I}C_i$ are both defined and $(f_i:B_i\cong C_i)_{i\in I}$ are isomorphisms, then there is an isomorphism $g: \bigast_{i\in I}B_i\to \bigast_{i\in I}C_i$ s.t.  $g\restriction B_i=f_i$ for all $i\in I$;
		\item if $\bigast_{i\in I}B_i$ and $\bigast_{i\in I}C_i$ are both defined and $B_i\preccurlyeq C_i$ for all $i\in I$, then $\bigast_{i\in I}B_i \preccurlyeq \bigast_{i\in I}C_i$;
		\item  the operator $\ast$ satisfies the infinitary associativity law, i.e., if $\bigast_{i\in I}B_i$ is defined then for all  $J\subseteq I$ we have that $\bigast_{i\in I}B_i=\bigast_{i\in J}B_i\ast \bigast_{i\in I\setminus J}B_i$.
	\end{enumerate}
\end{definition}
We use the previous notion to define the concept of a \emph{canonical amalgamation class}. Intuitively, the following definition captures the key properties of $(\mathbf{FV},\sleq)$ in the setting of weak $ \mrm{AEC}$s. However, as we shall see in Section~\ref{Sec:application}, one can find examples of weak $ \mrm{AEC}$s that satisfy $\mrm{CP}(\mathbf{K},\ast)$ and that are not of the form $(\mathbf{FV},\sleq)$, i.e., they do not consist of the free algebras of some variety $\mathbf{V}$ with the relation of free factor with a free complementary factor defined in \ref{definition V-free product} (cf. the examples in Sections~\ref{Application:groups}-\ref{Application:ngons}). Since in this section we work with weak $ \mrm{AEC}$s with canonical joint embedding, we use the symbol $\sleq$ instead of $\preccurlyeq$ to denote the strong submodel relation.

\begin{definition}\label{context:amalgamation_class}
	Let $(\mathbf{K},\sleq)$ be a weak $\mrm{AEC}$ with $\mathrm{LS}(\mathbf{K}, \sleq)=\mu$, then we say that $(\mathbf{K},\sleq)$ is a  \emph{canonical amalgamation class} if it satisfies the following properties:
	\begin{enumerate}[(C1)]
		\item\label{C1}  $(\mathbf{K},\sleq)$ has canonical joint embedding (\ref{def:canonical_jep});
		\item\label{C2} if $A,B\in \mathbf{K}$ and $A\sleq B$, then there is $C\in \mathbf{K}$ such that $B=A\ast C$;
		\item\label{C3} for any ordinal $\alpha$ we have that $\bigast_{i<\alpha} B_i=\bigcup_{\beta<\alpha}(\bigast_{i<\beta} B_i)$.
\end{enumerate}
\end{definition}

\begin{remark}\label{remark:canonical_amalgam}
	We explain the usage of the name ``canonical amalgamation'' in the previous definition. The key point is that, in virtue of  Condition~\ref{context:amalgamation_class}(\hyperref[C2]{C2}), the canonical joint embedding property already entails the canonical amalgamation property, i.e., one can extend the operator $\ast$ from Definition~\ref{def:canonical_jep} to an operator $\ast_A$ satisfying analogous properties. The key point is that, if $A\sleq B$ and $A\sleq C$, then by Condition~\ref{context:amalgamation_class}(\ref{C2}) we have that $B=A\ast D_0$ and $C=A\ast D_1$ for some $D_0,D_1\in \mathbf{K}$ and it is then immediate to verify that $A\ast D_0\ast D_1$ satisfies the universal properties of the canonical amalgam over $A$, meaning that $A\ast D_0\ast D_1\cong B_0\ast_A B_1$. As we shall not need this in the rest of the paper we leave to the reader to verify these facts.  Additionally, we point out that Condition~\ref{context:amalgamation_class}(\hyperref[C3]{C3}) could be regarded as a restricted version of Smoothness, which holds also in cases (as for instance that of free algebras, as we shall elaborate in Section~\ref{Application:varieties}) where the full version of the Smoothness Axiom fails. 
\end{remark}

In the context of canonical amalgamation classes (i.e., \ref{context:amalgamation_class}), we can isolate the following version of the Construction Principle $\mrm{CP}(\mathbf{K},\ast)$. As the reader can readily verify, this is an immediate abstraction of the classical Construction Principle by Eklof, Mekler and Shelah from Definition~\ref{definition CP}, and amounts essentially to a specific failure of Smoothness in $(\mathbf{K},\sleq)$. Crucially, this version of CP is easier to verify in practice than $\mrm{CP}^\delta_{\kappa,\lambda}(\mathbf{K},\preccurlyeq)$ --- this is the main motivation behind it.

\begin{definition}\label{CP-AEC:coproducts}
	Let $(\mathbf{K},\sleq)$ be a canonical amalgamation class  (i.e., $(\mathbf{K},\sleq)$ is a weak $\mrm{AEC}$ which  satisfies \ref{context:amalgamation_class}(\hyperref[C1]{C1})-(\hyperref[C3]{C3}))  with $\mathrm{LS}(\mathbf{K}, \sleq)=\mu$. The Construction Principle $\mrm{CP}(\mathbf{K},\ast)$ is the following statement:
	\begin{enumerate}[(C4)]
		\item\label{C4} there are models $A_i \in \mathbf{K}_{\mu}$, for  all $i<\omega$, and $B\in \mathbf{K}_{\mu}$ such that:
		\begin{enumerate}[(a)]
			\item for all $i<\omega$, $A_i\sleq A_{i+1}$ and $A_i\sleq B$,
			\item $\bigcup_{i<\omega}A_i \not \pleq B$ (cf.~\ref{def:coherentisation}).
		\end{enumerate}
	\end{enumerate}
\end{definition}

\subsection{Strongly categorical amalgamation classes}\label{subsec:4.2}

In this section, we relate canonical amalgamation classes to the strongly categorical weak $ \mrm{AEC}$s from Definition~\ref{context:strong_categoricity}. We fix the following assumptions.

\begin{context}\label{assumptions:subsec_1}
	We fix in this section a canonical amalgamation class $(\mathbf{K},\sleq)$  with $\mathrm{LS}(\mathbf{K}, \sleq)=\mu$ and satisfying $\mrm{CP}(\mathbf{K},\ast)$ from \ref{CP-AEC:coproducts}, i.e., $(\mathbf{K},\sleq)$ satisfies Conditions~\ref{context:amalgamation_class}(\hyperref[C1]{C1})-(\hyperref[C3]{C3}) and Condition~\ref{CP-AEC:coproducts}(\hyperref[C4]{C4}). Additionally, we require that $(\mathbf{K},\sleq)$ satisfies the two additional conditions:
	\begin{enumerate}[(C5)]
		\item[(C5)]\label{C5} $\mathbf{K}$ is $\kappa^+$-categorical for all $\kappa\geq \mu$;
		\item[(C6)]\label{C6} for every $\kappa\geq \mu$ there is a universal model $M_\kappa\in \mathbf{K}_{\kappa}$ (cf.~\ref{def:universality}).
	\end{enumerate}
\end{context}

We notice that Condition~\ref{assumptions:subsec_1}(\hyperref[C6]{C6}) is apparently weaker than Condition~\ref{context:strong_categoricity}(\ref{strong_categoricity:two}). However, in the setting of canonical amalgamation  classes, the former entails the latter. We establish this in the next lemma, where we also use this fact to build suitable limit models. We first define some auxiliary structures.
	
\begin{definition}\label{def:auxiliary_models}
	Let $\kappa\geq \mu$ and let $M_\kappa\in \mathbf{K}_\kappa$ be the universal model of size $\kappa$ from \ref{assumptions:subsec_1}(\hyperref[C6]{C6}). We define $N_\kappa=\bigast_{i<\omega}M^i_\kappa$, where each $M^i_\kappa$ is a disjoint copy of $M_\kappa$.
\end{definition}

\noindent Given a structure $B\in \mathbf{K}$, it is quite straightforward to see that $B\ast M_\kappa$ is $(\mathbf{K}, \sleq)$-universal over $B$. In the context of canonical amalgamation classes we actually have more, i.e., $B\ast M_\kappa$ is $(\mathbf{K}, \sleq)$-universal over all $A\sleq B$. We establish this in the following lemma and we explain how to obtain $(\kappa,\omega)$-$(\mathbf{K}, \sleq)$-limit models. Recall that $M_\kappa$ is the universal model of size $\kappa$ given by \ref{assumptions:subsec_1}(\hyperref[C6]{C6}) and that $N_\kappa$ was defined in \ref{def:auxiliary_models} above.
	
\begin{lemma}\label{lemma:universality}\label{prop:auxiliary_limits}
	Let $\kappa\geq \mu$ and let $A,B\in \mathbf{K}_\kappa$ with $A\sleq B$, then:
	\begin{enumerate}[(1)]
		\item $B\ast M_\kappa$ is $(\mathbf{K}, \sleq)$-universal over $A$;
		\item $B\ast N_\kappa$ is $(\kappa,\omega)$-$(\mathbf{K}, \sleq)$-limit over $A$ (cf.~\ref{def:limits}).
	\end{enumerate}
\end{lemma}
\begin{proof}
	We first prove (1).	Since $A\sleq B$, we have that $B=A\ast E$ for some $E\in \mathbf{K}$. Suppose that we have  $A\sleq C$ and $|C|=|A|=\kappa$.  By Condition~\ref{context:amalgamation_class}(\hyperref[C2]{C2}) there is some $D\in \mathbf{K}_{\kappa}  $ such that $C=A\ast D$.  Since $M_\kappa$ is universal, it follows that there is a strong embedding $f_0:D\to M_\kappa$. Then, by the universal property of the amalgam, we obtain a strong embedding $f_1:A\ast D\to A\ast M_\kappa$ such that $f_1\restriction A =\mrm{id}_A$ and $f_1\restriction D =f_0$. Then, again by the universal mapping property of the amalgam, we obtain a strong embedding
	\[ f_2:A\ast D\ast E\to A\ast M_\kappa\ast E\]
	such that $f_2\restriction A =\mrm{id}_A$, $f_2\restriction E =\mrm{id}_E$ and $f_2\restriction D=f_0$. Since $C=A\ast D\sleq A\ast D\ast E$ and $A\ast M_\kappa\ast E=B\ast M_\kappa$, this shows that $C$ strongly embeds to $B\ast M_\kappa$ over $A$. We conclude that  $B\ast M_\kappa$ is $(\mathbf{K}, \sleq)$-universal over $A$.
	
	\smallskip
	\noindent We prove (2). Let $(M^i_\kappa)_{i<\omega}$ be $\omega$-many disjoint copies of $M_\kappa$. By Clause (1) we have that $B\ast M^0_\kappa$ is $(\mathbf{K}, \sleq)$-universal over $A$ and, moreover, for every $n<\omega$ the model $B\ast  \bigast_{i\leq n+1}M^i_\kappa$ is $(\mathbf{K}, \sleq)$-universal over $B \ast \bigast_{i\leq n}M^i_\kappa$. Since $\bigcup_{n<\omega}(B\ast  \bigast_{i<\omega}M^i_\kappa)=B\ast N_\kappa$, we conclude that $B\ast N_\kappa$ is $(\kappa,\omega)$-$(\mathbf{K}, \sleq)$-limit over $A$.
\end{proof}

\begin{proposition}\label{prop_1:amalg_class}
	Let $(\mathbf{K},\sleq)$ be as in \ref{assumptions:subsec_1}, then $(\mathbf{K},\sleq)$ is a strongly $\kappa^+$-categorical weak $\mrm{AEC}$ for every $\kappa\geq \mu$ (cf. Definition~\ref{def:auxiliary_models}). 
\end{proposition}
\begin{proof}
	Clearly,  Condition~\ref{context:strong_categoricity}(\ref{strong_categoricity:one}) follows immediately from Condition~\ref{assumptions:subsec_1}(\hyperref[C5]{C5}). Condition~\ref{context:strong_categoricity}(\ref{strong_categoricity:three}) follows by noticing that for $\kappa\geq \aleph_0$ we have that if $A,B\in\mathbf{K}_\kappa$ are isomorphic disjoint models, then $A\ast B \in\mathbf{K}_\kappa$, $A\subsetneq A\ast B$ and $A\sleq A\ast B$. Condition~\ref{context:strong_categoricity}(\ref{strong_categoricity:two}) follows from Lemma~\ref{prop:auxiliary_limits}.
\end{proof}
	
The key rationale behind $\mrm{CP}(\mathbf{K},\ast)$ is that it entails  $\mrm{CP}_{\kappa,\omega}(\mathbf{K},\sleq)$ for all $\kappa\geq \mu$. In contrast to  $\mrm{CP}^\delta_{\kappa,\lambda}(\mathbf{K},\sleq)$, the sequence of models provided by Condition~\ref{CP-AEC:coproducts}(\hyperref[C4]{C4}) does not necessarily consist of limit models, making it simpler to find such configurations in concrete cases. Crucially, the existence of canonical amalgams allows one to lift the sequence $(A_i)_{i<\omega}$ and $B$ from $\mathbf{K}_{\mu}$ to obtain a sequence of limit models in \emph{any} cardinal $\kappa\geq \mu$. We establish this in the next proposition. 

\begin{lemma}\label{prop:build_large_CP}
	For every cardinal $\kappa\geq \mu$ there is a $\sleq$-chain $(A'_i)_{i<\omega}$ in $\mathbf{K}_\kappa$ and $B'\in\mathbf{K}_\kappa$ such that:
	\begin{enumerate}[(a)]
		\item $A'_{i+1}$ is a $(\kappa,\omega)$-$(\mathbf{K}, \sleq)$-limit model over $A'_i$ for all $i<\omega$;			
		\item  $B'$ is a $(\kappa,\omega)$-$(\mathbf{K}, \sleq)$-limit model over $A'_i$ for all $i<\omega$;
		\item $\bigcup_{i<\omega}A'_i \not \pleq B'$.
	\end{enumerate}
\end{lemma}
\begin{proof}
	Let  $(A_i)_{i<\omega}$ and $B$ be the structures of size $\mu$ given in the statement of $\mrm{CP}(\mathbf{K},\ast)$ (cf.~\ref{CP-AEC:coproducts}(\hyperref[C4]{C4})). Let $A=\bigcup_{i<\omega}A_i$ and recall that $A\not\pleq B$. We define $(A'_i)_{i<\omega}$ and $B'$ in $\mathbf{K}_{\kappa}$ so to satisfy the statement of the proposition.  Let $(N^i_\kappa)_{i<\omega}$ be a disjoint sequence of models all isomorphic to $N_\kappa$ (recall \ref{def:auxiliary_models}) and also disjoint from $B$ and each $A_i$. For every $n<\omega$ we let $A'_n=A_n \ast \bigast_{i\leq n}N^i_\kappa$, we let $A'=A \ast \bigast_{i<\omega}N^i_\kappa$ and $B'=B\ast \bigast_{i<\omega}N^i_\kappa$. 
	\newline First, notice that we have that $A'_n= A_{n} \ast \bigast_{i\leq n}N^i_\kappa$ and $A'_{n+1}= A_{n+1} \ast \bigast_{i\leq n+1}N^i_\kappa$. Since we have by assumption that $A_n\sleq A_{n+1}$ and clearly $\bigast_{i\leq n}N^i_\kappa\sleq \bigast_{i\leq n+1}N^i_\kappa$, it follows immediately from Definition~\ref{def:canonical_jep} that $A'_n\sleq A'_{n+1} $. Moreover, since we also have that $ A'_{n}\sleq A_{n+1} \ast \bigast_{i\leq n}N^i_\kappa$, we obtain from Lemma~\ref{prop:auxiliary_limits} that $A'_{n+1}$ is a $(\kappa,\omega)$-$(\mathbf{K}, \sleq)$-limit model over $A'_n$. Similarly, we have that $A'_n\sleq B\ast \bigast_{i\leq n}N^i_\kappa$, and so by Lemma~\ref{prop:auxiliary_limits} we have that $ B\ast \bigast_{i\leq n+1}N^i_\kappa $ is $(\mathbf{K}, \sleq)$-limit over $A'_n$. Also by Lemma~\ref{prop:auxiliary_limits}, we have that $ B\ast \bigast_{i\leq j+1}N^i_\kappa $ is $(\mathbf{K}, \sleq)$-limit over $ B\ast \bigast_{i\leq j}N^i_\kappa $ for all $j<\omega$. Since  $\mrm{cf}(\omega)=\omega$ it then follows from Corollary~\ref{prop:limitoflimit} that  $B'$ is also a $(\kappa,\omega)$-$(\mathbf{K}, \sleq)$-limit model over $A'_n$. 
	\newline Finally, it remains to show that $\bigcup_{i<\omega}A'_i \not \pleq  B'$. First, notice that $A'=A \ast \bigast_{i<\omega}N^i_\kappa=\bigcup_{i<\omega}A'_i $ and that by construction we have $A\sleq A'$ and $B\sleq B'$. If there is some $C\in \mathbf{K}$ such that $A'\sleq C$ and $B'\sleq C$, by transitivity we obtain that $A\sleq C$ and $B\sleq C$, contradicting our choice of the models $A$ and $B$.
\end{proof}

\begin{proposition}\label{prop_2:amalg_class}
	Let $(\mathbf{K},\sleq)$ be as in \ref{assumptions:subsec_1}, then for all $\kappa\geq \mu$ the class $(\mathbf{K},\sleq)$ is a strongly $\kappa^+$-categorical weak $\mrm{AEC}$ satisfying $\mrm{CP}_{\kappa,\omega}(\mathbf{K},\sleq)$.
\end{proposition}
\begin{proof}
	Immediate from  Proposition~\ref{prop_1:amalg_class} and Lemma~\ref{prop:build_large_CP}.
\end{proof}

Finally, putting together the results from this section and the previous one, we derive the following results.

\begin{theorem}[$\mrm{ZFC}$ + $V=L$]\label{theorem:CP}\label{main:theorem}
	Let $(\mathbf{K},\sleq)$ be a weak $ \mrm{AEC}$ satisfying the assumptions from \ref{assumptions:subsec_1}, then for all $\kappa\geq \mu$ there is an $\mathfrak{L}_{\infty,\kappa^+}$-free structure $M\notin \mathbf{K}$ of size $\kappa^+$. In particular, $\mathbf{K}$ is not  axiomatisable in $\mathfrak{L}_{\infty,\infty}$.
\end{theorem}
\begin{proof}
	The first part of the claim follows from Proposition~\ref{prop_2:amalg_class} together with Theorem~\ref{thm:2} and the fact that $V=L$ entails $E^+(\kappa^+)$ (see Fact~\ref{existence:stationary}). The second part of the statement follows immediately from the first. 
\end{proof}

\noindent  For the case when the L\"owenheim-Skolem number is $\aleph_0$, we additionally obtain the following non-axiomatisability result in $\mrm{ZFC}$.

\begin{corollary} \label{main:corollary}
Let $(\mathbf{K},\sleq)$ be as in \ref{assumptions:subsec_1} with $\mathrm{LS}(\mathbf{K}, \sleq)=\aleph_0$, then there is an $\mathfrak{L}_{\infty,\omega_1}$-free structure $M\notin \mathbf{K}$ of size $\aleph_1$. Thus, $\mathbf{K}$ is not  axiomatisable in $\mathfrak{L}_{\infty,\omega_1}$.
\end{corollary}
\begin{proof}
	Immediate from \ref{theorem:CP} and the fact that $E^+(\aleph_1)$ is a theorem of ZFC.
\end{proof}

\section{Applications of $\mrm{CP}(\mathbf{K},\ast)$}\label{Sec:application}

In the previous section we have introduced the principle $\mrm{CP}(\mathbf{K},\ast)$, which generalises the Construction Principle from Eklof, Mekler and Shelah to the abstract setting of weak $ \mrm{AEC}$s with canonical amalgamation. We show that such a generalisation does not carry only an intrinsic theoretical interest, but it also allows for novel non-axiomatisability results outside of the original scope of algebraic varieties.  On the one hand, we show in Section~\ref{Application:varieties} that in the class of free algebras $(\mathbf{FV},\sleq)$ the standard Construction Principle from \ref{definition CP} entails our principle $\mrm{CP}(\mathbf{K},\ast)$. On the other hand, we exhibit some examples of applications of $\mrm{CP}(\mathbf{K},\ast)$ which are not of the form $(\mathbf{FV},\sleq)$. This shows that $\mrm{CP}(\mathbf{K},\ast)$ is strictly wider in scope than the original Construction Principle, and that it makes possible to derive novel non-axiomatisability results. First, we consider in  Section~\ref{Application:groups} the class of free products of cyclic groups of some fixed order and in Section~\ref{Application:tfab} the class of direct sums of torsion-free abelian groups of rank~$1$ which are different from $\mathbb{Q}$. These classes provide examples of weak $\mrm{AEC}$s of algebras that do not form a variety, and thus do not directly fall under the scope of the original CP (we elaborate on this further in \ref{Application:groups} and \ref{Application:tfab}). Then, in Section~\ref{Application:steiner} and Section~\ref{Application:ngons} we study two applications of  $\mrm{CP}(\mathbf{K},\ast)$ in incidence geometry, i.e., we consider the classes of free $(k,n)$-Steiner systems and of free generalised $n$-gons. 

More precisely, we show in the following section that each of the classes $(\mathbf{K},\sleq)$ mentioned above satisfies Conditions~\ref{context:amalgamation_class}(\hyperref[C1]{C1})-(\hyperref[C3]{C3}), Condition~\ref{CP-AEC:coproducts}(\hyperref[C4]{C4}) and also Conditions~\ref{assumptions:subsec_1}(\hyperref[C5]{C5})-(\hyperref[C6]{C6}). We can thus apply our results from Section~\ref{CP:coproduct_case}, in particular Theorem~\ref{main:theorem} and Corollary~\ref{main:corollary}, to all these contexts. We thus derive, for each of the classes $\mathbf{K}$ above, the existence of an $\mathfrak{L}_{\infty,\omega_1}$-free structure $M\notin \mathbf{K}$  of size $\aleph_1$ and, under $V=L$, for every $\kappa\geq \aleph_0$, the existence of an $\mathfrak{L}_{\infty,\kappa^+}$-free structure $M\notin \mathbf{K}$ of size $\kappa^+$. It follows in $\mrm{ZFC}$ that all the classes mentioned above are not axiomatisable in $\mathfrak{L}_{\infty,\omega_1}$ and, under $V=L$, that they are not axiomatisable in $\mathfrak{L}_{\infty,\infty}$. Additionally, Theorem~\ref{prop:free_algebras} shows that our work subsumes the classical results of Eklof, Mekler and Shelah.

\subsection{Free algebras in varieties}\label{Application:varieties}

We show in this section that the classical version of the Construction Principle by Eklof, Mekler and Shelah entails $\mrm{CP}(\mathbf{K},\ast)$. Namely, we show that if $\mathbf{V}$ is a variety of algebras which satisfies the Construction Principle from Def. \ref{definition CP}, then its subclass of free algebras $\mathbf{FV}$ with the relation of free factor with a free complementary factor (cf.~\ref{definition V-free product}) satisfies Conditions~\ref{context:amalgamation_class}(\hyperref[C1]{C1})-(\hyperref[C3]{C3}), \ref{CP-AEC:coproducts}(\hyperref[C4]{C4}) and \ref{assumptions:subsec_1}(\hyperref[C5]{C5})-(\hyperref[C6]{C6}). We fix the context of the present section and refer the reader to \ref{definition V-free product} for the definition of the free product $\ast$ in varieties of algebras and for our notational conventions concerning free algebras. Also, recall that, by Remark~\ref{fact:free_product_varieties}, for any family of disjoint free algebras $A_i\in \mathbf{FV}$, for $i<\kappa$, we have a free product $\bigast_{i<\kappa}A_i\in \mathbf{FV}$.

\begin{context}\label{context:free_varieties}
	In this section we fix a variety of algebra $\mathbf{V}$ in a signature $L$ of size $\mu$ and we denote by $\mathbf{FV}$ its subclass of free $\mathbf{V}$-algebras. For $A,B\in \mathbf{FV}$ we write $A\sleq B$ if there is a free complementary factor $C\in \mathbf{FV}$ such that $B=A\ast B$.
\end{context}

\noindent We recall the notion of a basis in the setting of varieties of algebras.

\begin{notation}
	Let $A$ be any algebra and $X\subseteq A$, then we write $\langle X\rangle_A$ for the subalgebra generated by $X$ in $A$, i.e., the smallest subalgebra of $A$ containing $X$.
\end{notation}

\begin{definition}\label{def:basis}
	Let $A\in \mathbf{FV}$, then a \emph{basis} for $A$ is a set of \emph{free generators} $X_A\subseteq A$, i.e., elements $X_A\subseteq A$ such that $A=F(X_A)$ and $x\notin \langle X_A\setminus \{x\}\rangle_A$ for all $x\in X_A$.
\end{definition}

\begin{fact}\label{fact:bases}
	Every $A\in \mathbf{FV}$ has a basis $X_A$. Moreover, if $A,B\in \mathbf{FV}$ are disjoint, then $A \ast B=F(X_AX_B)$ for any two bases $X_A$ and $X_B$ of $A$ and $B$, respectively. If $(A_i)_{i<\kappa}$ is a disjoint family of algebras from $\mathbf{FV}$, then $\bigast_{i<\kappa}A_i=F(\bigcup_{i<\kappa}X_i)$, where each $X_i$ is a basis of $A_i$.
\end{fact}

We next establish that $(\mathbf{FV},\sleq)$ is a strongly categorical weak $\mrm{AEC}$ with canonical amalgamation. This relates the setting of algebraic varieties to our general context of Section~\ref{CP:coproduct_case}.

\begin{lemma}\label{prop:varieties_canon}
	Let  $\mathbf{V}$ be a variety of algebras in a language of size $\mu$, then  $(\mathbf{FV},\sleq)$ is a weak $\mrm{AEC}$  with $\mathrm{LS}(\mathbf{K}, \sleq)=\mu$ and it satisfies Conditions~\ref{context:amalgamation_class}(\hyperref[C1]{C1})-(\hyperref[C3]{C3}) and Conditions~\ref{assumptions:subsec_1}(\hyperref[C5]{C5})-(\hyperref[C6]{C6}).
\end{lemma}
\begin{proof}
	First, we verify that  $(\mathbf{FV},\sleq)$ is a weak $\mrm{AEC}$. It is straightforward to verify that $(\mathbf{FV},\sleq)$ is an abstract class, i.e., that it satisfies Conditions~\ref{def_AC}(1)-\ref{def_AC}(3).  We verify Condition~\ref{def_AEC}(4.1), Condition~\ref{def_AEC}(4.2) and Condition~\ref{def_AEC}(6). First, suppose that  $(A_i)_{i<\delta}$ is a continuous $\sleq$-chain of free algebras in $\mathbf{FV}$, then by \ref{fact:bases} we can find a continuous sequence of bases $X_i$ of $A_i$ such that $X_i\subseteq X_{j}$ whenever $i<j<\delta$. Let $X_\delta=\bigcup_{i<\delta}X_i$, then we have by Remark~\ref{fact:bases} that $\bigcup_{i<\delta}A_i=F(X_\delta)\in\mathbf{FV}$, which verifies Condition~\ref{def_AEC}(4.1). Additionally, since $X_i\subseteq X_\delta$, it follows immediately that $F(X_\delta)=A_i\ast F(X_\delta\setminus X_i)$ and so  $A_i\sleq \bigcup_{i<\delta}A_i$, which verifies  Condition~\ref{def_AEC}(4.2). Consider now Condition~\ref{def_AEC}(6). Let $A\in \mathbf{FV}$ and consider a subset $B\subseteq A$. If $X_A$ is a basis of $A$, then there is a subset $Y\subseteq X_A$ of size $\mu+|B|$ so that every element in $B$ is a term over $Y$. Then let $Z=X_A\setminus Y$, clearly we have that $A=F(YZ)=F(Y)\ast F(Z)$. Since $|F(Y)|\leq \mu+ |B|$, this shows that $(\mathbf{FV},\sleq)$ has L\"owenheim-Skolem number $\mu$. It follows that $(\mathbf{FV},\sleq)$ satisfies  Condition~\ref{def_AEC}(6) and thus is a weak $\mrm{AEC}$.
	
	\smallskip
	\noindent 	Consider now Conditions~\ref{context:amalgamation_class}(\hyperref[C1]{C1})-(\hyperref[C3]{C3}). Condition~\ref{context:amalgamation_class}(\hyperref[C1]{C1}) follows immediately from the fact that the free product of free algebras is also free, together with the definition of the strong submodel relation $\sleq$. From this it is immediate to verify the properties from Definition~\ref{def:canonical_jep}.  Condition~\ref{context:amalgamation_class}(\hyperref[C2]{C2}) follows from the definition of $\sleq$ (cf.~\ref{context:free_varieties}) and Condition~\ref{context:amalgamation_class}(\hyperref[C3]{C3}) is a straightforward consequence of the definition of $\ast$ together with Fact~\ref{fact:bases}.
	
	\smallskip
	\noindent It remains to check that $(\mathbf{FV},\sleq)$ also satisfies Conditions~\ref{assumptions:subsec_1}(\hyperref[C5]{C5})-(\hyperref[C6]{C6}). Condition~\ref{assumptions:subsec_1}(\hyperref[C5]{C5}) follows immediately from the fact that, for all $\kappa\geq \mu^+$, the algebra $F(\kappa)$ is the only structure in $\mathbf{FV}$ of size $\kappa$ (cf.~Remark~\ref{context:univ_algebra}). For Condition~\ref{assumptions:subsec_1}(\hyperref[C6]{C6}) we reason as follows. Let $\mu\leq \kappa\leq \lambda$ and consider two (disjoint) free algebras $F(\kappa)$ and $F(\lambda)$ with bases $X_\kappa$ and $X_\lambda$, respectively. Since $\mu\leq\kappa\leq \lambda$ we have  $F(X_\kappa X_\lambda)\cong F(\lambda)$ and by the definition of the free product we have that $F(X_\kappa X_\lambda)=F(X_\kappa)\ast F(X_\lambda)\in \mathbf{FV}$. This shows that, for all $\lambda\geq \mu$, the free algebra $F(\lambda)$ is universal.
\end{proof}

To connect the classical work on varieties to our results from Section~\ref{CP:coproduct_case}, it remains to show that under the Construction Principle (CP) from \ref{CP:mekler-shelah} one can obtain $\mrm{CP}(\mathbf{K},\ast)$, namely  Condition~\ref{CP-AEC:coproducts}(\hyperref[C4]{C4}). We establish this in the next Lemma.

\begin{lemma}\label{CP_entails_CPast}
	Let $\mathbf{V}$ be a variety in a language of size $\mu$ that satisfies the Construction Principle from \ref{definition CP}, then $(\mathbf{FV},\sleq)$ satisfies $\mrm{CP}(\mathbf{K},\ast)$, i.e., Condition~\ref{CP-AEC:coproducts}(\hyperref[C4]{C4}).
\end{lemma}
\begin{proof}
	Let $A$ and $B$ be the two $\mathbf{V}$-free algebras in $\mathbf{V}$ witnessing the Eklof-Mekler-Shelah Construction Principle from \ref{definition CP}. In particular, since the language of $\mathbf{V}$ has size $\mu$ and $A$, $B$ are countably generated, they have both size $\mu+\aleph_0=\mu$. Following \ref{definition CP}, we let $X_A=\{ a_i : i<\omega  \}$ be a basis of $A$ and for every $n<\omega$ we let $A_n=\langle (a_i)_{i\leq n}\rangle_B$. We show that $A,B$ and $(A_n)_{n<\omega}$ witness $\mrm{CP}(\mathbf{K},\ast)$. 
	
	\smallskip
	\noindent First, notice that by the fact that $(a_i)_{i<\omega}$ is a free basis of $A$, it follows immediately that $A_{n}\sleq A_{n+1}$ for all $n<\omega$, and so $(A_n)_{n<\omega}$ forms a $\sleq$-chain. By \ref{definition CP}(1) it also follows that $A_n\sleq B$ for all $n<\omega$. So Condition~\ref{CP-AEC:coproducts}(\hyperref[C4]{C4a}) holds. 
	
	\smallskip
	\noindent We next verify Condition~\ref{CP-AEC:coproducts}(\hyperref[C4]{C4b}). First, notice that by \ref{definition CP}(2), and the fact that the language of $\mathbf{V}$ has size $\mu$, we have that $A\not\sleq B\ast F(\mu) $. Suppose towards contradiction that there is some $C\in \mathbf{K}$ such that $A\sleq C$, $B\sleq C$ and $|C|=|A|+|B|+\mathrm{LS}(\mathbf{K}, \sleq)$. Notice that since the language of $\mathbf{V}$ is of size $\mu$, it follows immediately from Lemma~\ref{prop:varieties_canon} that $\mathrm{LS}(\mathbf{K}, \sleq)=\mu$. Since $B\sleq C$ it follows that $C=B\ast D$ for some $D\in \mathbf{K}_{\leq\mu}$. Since $|D|\leq \mu$, it follows that $D\ast F(\mu)\cong F(\mu)$, and so by \ref{def:canonical_jep}(b)  we have $B\ast D\ast F(\mu)\cong_B B\ast F(\mu)$. As $A\sleq B\ast D$ and $B\ast D\sleq B\ast D\ast F(\mu)$, it follows:
	\[A\sleq C= B\ast D\sleq B\ast D\ast F(\mu),\]
	and so, since $B\ast D\ast F(\mu)\cong_B B\ast F(\mu)$ and $A\leq B$, it follows from \ref{def_AC}(1) that  $A\sleq B\ast F(\mu)$. This contradicts \ref{definition CP}(2) and shows that the Construction Principle $\mrm{CP}(\mathbf{K},\ast)$ holds.
\end{proof}

From the previous lemmas, we finally recover the original non-axiomatisability results by Eklof, Mekler and Shelah in our abstract setting.

\begin{theorem}[$\mrm{ZFC}$ + $V=L$]\label{prop:free_algebras}
	Let $\mathbf{V}$ be a variety of algebras in a language of size $\mu$ and satisfying the classical Construction Principle from Definition~\ref{definition CP}, then for all $\kappa\geq \mu$ there is an $\mathfrak{L}_{\infty,\kappa^+}$-free algebra $M\notin \mathbf{FV}$ of size $\kappa^+$. In particular, $\mathbf{FV}$ is not axiomatisable in $\mathfrak{L}_{\infty,\infty}$.
\end{theorem}
\begin{proof}
	Immediate from Lemma~\ref{prop:varieties_canon}, Lemma~\ref{CP_entails_CPast} and Theorem~\ref{main:theorem}.
\end{proof}

\begin{corollary} 
	Let $\mathbf{V}$ be a variety of algebras in a countable language and satisfying the classical Construction Principle from Definition~\ref{definition CP}, then there is an $\mathfrak{L}_{\infty,\omega_1}$-free algebra $M\notin \mathbf{FV}$ of size $\aleph_1$. Thus, $\mathbf{FV}$ is not  axiomatisable in $\mathfrak{L}_{\infty,\omega_1}$.
\end{corollary}
\begin{proof}
	Immediate from Lemma~\ref{prop:varieties_canon}, Lemma~\ref{CP_entails_CPast} and Corollary~\ref{main:corollary}.
\end{proof}

\subsection{Free products of cyclic groups of fixed order}\label{Application:groups}

We consider in this section an algebraic application of $\mrm{CP}(\mathbf{K},\ast)$ which does not come from varieties of algebras, namely free products of cyclic groups of some fixed order. This shows that, even in the setting of algebras, the scope of $\mrm{CP}(\mathbf{K},\ast)$ is strictly larger than the one of the original Construction Principle. We fix the context of this section and recall some notation and terminology from group theory.

\begin{context}\label{context:free_cyclic_groups}
	We fix in this section an order $2\leq n\in\mathbb{N}\cup \{\infty\}$. We let $\mathbf{K}$ be the class of free products of any possible cardinality of cyclic groups of order $n$. Given $A,B\in \mathbf{K}$ we write $A\ast B$ for their free product, and we let $A\sleq B$ if there is some $C\in \mathbf{K}$ such that $B=A\ast C$.
\end{context}

\begin{fact}\label{fact:cyclic_basis}
	Suppose $A\in \mathbf{K}$, then there is a basis $X_A$ of $A$ in the sense of Definition~\ref{def:basis}. In particular, if $A=\bigast_{i<\kappa}A_i$ then a basis for $A$ consists of exactly one element of order $n$ from each $A_i$ with $i<\kappa$.
\end{fact}

\begin{notation}\label{normal_closure}
	Let $A\leq B$ be groups, then we write $\langle \langle A \rangle \rangle_B$ for the normal closure of $A$ in $B$, and we drop the index $B$ when the ambient group is clear.
\end{notation}

We show that $(\mathbf{K},\sleq)$ is a weak $\mrm{AEC}$ satisfying Conditions~\ref{context:amalgamation_class}(\hyperref[C1]{C1})-(\hyperref[C6]{C6}). 

%

\begin{lemma}\label{cyclic_grups:AEC}
	The class $(\mathbf{K},\sleq)$ is a weak $\mrm{AEC}$  with $\mathrm{LS}(\mathbf{K}, \sleq)=\aleph_0$ and it satisfies Conditions~\ref{context:amalgamation_class}(\hyperref[C1]{C1})-(\hyperref[C3]{C3}) and Conditions~\ref{assumptions:subsec_1}(\hyperref[C5]{C5})-(\hyperref[C6]{C6}).
\end{lemma}
\begin{proof}
	The fact that $(\mathbf{K},\sleq)$ is a weak $\mrm{AEC}$ with $\mathrm{LS}(\mathbf{K}, \sleq)=\aleph_0$ and that it satisfies Conditions~\ref{context:amalgamation_class}(\hyperref[C1]{C1})-(\hyperref[C3]{C3}) and  Conditions~\ref{assumptions:subsec_1}(\hyperref[C5]{C5})-(\hyperref[C6]{C6}) follows analogously as in \ref{prop:varieties_canon} by using the notion of basis from Fact~\ref{fact:cyclic_basis}. 
\end{proof}

\noindent Next, we show that $(\mathbf{K},\sleq)$  also satisfies Conditions~\ref{assumptions:subsec_1}(\hyperref[C4]{C4}) i.e., the Construction Principle $\mrm{CP}(\mathbf{K},\ast)$. 

\begin{lemma} \label{CP:cyclyc_groups}
	The class $(\mathbf{K},\sleq)$ satisfies $\mrm{CP}(\mathbf{K},\ast)$, i.e., Condition~\ref{CP-AEC:coproducts}(\hyperref[C4]{C4}).
\end{lemma}
\begin{proof}  
	We construct a configuration of structures $(A_i)_{i<\omega}$ and $B$ in $\mathbf{K}$ witnessing $\mrm{CP}(\mathbf{K},\ast)$. Let $B$ be the free product of size $\aleph_0$ of cyclic groups of order $n$ and let $B = \langle x_i : i < \omega \rangle_B$, so that $\{x_i : i < \omega \}$ is a basis of $B$. We distinguish two cases.
	\smallskip
	
	\noindent  \underline{{Case 1}}. $ n =2$.
	\newline\noindent   For every $i < \omega$, let $y_i = x_{i+1}x_i x_{i+1}$ and $A_i = \langle y_j : j \leq i \rangle_A$. Let $A = \bigcup_{i < \omega} A_i$. Then easily, for every $i < \omega$, $A_i \leq_* B$. We claim that $A \not\pleq B$. Suppose that this is the case, then $A \leq_* B \ast D := C$. So there is a Coxeter basis $S_A$ of $A$ and a Coxeter basis $S_C$ of $C$ such that $S_A \subseteq S_C$, since $C = B \ast C$ with $B$ and $C$ free Coxeter groups, we have that $S_C = S_B \cup S_D$, with $S_B,S_D$ being disjoint Coxeter generators of $B$ and $D$, respectively. Suppose that $S_A \not\subseteq S_B$. Then there is $s \in S_C \cap S_A$ such that $s \notin S_B$. On the other hand, $s \in S_A \subseteq A \leq B$ and furthermore $s \in R_C$, where $R_C = (S_C)^C$ --- namely, $R_C$ are the reflections of the Coxeter groups $C$. Since it is well-known that $R_C \cap B = R_B$ (cf.~\cite[Cor.~1.4]{Swi}), it follows that $s \in R_B$, where $R_B = (S_B)^B$ (recall also that the set of reflections in a Coxeter group does not depend on the choice of the generators, cf. \cite[3.65(2)]{MPS}). Thus $s \in S_C \setminus S_B$ is conjugate to a reflection $s' \in R_B = (S_B)^B$ with $s' \neq s$, but this is impossible in an even Coxeter group, such as one with no relations between distinct generators (cf.~\cite[Lemma 3.3.3]{davis}). Hence we conclude that $S_A \subseteq S_B$,  i.e., $A \leq_* B$, but it is easy to see that this is impossible. In fact, it suffices to show that $A$ is a proper subgroup of $B$. Why? Suppose that this is the case, then it cannot be the case that $A \leq_* B$ as for every conjugacy class of element of order $2$ in $B$, i.e., for every $x_i^B$, there is $a \in A$ in that conjugacy class (namely $y_i = x_{i+1}x_i x_{i+1} = x_{i+1}x_i x^{-1}_{i+1}$, recall that $n = 2$), and the only way to achieve $A \leq_* B$ for a proper subgroup $A$ of $B$ is to add conjugacy classes to $A$. Finally, to see that $A$ is a proper subgroup of $B$ it suffices to show that $x_0 \notin A$, and this is easy to see using normal forms, i.e., no word in the the $y_i$'s can give $x_0$ (notice that the $y_i$'s have order $2$).
	\smallskip
	
	\noindent   \underline{{Case 2}}. $ n > 2$.
	\newline \noindent  For every $i < \omega$, let $y_i = x_ix^{-2}_{i+1}$ and $A_i = \langle y_j : j \leq i \rangle_A$. Let $A = \bigcup_{i < \omega} A_i$. As in Case 1, it is easy to verify that $A_i \sleq B$ for every $i < \omega$. We claim that $A\not\pleq B$. Suppose $B\sleq D \in \mathbf{K}$, we argue that $A\not \sleq D$. Let $E = \langle \langle A \rangle \rangle_D$ (cf.~\ref{normal_closure})). Now, easily, for every $i < j < \omega$, $x_i/E \neq x_j/E$ and $D/E  \models x^2_{i+1} = x_i$. Suppose that $A\sleq  D$, then $D/E \in \mathbf{K}$, contradicting the fact that a free product of cyclic groups of order $n>2$ does not contain a non-trivial $2$-divisible element, i.e., an element $x_0$ such that for every $0 < i < \omega$ there is $z_i$ such that $z_i^{2^i} = x_0$.
\end{proof}

\begin{theorem}\label{prop:cyclic_groups}
	Let $(\mathbf{K}, \sleq)$ be the class of free products of cyclic groups of fixed order $2\leq n\in\mathbb{N}\cup \{\infty\}$, then there is an $\mathfrak{L}_{\infty,\omega_1}$-free structure $M\notin \mathbf{K}$ of size $\aleph_1$ and, under $V=L$, there is for every $\kappa\geq \aleph_0$ an $\mathfrak{L}_{\infty,\kappa^+}$-free structure $M\notin \mathbf{K}$ of size $\kappa^+$.
\end{theorem}
\begin{proof}
	It follows by Lemma~\ref{cyclic_grups:AEC}, Lemma~\ref{CP:cyclyc_groups}, Theorem~\ref{main:theorem} and Corollary~\ref{main:corollary}.
\end{proof}

\begin{corollary}
	Let $\mathbf{K}$ be the class of free products of cyclic groups of some fixed  order $2\leq n\in\mathbb{N}\cup \{\infty\}$, then $\mathbf{K}$ is not axiomatisable in $\mathfrak{L}_{\infty,\omega_1}$ and, under $V=L$, $\mathbf{K}$ is not axiomatisable in $\mathfrak{L}_{\infty,\infty}$. 
\end{corollary}

\subsection{Direct sums of torsion-free abelian groups of rank~$1$}\label{Application:tfab}

We consider in this section torsion-free abelian groups of rank~$1$ which are not isomorphic to the group of rationals $(\mathbb{Q},+)$. We start by recalling some definitions and basic facts from the theory of torsion-free abelian groups. We refer  the reader to \cite[\S 12.1]{Fuchs} for proofs of the following facts and further background.

\begin{fact}\label{remark:basis} Let $G$ be a torsion-free abelian group, then $G$ is a subgroup of a $\mathbb{Q}$-vector space $V$ whose basis is a maximal subset of independent vectors in $G$. The size of this basis is called the \emph{rank} of $G$.
\end{fact}

\noindent The key tools in the study of torsion-free abelian groups are provided by the notions of height, characteristic and type. We recall their definitions.

\begin{definition}\label{definitions:tfag} 	Let $G$ be a torsion-free abelian group.
	\begin{enumerate}[(a)]
		\item Given an element $a\in G$ and a prime number $p$, the \emph{p-height} $h_p(a)$ of $a$ is the greatest integer $k$ such that the equation $p^k x=a$ has solutions in $A$, if it exists, and it is $\infty$ otherwise.
		\item Let $p_1,p_2,\dots$ be an increasing enumeration of the prime numbers, then the sequence of $p$-heights $\chi(a) = (h_{p_1}(a),h_{p_2}(a),\dots )$ of an element $a\in G$ is called the \emph{characteristic~of~$a$.}
		\item Two characteristics $\chi(a)=(k_1,k_2,\dots,)$ and $\chi(b)=(\ell_1,\ell_2,\dots,)$ are said to be \emph{equivalent} if $k_n=\ell_n$ for almost all $n$ and such that both $k_n$ and $\ell_n$ are finite whenever $k_n\neq\ell_n$. We then write $\chi(a)\equiv \chi(b)$ and call \emph{types} the equivalence classes of characteristics. We write $\mathbf{t}(a)$ for the type of the characteristics $\chi(a)$.
	\end{enumerate}
\end{definition}

In this section we are interested in torsion-free abelian groups of \emph{rank~$1$}. By Fact~\ref{remark:basis} these are exactly the subgroups of $\mathbb{Q}$ and are thus often called \emph{rational groups}. We recall the following important fact (cf.~\cite[p.~411]{Fuchs}).

\begin{fact}\label{fact:tfag}
	Let $G$ be a torsion-free abelian group of rank~$1$, then all its non-zero elements have the same characteristic. We thus define the \emph{type $\mathbf{t}(G)$ of $G$} simply as the type of any of its non-zero elements.
\end{fact}

\noindent The notion of type in turn provides us with a classification of torsion-free abelian groups of rank~$1$. We recall the following theorem, originally proved by Baer in \cite{Baer} and refer the reader also to \cite[\S 12.1; Theorem~1.1]{Fuchs} for its proof.

\begin{fact}[Baer \cite{Baer}]\label{thm:Baer} Two torsion-free abelian groups of rank~$1$ are isomorphic exactly if they have the same type. Every type is realized by a subgroup of $(\mathbb{Q},+)$, i.e., for every type $\mathbf{t}$ there is a subgroup $G_{\mathbf{t}}\leq \mathbb{Q}$ such that $\mathbf{t}(G)=\mathbf{t}$.
\end{fact}

\begin{notation}
	Given a type $\mathbf{t}$ we write $G_{\mathbf{t}}$ for the unique (up to isomorphism) subgroup of $\mathbb{Q}$ with type $\mathbf{t}$.
\end{notation}

In this section we fix a specific type $\mathbf{t}$ such that $\mathbf{t}\neq \mathbf{t}(\mathbb{Q})$, i.e., every characteristic $\chi(G)=(\ell_1,\ell_2,\dots)$ of type $\mathbf{t}$ contains at least one number $\ell_i$ different from $\infty$. We study the class of all directs sums (of any possible cardinality) of the unique torsion-free group $G_{\mathbf{t}}$ with $\mathbf{t}(G_{\mathbf{t}})=\mathbf{t}$. Notice that, since the type of the characteristic $(0,0,0,\dots)$ is exactly the type of the integers $(\mathbb{Z},+)$, our approach generalises the work of Mekler \cite{mekler} on almost-free abelian groups to all torsion-free abelian groups of rank~$1$ with arbitrary type $\mathbf{t}\neq \mathbf{t}(\mathbb{Q})$. We fix the following context.

\begin{context}\label{context:tfag_groups}
	Let  $\mathbf{t}$ be any type such that $\mathbf{t}\neq \mathbf{t}(\mathbb{Q})$. In this section we let $\mathbf{K}_{\mathbf{t}}$ be the class of all direct sums of $G_{\mathbf{t}}$ and we let $A\dleq B$ for $A,B\in \mathbf{K}_{\mathbf{t}}$ if there is some $C\in \mathbf{K}_{\mathbf{t}}$ such that $B=A\oplus C$, where $\oplus$ denotes the usual operation of direct sum. Notice that $B=A\oplus C$ implies that $B/A=C\in\mathbf{K}_{\mathbf{t}}$ (cf.~\cite[p.~44]{Fuchs}). 
\end{context}

Before verifying that the class $(\mathbf{K}_{\mathbf{t}}, \dleq)$ satisfies Conditions~\ref{context:amalgamation_class}(\hyperref[C1]{C1})-(\hyperref[C3]{C6}), we introduce the following key definition, which is from \cite[p.~410]{Fuchs}.

\begin{definition}\label{def:type_closure}
	Let $\chi=(\ell_1,\ell_2,\dots)$ be some characteristic, then we let $R_{\chi}$ be the following subgroup of $\mathbb{Q}$:
	\begin{align*}
	R_{\chi}\coloneqq \Bigl \langle \Bigl \{p_n^{-k_n}    : k_n \leq \ell_n \text{ and }  n<\omega     \Bigr\}     \Bigr\rangle_{\mathbb{Q}} 
	\end{align*}
	where $p_n$ refers to the $n$-\emph{th} prime number. Similarly, we then write $R_\mathbf{\chi}(x)$ for the corresponding subgroup of $\mathbb{Q}(x)$, where $\mathbb{Q}(x)$ is the $\mathbb{Q}$-vector space with basis $\{x\}$.
\end{definition}

\begin{remark}\label{tfab:key_remark}
	Let $\chi=(\ell_1,\ell_2,\dots)$ be some characteristic, then in $R_{\chi}$ the element 1 has exactly characteristic $\chi(1)=\chi$ (cf.~\cite[p.~410]{Fuchs}). Moreover, since when viewed as a $\mathbb{Q}$-vector space, the basis of $R_{\chi}$ is $\{1\}$, it follows from Fact~\ref{fact:tfag} that all elements in $R_{\chi}$ have the same type of $\chi$. For this reason we often write $R_\mathbf{t}$ instead of $R_{\chi}$, where $\mathbf{t}=\mathbf{t}(1)$ in $R_\chi$. Additionally, it follows that if in a torsion-free group $A\leq \mathbb{Q}$ the element $a\in A$ has characteristic $\chi(a)=\chi$, then $R_\chi(a)=\langle a \rangle_A^*$, where $ \langle a \rangle_A^* $ refers to the \emph{pure closure} of $a$ in $A$ (cf.~\cite[p.~410]{Fuchs}).
\end{remark}

\begin{remark}\label{tfab:key_remark_1}
	We point out an important connection between torsion-free abelian groups of rank~$1$ and free modules. By \ref{thm:Baer} a torsion-free abelian group $G_{\mathbf{t}}$ of rank~$1$ and type $\mathbf{t}$ is a subgroup of $\mathbb{Q}$ generated by some element $a\in G_{\mathbf{t}}$ with type $\mathbf{t}$. It is possible to verify that $R_{\mathbf{t}}$ is actually a subring of $\mathbb{Q}$, whence from Definition~\ref{def:type_closure} and Remark~\ref{tfab:key_remark} it follows that $G_{\mathbf{t}}$ is a cyclic free $R_{\mathbf{t}}$-module. In turn, one can verify that direct sums of torsion-free groups of type $\mathbf{t}$ correspond to free products of $R_{\mathbf{t}}$-modules of rank~$1$.  This indicates an important connection between our approach in this section and the classical results on varieties of modules from \cite{EM2}. Crucially, however, $R_{\mathbf{t}}$-modules and torsion-free abelian groups differ in their signature, which is a key fact when dealing with their axiomatisability.  In particular, from the classical Construction Principle from Definition~\ref{definition CP} (and reasoning as in \ref{CP:tfag_groups} below) one can derive some non-axiomatisability results for free $R_{\mathbf{t}}$-modules. However, it is not in principle obvious how to transfer these results to additionally derive the non-axiomatisability of torsion-free abelian groups of rank~$1$. A key advantage of the Construction Principle  $\mrm{CP}(\mathbf{K},\ast)$ is to dispense with this dependency on the signature. We thus deal in the rest of this section directly with torsion-free abelian groups of rank~$1$ and show how to apply $\mrm{CP}(\mathbf{K},\ast)$ in this case.
\end{remark}

\begin{lemma}\label{tfag:AEC}
	The class $(\mathbf{K}_{\mathbf{t}},\dleq)$ is a weak $\mrm{AEC}$  with $\mathrm{LS}(\mathbf{K}_{\mathbf{t}}, \dleq)=\aleph_0$ and it satisfies Conditions~\ref{context:amalgamation_class}(\hyperref[C1]{C1})-(\hyperref[C3]{C3}) and Conditions~\ref{assumptions:subsec_1}(\hyperref[C5]{C5})-(\hyperref[C6]{C6}).
\end{lemma}
\begin{proof}
	This follows by reasoning as in Lemma~\ref{prop:varieties_canon}, using the closure operators from Definition~\ref{def:type_closure} and the notion of basis from \ref{remark:basis}.
\end{proof}

\begin{lemma} \label{CP:tfag_groups}
	The class $(\mathbf{K}_{\mathbf{t}},\dleq)$ satisfies $\mrm{CP}(\mathbf{K}_{\mathbf{t}},\ast)$, i.e., Condition~\ref{CP-AEC:coproducts}(\hyperref[C4]{C4}).
\end{lemma}
\begin{proof}
	We describe a configuration of groups $(A_i)_{i<\omega}$ and $B$ in $\mathbf{K}_{\mathbf{t}}$ witnessing the Construction Principle  $\mrm{CP}(\mathbf{K}_{\mathbf{t}},\ast)$. Let $\chi(G_{\mathbf{t}})=(\ell_1,\ell_2,\ell_3,\dots)$ be a characteristic in the equivalence class $\mathbf{t}$. Since by assumption we have that $\mathbf{t}\neq \mathbf{t}(\mathbb{Q})$, there is at least one number $\ell_n$ which is finite. By the definition of the equivalence relation from \ref{definitions:tfag}(c), we can assume without loss of generality that $\ell_n=0$.
	
	\smallskip
	\noindent Consider the group $R_{\mathbf{t}}\leq \mathbb{Q}$ from Definition~\ref{def:type_closure} and for every $i<\omega$ let $R_{\mathbf{t}}(x_i)$ be a group isomorphic to $R_{\mathbf{t}}$ under the map induced by $x_i\mapsto 1$ (cf.~\ref{def:type_closure}). Additionally, we assume $R_{\mathbf{t}}(x_i)$ and $R_{\mathbf{t}}(x_j)$ are disjoint whenever $i\neq j$. Clearly $R_{\mathbf{t}}(x_i)\in \mathbf{K}_{\mathbf{t}}$ for all $i<\omega$, and thus  $B=\bigoplus_{i<\omega}R_{\mathbf{t}}(x_i)$ also belongs to $ \mathbf{K}_{\mathbf{t}} $.  For all $i<\omega$, we define the term $y_i$ by letting:
	\[ y_i\coloneqq x_i- p_nx_{i+1},   \]
	where $p_n$ is exactly the \emph{$n$-th} prime number (and recall that we are assuming that $\ell_n = 0$). We then define, for $i < \omega$, $A_i= \bigoplus \{ R_{\mathbf{t}}(y_j) : j =0, ..., i\}$.  We show that the models $(A_i)_{i<\omega}$ and $B$ witness $\mrm{CP}(\mathbf{K}_{\mathbf{t}},\ast)$. First, notice that by construction we have that $A_{i+1}=A_i\oplus R_{\mathbf{t}}(y_{i+1})$ for all $i<\omega$, thus it follows that $A_i\dleq A_{i+1}$ and that $(A_i)_{i<\omega}$ forms a $\dleq$-chain. 
	
	\smallskip
	\noindent We claim that we also have $A_i\dleq B$, for all $i<\omega$. Let $C_i={\bigoplus \{R_{\mathbf{t}}(x_j) : i<j<\omega\} }$, then clearly $C_i\in \mathbf{K}_{\mathbf{t}}$ since it is the direct sum of all $R_{\mathbf{t}}(x_j)$ for $j<\omega$ with $i<j$. In order to see that $A_i\oplus C_i=B$ (so that $A_i\dleq B$) is suffices to observe that the map $R_{\mathbf{t}}(x_0) \mapsto R_{\mathbf{t}}(y_0)$, ..., $R_t(x_i) \mapsto R_t(y_i)$, $R_t(x_{i+1}) \mapsto R_{\mathbf{t}}(x_{i+1})$ is an automorphism of $\bigoplus \{R_{\mathbf{t}}x_j : j=0, ..., i+1\}$, as then obviously we have the following:
	$$B = \bigoplus \{R_{\mathbf{t}}(y_j) : j = 0, ..., i\} \oplus R_{\mathbf{t}}(x_{i+1}) \oplus C_{i+1} = \bigoplus \{R_{\mathbf{t}}(y_j) : j = 0, ..., i\} \oplus C_{i}.$$
	
	\smallskip
	\noindent Finally, it remains to show that $A\not\dleq^{\mrm{c}} B$. Suppose towards contradiction that $D=A\oplus C$, $D=B\oplus E$ for some $C,D,E\in \mathbf{K}_{\mathbf{t}}$ with $D$ of size $\leq |A|+|B|+\aleph_0$. As remarked in \ref{context:tfag_groups}, this entails that $D/A\cong C \in \mathbf{K}_{\mathbf{t}}$. We now claim that in $D/A$ the element $x_0+A$ has $p_n$-height $\infty$. First, since $x_0-p_n x_1 =y_0 $, it follows that in the quotient group $D/A$ we have that $x_0 +A=p_n(x_1+A)$. Suppose inductively that we have shown $x_0 +A=p^i_n(x_{i}+A)$ for some $i<\omega$, then since $x_{i}-p_n x_{i+1} =y_{i} $ it follows that $x_{i}+A=p_n(x_{i+1}+A)$. So from the inductive assumption we obtain that $x_0+A=p^{i+1}_n(x_{i+1}+A)$. This shows  that in $D/A$ the element $x_0+A$ has $p_n$-height $\infty$. Now, since $D/A=C\in \mathbf{K}_{\mathbf{t}}$, it follows that $C$ is a direct sum of torsion-free abelian groups of type $\mathbf{t}$, so in particular $R_{\mathbf{t}}(x_0+A)$ is a subgroup of $C$. It follows from \ref{fact:tfag} that $\mathbf{t}(x_0+A)=\mathbf{t}$ and so that $X_0+A$ has finite $p_n$-height, which contradicts the above. This shows $A\not \dleq D$ and thus completes our proof.
\end{proof}

\begin{theorem}\label{prop:tfag}
	Let $(\mathbf{K}_{\mathbf{t}}, \dleq)$ be as in \ref{context:tfag_groups}, then there is an $\mathfrak{L}_{\infty,\omega_1}$-free structure $M\notin \mathbf{K}_{\mathbf{t}}$  of size $\aleph_1$ and, under $V=L$, there is for every $\kappa\geq \aleph_0$ an $\mathfrak{L}_{\infty,\kappa^+}$-free structure $M\notin \mathbf{K}_{\mathbf{t}}$  of size $\kappa^+$ .
\end{theorem}
\begin{proof}
	By Lemma~\ref{tfag:AEC}, Lemma~\ref{CP:tfag_groups}, Theorem~\ref{main:theorem} and Corollary~\ref{main:corollary}.
\end{proof}

\begin{corollary}\label{corol:tbab}
	The class $\mathbf{K}_{\mathbf{t}}$ from \ref{context:tfag_groups} is not axiomatisable in $\mathfrak{L}_{\infty,\omega_1}$ and, under $V=L$, it is not axiomatisable in $\mathfrak{L}_{\infty,\infty}$. 
\end{corollary}

\subsection{Free $(k,n)$-Steiner systems}\label{Application:steiner}

In this section we provide a first example of an application of the construction principle $\mrm{CP}(\mathbf{K},\ast)$ outside of algebra. In particular, we show that $\mrm{CP}(\mathbf{K},\ast)$ holds with respect to the class of free $(k,n)$-Steiner systems with the relation $\sleq$ from Definition~\ref{def:steiner_sleq}. We briefly expand on our claim that they do not fall under the original scope of the Eklof-Mekler-Shelah Construction Principle. As a matter of fact, Steiner systems are often studied under the lenses of universal algebra by relating them to so-called \emph{Steiner quasigroups} (cf.~\cite[p.~101]{burris}). However, Steiner quasigroups correspond exclusively to Steiner triple systems, i.e., $(2,3)$-Steiner systems, and they capture  arbitrary free $(k,n)$-Steiner systems only modulo a suitable expansion of the signature (we refer the interested reader to \cite{baldwin_steiner} for details). Crucially, our Construction Principle $\mrm{CP}(\mathbf{K},\ast)$ makes it possible to directly approach free $(k,n)$-Steiner systems as incidence structures, as it is usually done in the literature (see e.g. \cite{funk}). Compared to other model-theoretic works on free incidence structures, the distinctive feature of our approach in this section and the next is that we focus \emph{exclusively} on free $(k,n)$-Steiner systems (and, respectively, on free generalised $n$-gons), rather than on the wider class of infinite \emph{open} $(k,n)$-Steiner systems (or infinite open generalised $n$-gons), which are the models elementarily equivalent to the free ones and which were investigated before in \cite{paolini&hyttinen,Ammer_tent,PQ}. We next recall some preliminary notions and facts from  \cite{PQ} and refer the reader to \cite{funk} for an introduction to free $(k,n)$-Steiner systems and, more generally, to free structures in incidence geometry.

\begin{definition}\label{def:steiner}
	Let $L$ be the two-sorted language where the symbols $p_0,p_1,\dots$ denote \emph{points}, the symbols $b_0,b_1,\dots$ denote \emph{blocks} and the symbol $p\, \vert \, b$ means that the point $p$ is incident with the block $b$. For any $2\leq k<n$ the theory of \emph{$(k, n)$-Steiner systems} consists of the following axioms:
	\begin{enumerate}[(1)]
		\item given $k$-many distinct points $p_0,\dots,p_{k-1}$ there is a unique block $b$ such that $p_i\, \vert \, b$ for all $i<k$;
		\item for every block $b$ there are exactly $n$-many points $p_0,\dots,p_{n-1}$ such that $p_i\, \vert \, b$ for all $i<n$. 
	\end{enumerate}
	The universal fragment of the theory of $(k, n)$-Steiner systems is axiomatised exactly by the formula saying that, whenever both $\bigwedge_{i<k} p_i\, \vert\,b$ and $\bigwedge_{i<k} p_i\, \vert\,b'$ hold, then $b=b'$. We refer to models of this universal theory as \emph{partial  $(k, n)$-Steiner systems}. 
\end{definition}

\begin{definition}\label{Steiner:free_completion}
	Let $A$ be a partial $(k, n)$-Steiner system. We let $A_0=A$ and for every $2i<\omega$ we define $A_{2i+1}$ and $A_{2i+2}$ as follows: 	
	\begin{enumerate}[(a)]
		\item for every sequence of $k$ points $p_1,\dots,p_{k}$ from $A_{2i}$ which do not have a common incident block in $A_{2i}$, we add a new block $b\in A_{2i+1}$ which is incident only with $p_1,\dots,p_{k}$;
		\item for every block $b\in A_{2i+1}$ which is incident only to $\ell<n$ points from $A_{2i+1}$, we add $n-\ell$ points $p_\ell,\dots,p_{n-1}\in A_{2i+2}$ incident only with $b$.
	\end{enumerate}	
	Then the structure $F(A) := \bigcup_{i < \omega}A_i$ is called the \emph{free completion} of $A$ and we say that $F(A)$ is freely generated over $A$. We say that $A$ is \emph{non-degenerate} if $F(A)$ is infinite, and \emph{degenerate} otherwise. 
\end{definition}

\begin{definition}\label{Steiner:free_models}
	We say that a $(k,n)$-Steiner system $B$ is \emph{free} if it is the free completion $F(X)$ of some set of points $X$. We refer to $X$ as the \emph{basis} of $F(X)$.
\end{definition}

\begin{remark}\label{issue:freeness}
	We notice that our definition of free $(k,n)$-Steiner systems is slightly different from the definition of free $(k,n)$-Steiner systems from e.g. \cite{funk}, as we restrict our focus to free completion processes starting from a basis consisting of points only. The reason for this choice is simply to make immediately evident the uncountable categoricity of the class of free $(k,n)$-Steiner systems, and avoid obfuscating the main results from this section with non-essential combinatorial technicalities. We delay to Section~\ref{sec:remark_freeness} the problem of the ``right'' definition of free $(k,n)$-Steiner systems, where we further justify our Definition~\ref{Steiner:free_models}.
\end{remark}

We recall the notion of $\mrm{HF}$-order, which provides a key technical tool to study the first-order theory of free $(k, n)$-Steiner systems. We refer the reader to \cite{PQ} and \cite{paolini&hyttinen} for an introduction to $\mrm{HF}$-orders and pointers to the literature.

\begin{notation}\label{notation:downset}
	Let $<$ be a partial order on some set $X$ and let $Y\subseteq X$, we write $Y^\downarrow$ for the downset generated in $X$ by $Y$, i.e., $Y=\{x\in X : \exists y\in Y \text{ s.t. } x<y  \}$. If $Y=\{y\}$ we simply write $y^\downarrow$ instead of $\{y\}^\downarrow$.
\end{notation}

\begin{definition}\label{Steiner:def_hyperfree} \label{def:hf_order}
	For partial $(k, n)$-Steiner systems $A\subseteq B$, we define the following.
	\begin{enumerate}[(1)]
		\item An element $a\in A$ is \emph{hyperfree} in $B$ if it satisfies one of the following conditions:
		\begin{enumerate}[(i)]
			\item $a$ is a point incident with at most one block from $B$;
			\item $a$ is a block incident with at most $k$-many points from $B$.
		\end{enumerate}
		\item  We write $A\hleq B$ if there is a linear order $<$ of $B\setminus A$ such that every $c\in B\setminus A$ is hyperfree in $A\cup c^\downarrow$. If the linear order $<$ on $B\setminus A$ witnesses the fact that $A\hleq B$ then we say that $<$ is a \emph{$\mrm{HF}$-order over $A$}.
		\item  We say that $B$ is \emph{open} if $\emptyset \hleq B$.
	\end{enumerate}
\end{definition}

We recall two useful facts. Fact (1) follows immediately from the previous definitions, Fact (2) was proved in \cite[Prop.~3.51]{PQ}.

\begin{fact}\label{lemma:useful_facts}\label{lemma:useful_facts_2}
	The following facts hold:
	\begin{enumerate}[(1)]
		\item if $A=F(X_A)$, then there is a wellfounded $\mrm{HF}$-order of $A$ over $X_A$;
		\item let $A$ be any infinite free $(k,n)$-Steiner system and let $C\subseteq A$ with $C\hleq A$, then $F(C)\cong_C \mrm{acl}_A(C)\hleq A$.
	\end{enumerate}
\end{fact}

\begin{notation}
	Given the previous fact, we often slightly abuse the notation and identify $F(C)$ with $\mrm{acl}_A(C)$ whenever $C\hleq A$ for some infinite free $(k,n)$-Steiner system $A$. We often write $\mrm{acl}(C)$ instead of $\mrm{acl}_A(C)$, since all infinite free $(k,n)$-Steiner systems are elementary equivalent and $F(C)\preccurlyeq_{\omega,\omega} M$ whenever $C\hleq M$ and $M$ is elementary equivalent to a free $(k,n)$-Steiner system (cf.~\cite[Thm. 4.12]{PQ}).
\end{notation}

In \cite{PQ} we studied the first-order theory of all infinite free $(k,n)$-Steiner systems. We proved that it is complete, strictly stable, and we showed that the relation of $\mrm{HF}$-ordering $A\hleq B$ is equivalent to the relation of elementary embedding $A\preccurlyeq_{\omega,\omega} B$. Here we take another direction and focus on the class of free $(k,n)$-Steiner systems itself, rather than on the larger class of models of their first-order theory. We introduce the following strong submodel relation $\sleq$.

\begin{definition}\label{def:steiner_sleq}
	Let $A$ and $B$ be free $(k,n)$-Steiner systems, then we write $A\sleq B$ if there are bases $X_B$ of $B$ and $X_A$ of $A$ such that $X_A\subseteq X_B$, i.e., $A=F(X_A)$ and $B=F(X_B)$ (cf.~\ref{Steiner:free_completion}).
\end{definition}

\begin{context}\label{context:steiner}
	In this section, we fix the class $(\mathbf{K},\sleq)$ where  $\mathbf{K}$ is the class of all free $(k,n)$-Steiner systems and $\sleq$ is the relation from Definition~\ref{def:steiner_sleq}. 
\end{context}

\begin{remark}\label{remark:sleq_forking}
	We briefly mention --- without proof --- a more model-theoretic characterisation of the relation $\sleq$ (with respect to infinite models in $\mathbf{K}$). Let $\ind$ denote the relation of forking independence in first-order logic, let $A,B\in \mathbf{K}_{\geq\aleph_0}$ and let $X_A$, $X_B$ be respectively a basis of $A$ and $B$. Let $X_C=X_B\setminus A$, then by the characterisation of forking independence from \cite{PQ} we have that
	\begin{align*}
	A\sleq B \; \Longleftrightarrow \; A\hleq B \text { and } X_C\ind X_A.
	\end{align*}
\end{remark}

We proceed as in the previous sections and show that $(\mathbf{K},\sleq)$ fits the framework of canonical amalgamation classes from Section~\ref{Sec:application}. First, we define canonical amalgamation for free $(k,n)$-Steiner systems and then we show that $(\mathbf{K},\sleq)$ is a weak $\mrm{AEC}$ that satisfies Conditions~\ref{context:amalgamation_class}(\hyperref[C1]{C1})-(\hyperref[C3]{C3}) and Conditions~\ref{assumptions:subsec_1}(\hyperref[C5]{C5})-(\hyperref[C6]{C6}).

\begin{definition}\label{steiner:free_products}
	Let $A,B\in \mathbf{K}$ be disjoint free $(k,n)$-Steiner systems generated respectively by $X_A$ and $X_B$, then we let $A\ast B=F(X_AX_B)$. Similarly, if $B_i\in \mathbf{K}$ is a disjoint family of free $(k,n)$-Steiner systems with bases $(X_i)_{i\in I}$, then we let $\bigast_{i\in I} B_i=F(\bigcup_{i\in I}X_i)$.
\end{definition}

\begin{lemma}\label{steiner:AEC}
	The class $(\mathbf{K},\sleq)$ is a weak $\mrm{AEC}$ with $\mathrm{LS}(\mathbf{K}, \sleq)=\aleph_0$ and it satisfies Conditions~\ref{context:amalgamation_class}(\hyperref[C1]{C1})-(\hyperref[C3]{C3}) and Conditions~\ref{assumptions:subsec_1}(\hyperref[C5]{C5})-(\hyperref[C6]{C6}).
\end{lemma}
\begin{proof}
	We first prove that $(\mathbf{K},\sleq)$ is a weak $\mrm{AEC}$ with $\mathrm{LS}(\mathbf{K}, \sleq)=\aleph_0$. Clearly, $(\mathbf{K},\sleq)$ is an abstract class. We show that  $\mathrm{LS}(\mathbf{K}, \sleq)=\aleph_0$. Suppose that $A=F(X_A)\in \mathbf{K}$ with $|X_A|=\kappa$ and let $<$ be the wellfounded $\mrm{HF}$-order of $A$ (this exists by Lemma~\ref{lemma:useful_facts}). Given $B\subseteq A$ consider the set $X=B^\downarrow\cap X_A$ of elements $c\in X_A$ occurring before some elements of $B$ in the order $<$ (cf.~\ref{notation:downset}). Since $<$ is wellfounded we have that $|X|\leq |B|+\aleph_0$, and since $X\subseteq X_A$ it follows immediately from \ref{def:steiner_sleq} that $F(X)\sleq F(X)\ast F(X_A\setminus X)=A$. Given that $|F(X)|\leq |B|+\aleph_0$ this shows that $\mathrm{LS}(\mathbf{K}, \sleq)=\aleph_0$. 
	\newline We next verify Conditions~\ref{def_AEC}(4.1) and \ref{def_AEC}(4.2). Consider a continuous $\sleq$-chain $(A_i)_{i<\lambda}$ and let $X_i$ be a basis of $A_i$ for all $i<\lambda$. By Definition~\ref{def:steiner_sleq}  we can assume that $X_{i}\subseteq X_j$ whenever $i<j<\lambda$ and, clearly, $X_\alpha=\bigcup_{i<\alpha}X_i$ whenever $\alpha<\lambda$ is limit. Let $X_\lambda=\bigcap_{i<\lambda}X_i$, then it follows that $\bigcup_{i<\lambda}A_i=F(\bigcup_{i<\lambda}X_i)\in \mathbf{K}$, which verifies Condition~\ref{def_AEC}(4.1). Also, for all $i<\lambda$ we have that $A_i=F(X_i)\sleq F(X_i)\ast F(X_\lambda\setminus X_i)=\bigcup_{i<\lambda}A_i$, additionally showing that  Condition~\ref{def_AEC}(4.2) holds as well. Thus $(\mathbf{K},\sleq)$ is a weak $\mrm{AEC}$.
	
	\smallskip
	\noindent Consider Conditions~\ref{context:amalgamation_class}(\hyperref[C1]{C1})-(\hyperref[C3]{C3}).  Condition~\ref{context:amalgamation_class}(\hyperref[C1]{C1}) follows by Definition~\ref{steiner:free_products} and Condition~\ref{context:amalgamation_class}(\hyperref[C3]{C3}) follows by reasoning as in Lemma~\ref{prop:varieties_canon}. For  Condition~\ref{context:amalgamation_class}(\hyperref[C2]{C2}) notice that if $A\sleq B$, then there are bases $X_B$ of $B$ and $X_A$ of $A$ such that $X_A\subseteq X_B$, Let $X_C=X_B\setminus X_A$, then $B=F(X_AX_C)=F(X_A)\ast F(X_C)$, which verifies Condition~\ref{context:amalgamation_class}(\hyperref[C2]{C2}).  
	
	\smallskip
	\noindent Finally, Condition~\ref{assumptions:subsec_1}(\hyperref[C5]{C5}) follows immediately from \ref{Steiner:free_models}, i.e., for any $\kappa\geq \aleph_1$ we  have by definition that $F(X_\kappa)$ is the only free $(k,n)$-Steiner system, up to isomorphism. Consider Condition~\ref{assumptions:subsec_1}(\hyperref[C6]{C6}) with respect to some $\kappa\geq \aleph_0$. If $\lambda<\kappa$ then clearly $X_\lambda\subseteq X_\kappa$, where $X_\lambda$ and $X_\kappa$ are two sets of points respectively of size $\lambda$ and $\kappa$. It follows that $F(X_\lambda)\sleq F(X_\kappa)$ and thus that $F(X_\kappa)$ is $\sleq$-universal. This concludes our proof.
\end{proof}

Finally, we exhibit a configuration witnessing the fact that $(\mathbf{K},\sleq)$ satisfies the Construction Principle $\mrm{CP}(\mathbf{K},\ast)$. As in the former cases, this lemma displays that $(\mathbf{K},\sleq)$ does not satisfies Smoothness.

\begin{lemma}\label{CP:steiner}
	Let $(\mathbf{K},\sleq)$ be the class of free $(k,n)$-Steiner systems with the relation $\sleq$ from \ref{def:steiner_sleq}, then $(\mathbf{K},\sleq)$ satisfies $\mrm{CP}(\mathbf{K},\ast)$, i.e., Condition~\ref{CP-AEC:coproducts}(\hyperref[C4]{C4}).
\end{lemma}
\begin{proof}
	We exhibit models $(A_i)_{i<\omega}$ and $B$ from $\mathbf{K}$ witnessing Condition~\ref{CP-AEC:coproducts}(\hyperref[C4]{C4}). The main strategy in the following proof is to define a free $(k,n)$-Steiner systems $B$ and a strong substructure $A\sleq B$ so that every $D\in \mathbf{K}$ with $B\sleq D$ cannot be obtained from $A$ by means of a wellfounded $\mrm{HF}$-order. We do this by carefully selecting the generators of $A$ so that they belong to a special combinatorial configuration that we outline in the table below.  We start by letting $X_A=\{a_i : i<\omega \}$ and $X_P=\{p_i :i<\omega \}$ be two countable sets of points. We define $B\coloneqq F(X_AX_P)$ and for every $\ell<\omega$ we consider a combinatorial configuration in $B$. We illustrate this by means of the incidence table below.
	
	\smallskip
	\noindent We briefly explain the configuration described in the table below. Every element in the table must be considered as a ``word'' in $B$,  exactly as elements of a free group can be thought of as words over some set of generators. In particular, we are working over the generators $X_AX_P$ and defining the other elements in terms of them. Specifically, we notice that the element $b^7_\ell$ is the unique block in $B$ determined by $a_3,a_4,\dots, a_k,c^6_\ell$ together with $p_\ell$ and, in turn, $t_\ell$ is one of the remaining $(n-k)$-many points incidents to it. From the table above one can similarly see how every element $b_\ell^1,\dots,b^7_\ell$ and $c^1_\ell,\dots,c^6_\ell$ is obtained in  $B=F(X_AX_P)$ from the generators $X_AX_P$. For every $\ell<\omega$, we consider the set 
	\[X_\ell \coloneqq \{ a_i: i<\omega\}\cup \{ t_i : i\leq \ell\} \]
	and define the model $A_\ell=\mrm{acl}_B(X_\ell)$. We claim that $(A_\ell)_{\ell<\omega}$ and $B$ witness the Construction Principle $\mrm{CP}(\mathbf{K},\ast)$. 
		
				\begin{table}[H]
		\begin{tabular}{|l|l|l|l|l|l|c|l|} \hline
			& $b^1_\ell$& $b^2_\ell$& $b^3_\ell$& $b^4_\ell$& $b^5_\ell$&$b^6_\ell$ &$b^7_\ell$\\ \hline \hline 
			$t_\ell$& & & & & & &$\times$\\\hline 
			$a_1$& $\times$ & & $\times$& & & &\\ \hline
			$a_2$& $\times$& $\times$& & & & &\\ \hline
			$a_3$& $\times$& $\times$& $\times$& $\times$& $\times$&$\times$ &$\times$\\ \hline
			$a_4$& $\times$& $\times$& $\times$& $\times$& $\times$&$\times$ &$\times$\\ \hline 
			$\dots$& $\times$& $\times$& $\times$& $\times$& $\times$& $\times$&$\times$\\\hline
			$a_{k-1}$& $\times$& $\times$& $\times$& $\times$& $\times$&$\times$ &$\times$\\ \hline
			$a_k$& $\times$& $\times$& $\times$&$\times$ & $\times$&$\times$ &$\times$\\ \hline
			$p_{\ell+1}$& & $\times$& $\times$& & & &\\ \hline 
			$c^1_\ell$& $\times$& & & $\times$& & &\\ \hline 
			$c^2_\ell$& & $\times$& & $\times$& $\times$& &\\ \hline 
			$c^3_\ell$& & & $\times$& & $\times$& &\\\hline 
			$c^4_\ell$& & & & $\times$& &$\times$ &\\ \hline
			$c^5_\ell$& & & & & $\times$& $\times$ &\\ \hline 
			$c^6_\ell$& & & & & & $\times$ &$\times$\\\hline
			$p_{\ell}$& & & & & & &$\times$\\ \hline
		\end{tabular}
	\end{table}	

	\noindent  We first claim that $A_\ell \sleq B$ for all $\ell<\omega$. Starting from $\{ a_1,\dots, a_k, t_\ell \}\subseteq X_\ell\subseteq A_\ell$ we can proceed with the order suggested by the table above, namely:
	\begin{align*}
	p_{\ell+1}<b^1_\ell <b^2_\ell<b^3_\ell< c^1_\ell< c^2_\ell < c^3_\ell < b^4_\ell < b^5_\ell < c^4_\ell<c^5_\ell<b^6_\ell<c^6_\ell<b^7_\ell<p_\ell.
	\end{align*} 
	It is routine to check that this is in fact a $\mrm{HF}$-order. So from $X_\ell$ we can first adjoin $p_{\ell+1}$ and then obtained $p_{\ell}$. By proceeding in the same fashion, from $p_\ell$ we can then construct $p_{\ell-1}$, and so on. In particular, we obtain that
	\[ X_\ell \hleq X_\ell \cup \{p_{\ell+1}  \}\hleq  X_\ell \cup \{ p_i : i\leq \ell+1\}. \]
	We can then add all points $p_{j}$ with $j\geq i+2$, so it follows that 
	\[ X_\ell \cup \{ p_i : i\leq \ell+1\} \hleq  X_\ell \cup X_P \]
	and in turn we have $ X_\ell \cup X_P  \hleq B$. Now let $Y_\ell=\{p_i :i\geq \ell+1  \}$. By the table above one can readily verify that $ \{ p_i : i\leq \ell\}$ is algebraic over $X_\ell Y_\ell$, as in each step of the previous $\mrm{HF}$-construction we either adjoined a block incident to exactly $k$-many points or we adjoined a novel point incident to some block.  Thus, since the previous display shows that $X_\ell Y_\ell\hleq B$,  we obtain from Fact~\ref{lemma:useful_facts_2} that  $F(X_\ell Y_\ell)=\mrm{acl}(X_\ell Y_\ell)$. Additionally, since $B$ is clearly algebraic over $X_\ell Y_\ell$, we obtain that $F(X_\ell Y_\ell)=B$. We conclude that $X_\ell Y_\ell $ is a  basis of $B$ and so that $A_\ell \sleq B$ for all $\ell<\omega$. Moreover, we also obtain from Fact~\ref{lemma:useful_facts_2} that $A_\ell=\mrm{acl}(X_\ell)=F(X_\ell)$, which shows that $A_\ell \in \mathbf{K}$ for all $\ell<\omega$. Together with the previous observation this gives us that $A_\ell\sleq B$ for all $\ell<\omega$.
	
	\smallskip
	\noindent	Additionally, it follows immediately by  the definition of $A_\ell$ for $\ell<\omega$ that for $\ell<m$ we have $A_\ell=F(X_\ell)\sleq F(X_m)=A_m$. Moreover, it is also clear by the choice of $t_i$ for all $i<\omega$ that $t_{\ell+1}\notin A_\ell$ for all $\ell<\omega$. Thus $(A_\ell)_{\ell<\omega}$ forms a strict $\sleq$-chain in $\mathbf{K}$ with $A_\ell\sleq B$ for all $\ell<\omega$.
	
	\smallskip
	\noindent   Finally,  let  $A=\bigcup_{\ell<\omega}A_\ell$ with basis $X_\omega=\bigcup_{\ell<\omega}X_\ell$. It remains to show that $A\not \pleq B$. Let $D\in \mathbf{K}$ and suppose that $B\sleq D$, we  claim that $A\not\sleq D$. Suppose towards contradiction that $A\sleq D$, i.e., there is a basis $X_D$ of $D$ and a basis $Y_A$ of $A$ such that $Y_A\subseteq X_D$. Then $D=F(Y_A)\ast F(X_C)$ for $X_C=X_D\setminus Y_A$. Since $|Y_A|=\aleph_0$ and $A=F(Y_A)=F(X_\omega)$, we can assume without loss of generality that $Y_A=X_\omega$, whence by Lemma~\ref{lemma:useful_facts} it follows that there is a \emph{wellfounded} $\mrm{HF}$-order $<_D$ of $D$ over $X_\omega$.  Now, since $<_D$ is wellfounded, there is an index $\ell<\omega$ such that $p_\ell$ appears first in the order $<_D$ among $(p_i)_{i<\omega}$. Consider then the following subconfiguration from the table above, i.e.:
	\[C_\ell=\{p_{\ell+1}\} \cup \{ b^i_\ell \mid 1\leq i\leq 7\} \cup  \{ c^i_\ell \mid 1\leq i\leq 6\}  \]
	and notice that $C_\ell\subseteq D$ but $C_\ell\not\subseteq X_\omega$. By Definition~\ref{Steiner:def_hyperfree} and by directly inspecting the table above we see that $C_\ell$ does not contain any element hyperfree in $X_\omega\cup C_\ell\cup \{ p_\ell \}$. Since we assumed that $p_\ell$ is first in $X_P$ according to the ordering $<_D$, this contradicts the fact that $<_D$ is a $\mrm{HF}$-order of $D$ over $X_\omega$. This shows that $D$ does not admit a wellfounded $\mrm{HF}$-order over $A$. By Lemma~\ref{lemma:useful_facts} this shows that $A\not\sleq D$ and thus completes our proof.
\end{proof}

\begin{theorem}\label{prop:steiner}
	Let $(\mathbf{K}, \sleq)$ be as in \ref{context:steiner}, then  there is an $\mathfrak{L}_{\infty,\omega_1}$-free structure $M\notin \mathbf{K}$ of size $\aleph_1$ and, under $V=L$, there is for every $\kappa\geq \aleph_0$ an $\mathfrak{L}_{\infty,\kappa^+}$-free structure $M\notin \mathbf{K}$  of size $\kappa^+$.
\end{theorem}
\begin{proof}
	Immediate by Lemmas~\ref{steiner:AEC}, \ref{CP:steiner}, Theorem~\ref{main:theorem} and Corollary~\ref{main:corollary}.
\end{proof}

\noindent We conclude with the following corollary. Since degenerate (i.e., finite) free $(k,n)$-Steiner systems are often treated as a pathological examples, we state the following corollary also with respect to the class of infinite  $(k,n)$-Steiner systems. Clearly, the following result follows immediately from the previous theorem.

\begin{corollary}
	The class of (infinite) free $(k,n)$-Steiner systems is not axiomatisable in $\mathfrak{L}_{\infty,\omega_1}$ and, under $V=L$, it is not axiomatisable in $\mathfrak{L}_{\infty,\infty}$. 
\end{corollary}

\subsection{Free generalised $n$-gons}\label{Application:ngons}

We provide in this section a second geometric application of the Construction Principle $\mrm{CP}(\mathbf{K},\ast)$, i.e., we consider the case of free generalised $n$-gons for all $n\geq 3$. This section complements the model-theoretic study of the first-order theory of free generalised $n$-gons, for which we redirect the reader to \cite{Ammer_tent,paolini&hyttinen,PQ}. We follow the same pattern of the previous section on Steiner systems --- we omit the details of those proofs which are the same and focus on the reason why free generalised $n$-gons satisfy $\mrm{CP}(\mathbf{K},\ast)$. We start by recalling some preliminary definitions.

\begin{definition}\label{def:graphs}
	Let $(G,E)$ be a graph, then we define the following:
	\begin{enumerate}[(1)]
		\item the \emph{valency} of an element $a\in G$ is the size of the set $\{b\in G : aEb \}$ of neighbours of $a$;
		\item  the \emph{distance} $d(a,b) = d(a, b/G)$ between two elements $a,b\in G$ is the length $n$ of the shortest path $a=c_0Ec_1E\dots E c_{n-1}Ec_n=b$ in $G$, and it is $\infty$ if there is no such path;
		\item the \emph{girth} of a graph is the length of its shortest cycle;
		\item the \emph{diameter} of a graph is the maximal distance between any two elements;
		\item the graph $G$ is \emph{bipartite} if there are $A,B\subseteq G$ such that $A\cap B=\emptyset$, $A\cup B= G$, and every edge has one vertex in $A$ and one in $B$.
	\end{enumerate}
\end{definition}

\begin{definition}\label{def:n_gon}
	Let $L$ be the two-sorted language where the symbols $p_0,p_1,\dots$ denote \emph{points} and the symbols $\ell_0,\ell_1,\dots$ denote \emph{lines}, the symbol $p\, \vert\,\ell$ (or $\ell\, \vert\,p$) means that the point $p$ is \emph{incident} with the line $\ell$.  For every $n\geq 3$, we define a \emph{generalised $n$-gon} as a bipartite graph (with points and lines) with girth $2n$ and diameter $n$. A \emph{partial $n$-gon} is a model of the universal fragment of the theory of generalised $n$-gons, i.e., it is a bipartite graph with girth at least $2n$.
\end{definition}

The following free completion process for generalised $n$-gons is a special case of the one described in \cite{PQ}. We stress that, in the literature, it is generally assumed that the starting configuration in the completion process of a generalised $n$-gon forms a connected partial $n$-gon (cf.~\cite{funk}). However, as remarked in \cite[p.~51]{PQ}, already at the second stage of the completion process given by the following construction one obtains a connected graph, so that the following notion is consistent with Tits' original definition of free generalised $n$-gons from \cite{tits}. As we already mentioned in the case of Steiner systems (cf.~Remark~\ref{issue:freeness}), the following Definition~\ref{ngons:free_models} of free generalised $n$-gon is chosen to simplify technicalities in this section. We consider later in Section~\ref{sec:remark_freeness} the problem of the correct definition of free structures in incidence geometry. We next recall the notion of clean arc from \cite[Def.~2.3]{Ammer_tent}.

\begin{notation}\label{clean_arc}
	Given a partial $n$-gon $A$, a chain of $n-2$ elements $z_1,\dots,z_{n-2}$ such that $a\,\vert\, z_1$, $b\,\vert\, z_{n-2}$, and every element in $\{z_1, ..., z_{n-2}\}$ has exactly valency 2, is called a \emph{clean arc} between $a$ and $b$. Notice that $a$ and $b$ do \emph{not} belong to the clean arc, and we often refer to them as the \emph{endpoints} of the clean arc $\bar{z}=(z_1,\dots, z_{n-2})$. Notice also that the parity of a clean arc depends on the underlying value of $n$: clean arcs are even for even $n$ and odd for odd $n$.
\end{notation}

\begin{definition}\label{ngons:free_completion2}
	Let $A$ be a partial $n$-gon. We let $A_0=A$ and, for every $i<\omega$, we let $A_{i+1}$ be obtained by adding a clean arc $a\, \vert\,z_1\, \vert\,z_2\, \vert\,\dots\, \vert\,z_{n-2}\, \vert\,b$ between any two elements $a,b\in A_i$ such that
	\begin{enumerate}[(a)]
		\item either $d(a,b/A_i)=n+1$,
		\item or $d(a,b/A_i)=\infty$ and $a,b$ are of the appropriate sort (i.e., $a,b$ have the same sort if $n$ is odd and have different sorts if $n$ is even).
	\end{enumerate}
	The structure $F(A)\coloneq\bigcup_{i < \omega}A_i$ is called the \emph{free completion} of $A$ and we say that $F(A)$ is freely generated over $A$. We say that $A$ is \emph{non-degenerate} if $F(A)$ is infinite, and \emph{degenerate} otherwise.
\end{definition}

\begin{definition}\label{ngons:free_models}
	Let $\kappa$ be any cardinal, then the configuration $X_\kappa$ is the unique set of points and lines $x_i$ with $i<\kappa$ such that:
	\begin{enumerate}[(i)]
		\item $x_0$ is a point;
		\item  for every limit $\alpha<\kappa$ we have that $x_\alpha$ is a point;
		\item for every $\alpha<\kappa$ we have that the element $x_\alpha$ is incident to $x_{\alpha+1}$.
	\end{enumerate}
	For $n$ odd, we say that a generalised $n$-gon $B$ is \emph{free} if it is the free completion of the configuration $X_\kappa$ for some cardinal $\kappa$. For $n$ even, we say that a generalised $n$-gon $B$ is \emph{free} if it is the free completion of the configuration $X_\kappa$ for some cardinal $\kappa\geq \aleph_0$, or it is the free completion of $X_{k}$ for some even $k<\omega$. We denote by $\Gamma_i$ the free completion $F(X_i)$ and we refer to $X_\kappa$ as the \emph{basis} of $F(X_\kappa)$.
\end{definition}

\noindent We recall the definition of $\mrm{HF}$-orderings for generalised $n$-gons. It is important to notice that in this case the notion of hyperfree is defined at the level of tuples, rather than elements (cf.~\cite{Ammer_tent} and \cite[Remark 4.35]{PQ}). As the reader can easily verify, the following definition specialises \cite[Def. 3.10]{PQ} in the setting of generalised $n$-gons.

\begin{definition}\label{ngons:def_hyperfree}
	For finite partial $n$-gons $A\subseteq B$, we define the following.
	\begin{enumerate}[(1)]
		\item A tuple $\bar{a}\in A^{<\omega}$ is \emph{hyperfree in $B$} if it satisfies one of the following conditions:
		\begin{enumerate}[(i)]
			\item $\bar{a}=a$ is a \emph{loose end}, i.e., it is a single element incident with at most one element in $B$;
			\item $\bar{a}$ is a \emph{clean arc}, i.e., $\bar{a}$ is a chain $a_1\, \vert \, a_2\, \vert\ \dots\, \vert\ a_{n-3}\, \vert\,a_{n-2}$ where each $a_i$ for $1\leq i\leq n-2$ is incident exactly to two elements in $B$ (recall \ref{clean_arc}).
		\end{enumerate}
		\item We write $A\hleq B$ if there is a linear order $<$ of $B\setminus A$ such that every $c\in B\setminus A$ belongs to a tuple $\bar{d}=(d_1,\dots, d_\ell)\in (B\setminus A)^{<\omega}$ which is  hyperfree in $A\cup \{d_1,\dots, d_\ell\}^\downarrow$. If the linear order $<$ on $B\setminus A$ witnesses the fact that $A\hleq B$ then we say that $<$ is a \emph{$\mrm{HF}$-order over $A$}.
		\item  We say that $B$ is \emph{open} if $\emptyset \hleq B$.
	\end{enumerate}	
\end{definition}

We recall two useful facts which we shall use in the proofs below. Fact (1) follows immediately from the previous definitions, Fact (2) was proved in \cite[Lemma 2.18]{Ammer_tent}. 

\begin{fact}\label{lemma:useful_facts_ngons}
	The following facts hold:
	\begin{enumerate}[(1)]
		\item if $X$ is a partial generalised $n$-gon and $A=F(X)$, then there is a wellfounded $\mrm{HF}$-order of $A$ over $X$;
		\item let $A\in \mathbf{K}$ and $|A|\geq \aleph_0$, $C\subseteq A$ and $C\hleq A$, then $F(C)\cong \mrm{acl}_A(C)\hleq A$.
	\end{enumerate}
\end{fact}

\begin{notation}
	As for Steiner systems, by the previous fact we often identify $F(C)$ with $\mrm{acl}_A(C)$ whenever $C\hleq A\in \mathbf{K}$. We also write simply $\mrm{acl}(C)$ instead of $\mrm{acl}_A(C)$, since all infinite free generalised $n$-gons are elementary equivalent and $F(C)\preccurlyeq_{\omega,\omega} M$ whenever $C\hleq M$ and $M$ is elementary equivalent to a free generalised $n$-gon (by \cite{paolini&hyttinen} for the case $n=3$ and \cite{Ammer_tent} for $n\geq 3$; cf. also \cite[\S 4.3]{PQ}).
\end{notation}

The case of free generalised $n$-gons is analogous to the case of free $(k,n)$-Steiner systems. Their first order theory was investigated in \cite{paolini&hyttinen} for $n=3$ (i.e., for the subcase of projective planes, which are exactly the generalised $3$-gons), and later in \cite{Ammer_tent} for the more general case with $n\geq 3$. It was shown in these works that the first-order theory of all infinite free generalised $n$-gons is complete, strictly stable, and that the relation of $\mrm{HF}$-ordering $A\hleq B$ is equivalent to the relation of elementary embedding $A\preccurlyeq_{\omega,\omega} B$. We focus here on the class of free generalised $n$-gons with the strong submodel relation $\sleq$ defined below.

\begin{definition}\label{ngons:free_products}
	Let $A,B\in \mathbf{K}$ be disjoint free generalised $n$-gons generated respectively by $X_A$ and $X_B$, then we let $A\ast B=F(X_AX_B)$. Similarly, if $B_i\in \mathbf{K}$ is a disjoint family of free generalised $n$-gons with bases $(X_i)_{i\in I}$, then we let $\bigast_{i\in I} B_i=F(\bigcup_{i\in I}X_i)$.
\end{definition}

\begin{definition}\label{def:ngons_sleq}
	Let $A$ and $B$ be two free generalised $n$-gons, then we write $A\sleq B$ if there is a free generalised $n$-gon $C$ and two bases $X_A$ and $X_C$ of $A$ and $C$ as in \ref{ngons:free_models} (which witness the fact that $A$ and $C$ are free) such that there is no incidence between $X_A$ and $X_C$ and $B=F(X_AX_C)$.
\end{definition}

\begin{context}\label{context:ngons}
	In this section, we fix the class $(\mathbf{K},\sleq)$ where  $\mathbf{K}$ is the class of all free generalised $n$-gons and $\sleq$ is the relation from Definition~\ref{def:ngons_sleq}. 
\end{context}

\noindent We proceed as in the previous sections and show that $(\mathbf{K},\sleq)$ fits the framework of canonical amalgamation classes from Section~\ref{Sec:application}. In particular, the following Lemma~\ref{projective_construction_principle} means that $(\mathbf{K},\sleq)$ is a strictly weak $\mrm{AEC}$ but not an $\mrm{AEC}$, as it crucially does not satisfy Smoothness.

\begin{lemma}\label{ngons:AEC}
	The class $(\mathbf{K},\sleq)$ is a weak $\mrm{AEC}$ with $\mathrm{LS}(\mathbf{K}, \sleq)=\aleph_0$ and it satisfies Conditions~\ref{context:amalgamation_class}(\hyperref[C1]{C1})-(\hyperref[C3]{C3}) and Conditions~\ref{assumptions:subsec_1}(\hyperref[C5]{C5})-(\hyperref[C6]{C6}).
\end{lemma}
\begin{proof}
	Conditions~\ref{assumptions:subsec_1}(\hyperref[C5]{C5})-(\hyperref[C6]{C6}) follow as in Lemma~\ref{steiner:AEC} and Conditions~\ref{context:amalgamation_class}(\hyperref[C2]{C2})-(\hyperref[C3]{C3}) follow directly from Definition~\ref{ngons:free_products} and Definition~\ref{def:ngons_sleq}.  We verify Condition~\ref{context:amalgamation_class}(\hyperref[C1]{C1}) for the case where $n$ is odd and leave the case with $n$ even to the reader. We consider only the finitary case, as then this argument extends inductively to the infinitary setting. Let $A=F(X_A)$ and $B=F(X_B)$ be two free generalised $n$-gons with $X_A$, $X_B$ as in \ref{ngons:free_models}. Let $(x_i)_{i<\kappa}$ and $(y_i)_{i<\lambda}$ be the chains from $X_A$ and $X_B$, respectively, and suppose without loss of generality that $\lambda\leq \kappa$. We verify that $A\ast B\in \mathbf{K}$. The remaining fact that $A\ast B$ satisfies the properties from \ref{def:canonical_jep} follows as in the previous cases.  We distinguish the following cases:
	
	\noindent
	\underline{Case 1}. $\lambda$ is infinite.
	\newline For every limit ordinal $\alpha<\lambda$ and $i<\omega$ we let $\bar{z}_i=(z^1_i,\dots,z^{n-2}_i)$ be the clean arc between $x_{\alpha+i}$ and $y_{\alpha+i}$. Then consider the set 
	\[Y=\{ z^1_{\alpha+i},\dots,z_{\alpha+i}^{n-2} : i= (n-1)\cdot j \text{ for some } 1\leq j<\omega   \}\]
	 and adjoin to $Y$ the sequences $x_{\alpha+i},\dots, x_{\alpha+i+n-1}$ and $y_{\alpha+i+n-1},\dots, y_{\alpha+i+2(n-1)}$ for all $i=(n-1)\cdot j$ and $1\leq j<\omega$. Then the configuration formed by $Y$ together with these sequences and the elements $x_i$ for all $i\geq\lambda$ forms a chain $X_C$ isomorphic to $X_{\kappa}$, as $\kappa=\kappa+\lambda$. Since  $d(x_{\alpha+i},x_{\alpha+i+n-1}/X_C)=d(y_{\alpha+i},y_{\alpha+i+n-1}/X_C)>n$ for all $\alpha<\lambda$ limit and $i<\omega$, it then follows by construction and Definition~\ref{ngons:free_completion2} that $F(X_C)=F(X_AX_B)$, which proves that $A\ast B=F(X_C)$ is a free $n$-gon and thus belongs to $\mathbf{K}$.

	 \noindent
	 \underline{Case 2}. $\lambda$ is a finite cardinal $\ell$.
	 \newline Let $\bar{z}=(z_1,\dots, z_{n-2})$ be the clean arc between $y_0$ and $x_0$. Then $X_C\coloneqq X_AX_B\cup\{z_1,\dots, z_{n-2}\} $ is a chain isomorphic to $X_\kappa$ from Definition~\ref{ngons:free_models} and by construction $F(X_C)=F(X_AX_B)=A\ast B$. As in the previous case we obtain that $A\ast B\in \mathbf{K}$. This verifies Condition~\ref{context:amalgamation_class}(\hyperref[C1]{C1}) and thus completes our proof. 
\end{proof}

\begin{lemma}\label{projective_construction_principle}
	Let $(\mathbf{K},\sleq)$ be the class of free generalised $n$-gons with the relation $\sleq$ from Definition~\ref{def:ngons_sleq}, then $(\mathbf{K},\sleq)$ satisfies $\mrm{CP}(\mathbf{K},\ast)$, i.e., Condition~\ref{CP-AEC:coproducts}(\hyperref[C4]{C4}).
\end{lemma}

\afterpage{
	\begin{landscape}
		\vspace*{\fill}
		\begin{table}[htb!]
			\centering
			\tiny			
			\vspace*{\fill}
					\begin{tabular}{|l|l|l|l|l|l|l|l|l|l|l|l|l|l|l|l|l|l|l|l|l|l|l|l|l|} \hline
						& $c^1_1$ & $c^{n-2}_1$ & $c^1_2$ & $c^{n-2}_2$ & $c^1_3$ & $c^{n-2}_3$ & $c^1_4$ & $c^{n-2}_4$ & $c^1_8$ & $c^{n-2}_8$ & $c^1_9$ & $c^{n-2}_9$ & $c^1_{10}$ & $c^{n-2}_{10}$ & $c^1_{13}$ & $c^{n-2}_{13}$ & $c^1_{15}$ &  $c^{n-2}_{15}$& $c^1_{17}$ &  $c^{n-2}_{17}$& $c^1_{19}$ &  $c^{n-2}_{19}$& $c^1_{21}$ &  $c^{n-2}_{21}$\\  \hline			\hline						
						$t^1_{\ell}$   &         &         &         &       &         &       &         &         &         &         &         &         &            &            &            &            &                   &                                             &         $\times$     &&& &&\\ \hline
						$t^{n-2}_{\ell}$   &         &         &         &       &         &       &         &         &         &         &         &         &            &            &            &            &                   &                                             &             & &$\times$&&&\\ \hline
						$a_1$        &$\times$       &         &         &       &         &       &         &         &         &         &         &         & $\times$          &            &            &            &            &                                                     &            &  &            &&$\times$&\\ \hline
						$a_{2}$      &         & $\times$      &$\times$       &       &         &       &         &         &         &         &         &         &            &            &            &            &            &                                                      &            &&            &&& \\ \hline
						$a_{3}$      &         &         &         & $\times$    & $\times$      &       &         &         &         &         & $\times$       &         &            &            &            &            &            &                                                      &            & &&&&\\ \hline
						$a_{4}$      &         &         &         &       &         &    &   $\times$      &         &  $\times$       &         &         &         &            &            &            &            &            &                                                      &            &&& &&\\ \hline
						$p^1_{\ell+1}$ &         &         &         &       &       & $\times$    &         &   $\times$       &        &         &         &         &            &            & $\times$         &            &            &                                                      &            & &&&&\\ \hline
						$c^1_5$       & $\times$       &         &         &       &         &       &         &         &         &      $\times$    &         &         &            &            &            &            &            &                                                      &            & &&&&\\ \hline
						$c^{n-2}_5$         &         &         &         &       &     $\times$    &      &         &         &         &       &         &         &            &            &            &            &            &                                                      &            & &&&&\\ \hline
						$c^1_6$      & $\times$       &         &         &       &         &       &         &       &         &         &         & $\times$       &            &            &            &            &            &                                                      &            &&&&& \\ \hline
						$c^{n-2}_6$         &         &         &         &       &         &       &   $\times$      &       &         &         &         &      &            &            &            &            &            &                                                      &            &&&&& \\ \hline
						$c^1_7$      &         &         &$\times$       &       &         &       &         &         &         &         &         &         &            & $\times$            &            &            &            &                                                      &            & &&&&\\ \hline
						$c^{n-2}_7$         &         &         &         &       &         &       &         $\times$ &       &         &         &         &         &            &          &            &            &            &                                                       &            &&&&&\\ \hline
						$ c^{1}_{11} $    &         &         &         &       &         &       &         &         &$\times$         &         &         &         &            &            &           &  $\times$           &            &                                                      &            & &&&&\\ \hline
						$c^{n-2}_{11}$     &         &         &         &       &         &       &         &         &         &         &      $\times$    &       &            &            &            &             &            &                                                       &            &&&&&\\ \hline
						$c^{1}_{12}$     &         &         &         &       & $\times$       &       &         &         &         &         &         &         &            &            &            &           &            $\times$&                                                      &            &&&&& \\\hline
						$c^{n-2}_{12}$     &         &         &         &       &         &       &         &         &         &         &         &         &        $\times$     &          &            &            &            &                                                     &            &&& && \\ \hline
						$c^{1}_{14}$     &  $\times$       &         &         &       &         &       &             &     &         &         &         &         &            &            &            &            &            & $\times$                                                      &            &&&&& \\ \hline
						$c^{n-2}_{14}$     &         &         &         &       &         &       &         &         &         &           &         &         &            &            &          $\times$   &        &            &                                                &            & && &&  \\ \hline
						$c^{1}_{16}$     &         &         &         &       &         &       &             &     &$\times$         &         &         &         &            &            &            &            &            &                                                      &            &$\times$&&&& \\ \hline
						$c^{n-2}_{16}$     &         &         &         &       &         &       &         &         &         &         &         &         &           $\times$  &            &            &           &            &                                                 &             & && && \\ \hline 
						$c^1_{18}$     &         &         &    $\times$     &       &         &       &         &         &         &         &         &         &            &            &            &            &            &       &           & &&$\times$&& \\ \hline
						$c^{n-2}_{18}$     &         &         &         &       &         &       &         &         &         &         &         &         &            &            &            &            &   $\times$         &       &           & &&&& \\ \hline                   
						$c^1_{20}$     &         &         &         &       &  $\times$       &       &         &         &         &         &         &         &            &            &            &            &            &       &           & &&&&  \\ \hline
						$c^{n-2}_{20}$     &         &         &         &       &         &     &         &         &         &         &         &         &            &            &            &            &            &       &           && $\times$&&& $\times$\\ \hline                          
						$p^1_\ell$     &         &         &         &       &         &       &         &         &         &         &         &         &            &            &            &            &            &       &       $\times$    & && & &\\ \hline
						$p^{n-2}_\ell$     &         &         &         &       &         &       &         &         &         &         &         &         &            &            &            &            &            &       &           & &&&& $\times$ \\ \hline                         
					\end{tabular}
			
			\caption{\label{tab:ngons} The configuration $D^n_\ell$ for $n\geq 3$ odd.}
				\end{table}			
			\vspace*{\fill}
\end{landscape}}

\begin{proof}
	The proof proceeds similarly to Lemma~\ref{CP:steiner}. The main idea is to construct a free generalised $n$-gon $B$ and a substructure $A\sleq B$ so that every $D\in \mathbf{K}$ with $B\sleq D$ cannot be obtained from $A$ by means of a wellfounded $\mrm{HF}$-order. We provide details for the proof with $n\geq 3$ odd, and we leave to the reader to adapt our reasoning to the case with $n$ even. We shall refer in this proof to a special rigid configuration inside a free generalised $n$-gon, which we describe in Table~\ref{tab:ngons}. To provide the reader with some intuition, we also add Table~\ref{tab:projective}, which simply specialises Table~\ref{tab:ngons} to the case of projective planes, namely with $n=3$. We first explain how to read Table~\ref{tab:ngons}.
	
	\smallskip
	\noindent Let $X_A=\{a_i: i<\omega\}$ and $X_P=\{p^j_i: i<\omega, \; 1\leq j\leq n-2 \}$ be two sets of points such that $\bar{p}_{i}=(p^1_i,\dots, p^{n-2}_{i})$ is a clean arc. For every $i<\omega$ we let $\bar{u}_{i}$ be a chain between $a_i$ and $a_{i+1}$ of length $>n$. For every $i<\omega$ we let $\bar{v}_{i}$ be the clean arc between $p^{n-2}_i$ and $p^1_{i+1}$, and we let $\bar{z}=(z_1,\dots, z_{n-2})$ be the clean arc between $p^1_0$ and $a_0$. We let $X_U$ be the set of all elements from the tuples $\bar{u}_{i}$, $X_V$ be the set of  elements from the tuples $\bar{v}_{i}$ and $X_Z=\{z_1,\dots, z_{n-2}\}$. The resulting configuration $X_AX_PX_UX_VX_Z$ is a chain of length $\omega$ (i.e., it is isomorphic to the configuration $X_\omega$ from \ref{ngons:free_models}), and thus its free completion $F(X_AX_PX_UX_VX_Z)$ is a free generalised $n$-gons. For notational convenience we let $X\coloneqq X_AX_PX_UX_VX_Z $.
	
	\smallskip
	\noindent For each $\ell<\omega$, we consider the configuration $D^n_\ell$ from  Table~\ref{tab:ngons} inside $B$. The key point of the argument is that $D^n_\ell$ admits two different $\mrm{HF}$-ordering, one starting from the generators $X$ and one from $X_A$, $\bar{p}_{\ell+1}$ and the clean arc $\bar{t}_\ell$, which we define below. First, we clarify that the elements in $D^n_\ell$ are essentially words over the  generators $a_1,a_{2}, a_3,a_4$ and $p^1_\ell,p^1_{\ell+1},p^{n-2}_\ell$. In fact, starting with these elements, the free construction process in the setting of $n$-gons provides us with the clean arc $c^1_1,\dots, c_1^{n-2}$, i.e., with a sequence of elements such that $a_1,c^1_1,\dots, c_1^{n-2},a_{2}$ forms a chain. Notice that, for the special case of projective planes we are simply adjoining the unique line incident to $a_1$ and $a_{2}$. In Table~\ref{tab:ngons} we represent this clean arc (and the others) by displaying only the first and last element in it, but one should keep in mind that \emph{every} element in the clean arc is incident to two elements. Similarly, we then adjoin a clean arc $c^1_2,\dots, c_2^{n-2}$ between $a_2$ and $a_3$, and we represent this in Table~\ref{tab:ngons} by showing its first and last element. We proceed in this fashion and we progressively add the clean arcs $\bar{c}_i$ for all $i\leq 16$, as displayed in the table. After this, we proceed in the following order:
	\[\bar{c}_{18}<\bar{c}_{21}<\bar{c}_{17}<\bar{c}_{20}<\bar{c}_{19} <\bar{t}_{\ell}. \]
	Since $p^1_{\ell},p^{n-2}_{\ell}$ are given, the reader can easily verify that this is indeed a $\mrm{HF}$-order and, additionally, that at every step we adjoin clean arcs, thus showing that the endresult is algebraic over the generators. It follows that the resulting configuration $D^n_\ell$ from Table~\ref{tab:ngons} is a subconfiguration of $B$ which is $\mrm{HF}$-constructed from the generators $a_0,a_1,a_{2}, a_3,a_4$ and $p^1_\ell,p^{n-2}_\ell,p^1_{\ell+1}$.  Also, notice that the tuple $\bar{t}_\ell=(t^1_{\ell},\dots,t^{n-2}_\ell)$ has been identified as the unique clean arc between $c^{1}_{17}$ and $c^{1}_{19}$. In particular, we have that $t^i_{\ell}\in B$ for all $\ell<\omega$ and $1\leq i \leq n-2$. 
	
	\smallskip
	\noindent We use the configuration $D^n_\ell$ to produce a chain of free models $(A_i)_{i<\omega}$ as in the statement of $\mrm{CP}(\mathbf{K},\ast)$. For $\ell=0$ let $\bar{s}_0=(s^1_0,\dots, s^{n-2}_{0})$ be the clean arc between $a_0$ and $t^1_\ell$. For $\ell=m+1$ we let $\bar{s}_\ell=(s^1_\ell,\dots, s^{n-2}_{\ell})$ be the clean arc between $t^1_\ell$ and $t^{n-2}_m$. Notice that by construction such clean arcs do not belong to the configuration $D^n_i$ from above for all $i\leq \ell$. Consider then the configuration
	\[Z_\ell\coloneqq X_AX_U\cup\{t^i_j : 1\leq i\leq n-2, \; j\leq \ell  \} \cup \{ s^i_j : 1\leq i\leq n-2, \; j\leq \ell   \},  \]
	which is a chain isomorphic to $X_\omega$ from \ref{ngons:free_models}. Then we define $A_\ell=F(Z_\ell)$ and it follows from Definition~\ref{ngons:free_models} and the previous observation that $A_\ell$ is free, i.e., $A_\ell \in \mathbf{K}$. Additionally, it is immediately clear by construction that we also have $A_i\sleq A_j$ whenever $i<j<\omega$.
	
		\smallskip
	\noindent  In order to verify $\mrm{CP}(\mathbf{K},\ast)$ we next show that also $A_\ell\sleq B$ holds for all $\ell<\omega$.  The key observation is that we can modify the ordering described before so to start from $\bar{t}_\ell$ and use it to obtain $\bar{p}_\ell$ (this is similar to how we proceed in Lemma~\ref{CP:steiner}). Starting from $Z_\ell$ we construct the configuration $D^n_\ell$ as follows. First, we adjoin the tuples $\bar{p}_{i}$ and $\bar{v}_i$ for all $i\geq \ell+1$, notice that by the choice of the tuples $\bar{v}_i$ each element can be added as a loose end. From this configuration we then proceed algebraically, i.e. we perform the free completion process from Definition~\ref{ngons:free_completion2}. We start by adding the tuples $\bar{c}_1,\dots, \bar{c}_{21}$ following the order indicated by their numbering and, then, we adjoin the tuple $\bar{p}_{\ell}$ as the clean arc witnessing that the distance between $c^1_{17}$ and $c^{n-2}_{21}$ is exactly $n$. We then adjoin the tuple $\bar{v}_{\ell}$.
	\newline Similarly, we then use the point $p^1_\ell$ and the generators $a_1,\dots, a_4$ to obtain the tuple $\bar{p}_{\ell-1}$ and $\bar{v}_{\ell-1}$, and we continue in this fashion until we adjoint also $\bar{p}_0$. Finally, we adjoin the clean arc $\bar{z}$ between $a_0$ and $p^1_0$. In this way we have adjoined all elements from $X$ starting from $Z_\ell$. From the table above one can verify that this is indeed part of the free completion process from $Z_\ell$ and the tuples $\bar{p}_j$, $\bar{v}_j$ with $j\geq \ell+1$. 
	
	\smallskip \noindent Let $Y_\ell= \{\bar{p}_j : j\geq \ell+1 \}\cup\{\bar{v}_j: j\geq \ell+1  \}$.  Since by construction and the reasoning above we have that $Z_\ell Y_\ell\hleq B$, it also follows that $F(Z_\ell Y_\ell)=\mrm{acl}(Z_\ell Y_\ell)=B$. Since $A=F(Z_\ell)$ and $Z_\ell\hleq Z_\ell Y_\ell$, as we argued above, we obtain that $A_\ell=F(Z_\ell)=\mrm{acl}(Z_\ell)$. It follows that $B=A_\ell \ast F(Y_\ell)$ and so $A_\ell\sleq B$ for all $\ell<\omega$.
	
	\smallskip \noindent Finally, let $A=\bigcup_{i<\omega}A_i$ and notice that $Z_\omega=\bigcup_{\ell<\omega}Z_\ell$ is a basis for it. Then since $\mathbf{K}$ is a weak $\mrm{AEC}$ and  $A_i\in \mathbf{K}$ for all $i<\omega$ it follows that  $A\in \mathbf{K}$. We claim that $A\not \pleq B$.  Suppose there is $E\in \mathbf{K}$ such that $A\sleq E$ and $B\sleq E$. By reasoning exactly as in Lemma~\ref{steiner:AEC}, it follows that there is a wellfounded $\mrm{HF}$-order $<_E$ of $E$ over $X_A$. Then let $\bar{p}_\ell\in X_P$ be the first tuple in $X_P$ occurring in the ordering $<_E$ and consider the configuration $D^n_\ell$ from Table~\ref{tab:ngons}. In particular, since $p^j_{\ell}<_E p^1_{\ell+1}$ for some $1\leq j\leq n-2$ it follows by directly inspecting the configuration $D^n_\ell$ from Table~\ref{tab:ngons} that it does not contain any tuple which is hyperfree in $Z_\omega\cup D^n_\ell $. Since we assumed that $p_\ell$ is first in the ordering $<_E$, this contradicts the fact that $<_E$ is a $\mrm{HF}$-order of $E$ over $X_\omega$. As in Lemma~\ref{steiner:AEC}, we obtain that $E$ does not admit a wellfounded $\mrm{HF}$-order over $A$ and so $A\not\sleq E$. This completes the proof.
\end{proof} 
	
		\begin{table}[H]			 
	\begin{tabular}{|l|l|l|l|l|l|l|l|l|l|l|l|l|} \hline
		& $c_1$&$c_2$&$c_3$&$c_4$& $c_8$& $c_9$& $c_{10}$& $c_{13}$& $c_{15}$& $c_{17}$& $c_{19}$&$c_{21}$\\ \hline \hline 
		$t_{\ell}$&& & & & & & & & &$\times$  &$\times$&\\\hline 
		$a_{1}$&$\times$& &  & & & & $\times$& &  &  &&$\times$\\ \hline
		$a_{2}$&$\times$ &$\times$& & & & & & & & &&\\ \hline
		$a_{3}$& & $\times$ &$\times$ & & & $\times$& & && && \\ \hline
		$a_{4}$& && & $\times$ &$\times$ & & & &&  &&\\ \hline
		$p_{\ell+1}$& & &$\times$&$\times$ & & & & $\times$&&  &&\\ \hline 
		$c_5$&$\times$ &&$\times$ & & $\times$& & & && && \\ \hline 
		$c_6$&$\times$ && &$\times$ & & $\times$& & && & &\\ \hline 
		$c_7$& & $\times$&& $\times$& & & $\times$& &&  &&\\ \hline 
		$c_{11}$& && & & $\times$& $\times$& & $\times$&&  &&\\ \hline 
		$c_{12}$&&&$\times$  & & & & $\times$& & $\times$& &&\\ \hline 
		$c_{14}$&$\times$ && & & & & &$\times$ & $\times$& &&\\ \hline 
		$c_{16}$& && & &$\times$ & & $\times$& && $\times$ &&\\ \hline
		$c_{18}$& &$\times$& & & & & & &$\times$&  &$\times$&\\\hline
		$c_{20}$& &&$\times$ & & & & & &&  &$\times$&$\times$\\\hline				
		$p_\ell$& && & & & & & &&$\times$  &&$\times$\\\hline 
	\end{tabular}\caption{\label{tab:projective} The configuration $D^3_\ell$.}
\end{table}

\begin{theorem}\label{prop:ngons}
	Let $(\mathbf{K}, \sleq)$ be as in \ref{context:ngons}, then  there is an $\mathfrak{L}_{\infty,\omega_1}$-free structure $M\notin \mathbf{K}$ of size $\aleph_1$ and, under $V=L$, there is for every $\kappa\geq \aleph_0$ an $\mathfrak{L}_{\infty,\kappa^+}$-free structure $M\notin \mathbf{K}$  of size $\kappa^+$.
\end{theorem}
\begin{proof}
	Immediate by Lemmas~\ref{steiner:AEC}, \ref{CP:steiner}, Theorem~\ref{main:theorem} and Corollary~\ref{main:corollary}.
\end{proof}

\noindent As in the case of Steiner systems, we obtain the following non-axiomatisability results also with respect to the class of infinite free generalised $n$-gons. The following corollary follows immediately from the previous theorem.

\begin{corollary}
	The class of (infinite) free generalised $n$-gons is not axiomatisable in $\mathfrak{L}_{\infty,\omega_1}$ and, under $V=L$, it is not axiomatisable in $\mathfrak{L}_{\infty,\infty}$. 
\end{corollary}

\subsection{A remark on the categoricity of free incidence structures}\label{sec:remark_freeness}

In the previous sections we have considered free structures in incidence geometry, and we focused in particular on $(k,n)$-Steiner systems and generalised $n$-gons. As we already mentioned, our definition of free $(k,n)$-Steiner systems and free generalised $n$-gons (cf.~\ref{Steiner:free_models}, \ref{ngons:free_models}) are slightly different from the standard ones from the geometric literature, which do not generally take into account the uncountable case and do not immediately entail that these are uncountably categorical classes. We consider this issue in this last section and we justify our previous definitions. We consider here only the case of free projective planes, namely $3$-gons, as they are the prototypical example of free objects in incidence geometry, and we leave it to the reader to adapt our arguments to the cases of Steiner systems and generalised $n$-gons for all $n\geq 3$.  Before introducing Definition~\ref{def:free_correct} and proving Theorem~\ref{theorem:free_categoricity} we recall the standard definition of free projective planes from the literature and explain its problems in the uncountable case.

Historically, finitely-generated free projective planes were first introduced by Hall in his seminal article \cite{hall_proj}, where he defined them as the free completions $\pi^k$ of special \emph{finite} configurations $\pi^k_0$ consisting of a line $\ell$, $k-2$ points on $\ell$ and two points off of $\ell$. In \cite{hall_proj} he proved that if $n\neq m$ then  $\pi^n\not\cong \pi^m$, and he showed that the free completion of any finite, non-degenerate, open projective plane is isomorphic to $\pi^n$ for some $n<\omega$. Hall's methods and definitions were significantly simplified by Siebenmann in \cite{sieben}, where he introduced the notion of extension process and hyperfree completion, which is essentially the precursor of the concept of $\mrm{HF}$-ordering. Siebenmann also took into consideration the countably generated free projective plane $\pi^\omega$, which he simply defined as the limit of all finitely-generated ones, i.e., $\pi^\omega=\bigcup_{n<\omega}\pi^n$ (the same had been done by Kopeikina in \cite{kope}). This focus on finitely-generated structures also meant that Siebenmann only considered wellfounded extension processes, thus, in our terminology, only wellfounded $\mrm{HF}$-orderings. Non-wellfounded $\mrm{HF}$-orders were first considered in \cite{paolini&hyttinen} to study the first-order theory of free projective planes, as they provide a tool to study arbitrary (i.e., possibly non-free) open projective planes, and to investigate their first-order model theory. To our knowledge, uncountable free projective planes have not really been considered in the geometric literature. As a matter of fact, one can see that the definition of free structure from \cite{funk} overgeneralises and leads to the paradoxical consequence that every open plane is free. In fact, by specialising the definition of free structure from \cite[Def. 5]{funk} to the specific case of projective planes, one would obtain that \emph{$A$ is a free projective plane if it is the free completion of some open projective plane $A_0$}, i.e., if $A=F(A_0)$ (see Definition~\ref{ngons:free_completion2} for $n=3$). However, this definition clearly overgenerates:  if $A$ is an infinite open projective plane, then it is easy to identify a set $A_0\subseteq A$ so that $A=F(A_0)$. This contrasts to the case where $A_0$ is finite, as Hall's result  that the free completion of any finite, open, projective plane is isomorphic to $\pi^n$ for some $n<\omega$ immediately entails that $F(A_0)$ is free whenever $A_0$ is a finite, open, projective plane (for a similar result concerning all generalised $n$-gons we refer the reader to \cite[Prop.~4.1]{Ammer_tent}).

Given the previous considerations, we believe that the problem with \cite[Def. 5]{funk} is no oversight, but a sign of a difficulty in defining free planes in the infinite setting, where it is not true that free planes are simply free completion of \emph{any} open configuration. In particular, we take that this indicates that to capture the notion of free projective planes of arbitrary cardinality one needs to specify some constraint on the generators $A_0$ of $F(A_0)$. One solution would be, as we did in Definition~\ref{ngons:free_completion2}, to identify a ``canonical'' set of generators, similarly to what Hall did, and simply stipulate that the free models are the free completions of this set. Interestingly, we notice as a side remark that in \cite{kope} Kopeikina essentially identified the classes of planes which are freely generated by lines and points without any incidence, and refers to them as \emph{completely free}. However, to simply stipulate a canonical set of generators seems quite arbitrary. To justify our definition we prove in this section that, for every uncountable cardinal $\kappa$, there is only one projective plane of size $\kappa$ which admits a \emph{wellfounded} $\mrm{HF}$-order. This result further exhibits the key importance of the notion of $\mrm{HF}$-orderings in the study of free structures, and it shows in particular how wellfoundedness determines whether an open plane is free or not. With these motivations in mind we introduce the following alternative definition of \emph{free projective plane}, and we show in Corollary~\ref{incidence:corollary} that for uncountable projective planes it is actually equivalent to the one from Definition~\ref{ngons:free_completion2}.

\begin{definition}\label{def:free_correct}
	Let $A$ be a projective plane, we say that $A$ is \emph{free} if there is a wellfounded $\mrm{HF}$-order $<$ of $A$.
\end{definition}

\begin{notation}
	Let $P$ be a projective plane, let $p_0,p_1\in P$ be two points and let $\ell_0,\ell_1\in P$ be two lines. We denote by $p_0\lor p_1$ the unique line in $P$ incident to $p_0$ and $p_1$, and we denote by $\ell_0\land \ell_1$ the unique point in $P$ incident to $\ell_0$ and $\ell_1$.
\end{notation}

\begin{lemma}\label{lemma:substitution_points}
	Let $A$ be an infinite open plane with a wellfounded $\mrm{HF}$-order $<$, let $\ell\in A$ be a line and let $B\subsetneq A$ be an infinite open subplane with $B=\mrm{acl}(B^\downarrow)$. Suppose $c$ is the least element under the ordering $<$ in $A\setminus B$, then there are at most two points $d_0,d_1\in A\setminus B$ which are incident only to $\ell\in B$ and such that $\mrm{acl}(Bc)=\mrm{acl}(Bd_0d_1)$.
\end{lemma}
\begin{proof}
	As in the assumptions, we let $A$ be an infinite open plane with a wellfounded $\mrm{HF}$-order $<$, we let $B\subsetneq A$ with $B=B^\downarrow$ and $c\in A\setminus B$. Also, we fix a line $\ell\in B$. Let $c$ be the least element in $ A\setminus B$ according to the order $<$ and notice that clearly $<$ witnesses the fact that $Bc\hleq A$.  Notice that by \ref{lemma:useful_facts_ngons} we have  $B=\mrm{acl}(B^\downarrow)=F(B^\downarrow)$.  We next distinguish the following cases.
	
	\smallskip
	\noindent \underline{Case 1.} $c$ is a point not incident to any element from $B$.
	\newline Let $p_0,p_1$ be two points in $F(B)$ not incident to $\ell$ (these exist since $F(B)$ is infinite) and consider the lines $\ell_0=c\lor p_0$ and $\ell_1=c\lor p_1$. Then take the points $d_0\coloneqq \ell_0\land \ell$ and $d_1\coloneqq \ell_1\land \ell$. Clearly $d_0,d_1\in F(Bc)$. Moreover, it is easy to see that $F(B)d_0d_1\hleq F(B)d_0d_1\ell_0\ell_1c$ and that $c$ is also algebraic over $F(B)d_0d_1$. It follows readily that $F(Bc)=F(Bd_0d_1)$. 
	
	\smallskip
	\noindent \underline{Case 2.} $c$ is a point incident to a line $\ell_1\in B$.
	\newline If $\ell_1=\ell$ we do not have to do anything. Otherwise we proceed as follows. First, let $p$ be any point in $F(B)$ not incident to $\ell$ and $\ell_1$, and consider the line $\ell_2=c\lor p$. Then take the point $d\coloneqq \ell_2\land \ell$. Clearly $d\in F(Bc)$. Moreover, it is easy to see that $F(B)d\hleq F(B)d\ell_2c$ and, since $c$ has two incidences in $F(B)d\ell_2$, we also have that $c$ is algebraic over $F(B)d$. Since $F(B)dc\ell\hleq A$ it follows readily that $F(Bd)=F(Bc)$. 	
	
	\smallskip
	\noindent \underline{Case 3.} $c$ is a line not incident to any element from $B$.
	\newline First, we construct the point $d_0=c\land \ell$. Then pick any line $\ell_0\neq \ell$ from $B$ and construct $p_0=c\land \ell_0$. Then pick $p_1\in B$ not incident to $\ell,\ell_0$, and construct $\ell_1=p_1\lor p_0$. Finally, let $d_1=\ell_1\land \ell$. Reasoning as in the previous case one can see that $F(Bb)=F(Bd_0d_1)$.
	
	\smallskip
	\noindent \underline{Case 4.} $c$ is a line incident to no point from $B$
	\newline This proceeds similarly as the previous case. \qedhere	
\end{proof}

\begin{theorem}\label{theorem:free_categoricity}
	Let $A$ be a projective plane of size $\kappa\geq\aleph_1$ and suppose it is free in the sense of Definition~\ref{def:free_correct}, i.e., it admits a wellfounded $\mrm{HF}$-order $<$, then $A\cong F(A_0X)$, where $A_0=\{a_0,a_1,a_2,a_3,a_4, a_0\lor a_1\}$ is a set of five points together with the line $a_0\lor a_1$, and $X$ is a set of points of size $\kappa$ incident to $a_0\lor a_1$.
\end{theorem}
\begin{proof}
	Let $A$ be a projective plane of size $\kappa\geq \aleph_1$ with a wellfounded $\mrm{HF}$-ordering $<_A$. We can find some ordinal $\delta$ with $\kappa\leq \delta<\kappa^+$ and an enumeration $(a_i)_{i<\delta}$ so that $a_i<_A a_j$ whenever $i<j$. Since every finitely-generated free projective plane is isomorphic to the configuration $\pi^n$, we can assume without loss of generality that  $a_0,\dots, a_4$ are points.  For each $i<\delta$ we let $A_i=\{a_j : j<i\}$. By the definition of $\mrm{HF}$-ordering in projective planes we can assume without loss of generality that, if $i\leq j$, $i\leq  k$,  and $a_j$ is incident to two elements from $A_i$ while $a_k$ is incident to at most one element from $A_i$, then $j<k$ and so $a_j<_Aa_k$. It follows that there is an ordinal $\beta$ and an increasing continuous chain of ordinals $(\gamma_i)_{i<\beta}$ such that:
	\begin{enumerate}[(a)]
		\item $\gamma_i<\delta$ for all $i<\delta$;
		\item $\bigcup_{i<\beta}\gamma_i=\delta$ or $\beta=\alpha+1$ and $\gamma_\alpha=\delta$;
		\item $A_{\gamma_0}=F(a_0,\dots, a_4)$;
		\item $A_{\gamma_{i+1}}=F(A_{\gamma_i} a_{\gamma_{i+1}})$ for every every $i<\beta$.
	\end{enumerate}
	Then, it follows from Lemma~\ref{lemma:substitution_points} that for all $i<\delta$ we can either find one point $a^0_{\gamma_{i+1}}\in A\setminus A_{\gamma_i}$ incident only to $a_0\lor a_1$ in $F(A_{\alpha_i})$ and such that $F(A_{\gamma_{i+1}})=(A_{\gamma_{i}}a^0_{\gamma_{i+1}})$, or we can find two points  $a^0_{\gamma_{i+1}}\in A\setminus A_{\gamma_i}$ and $a^1_{\gamma_{i+1}}\in A\setminus A_{\gamma_i}$ which are incident only to $a_0\lor a_1$ in $F(A_{\gamma_{i}})$ and $F(A_{\gamma_{i+1}})=(A_{\gamma_{i}}a^0_{\gamma_{i+1}}a^1_{\gamma_{i+1}})$. Inductively, it follows for every $\omega\leq \gamma<\delta$ that $A_{\gamma}\cong F(A_0X_\lambda)$ where $\lambda=|\gamma|$ and $X_\lambda$ is any set of $\lambda$-many points.  Since by point (b) above we have that either $A=F(\bigcup_{i<\beta}A_{\gamma_i})$ or  $A=F(A_{\gamma_\alpha} a_{\gamma_{\alpha+1}})$, it follows in both cases  that $A$ is exactly the open projective plane obtained by first adding the points $a_0,\dots, a_4$, and then $\kappa$-many points incidents exactly to  $a_0\lor a_1$. This completes our proof.
\end{proof}

\begin{corollary}\label{incidence:corollary}
	The class of projective planes with a wellfounded $\mrm{HF}$-construction is uncountably categorical. In particular, for every set of points $X_\kappa$ of size $\kappa\geq \aleph_1$, we have that $F(X_\kappa)$ is the only projective plane of size $\kappa$, up to isomorphism.
\end{corollary}

\end{document}